\documentclass[11pt]{article}
\usepackage{amssymb,latexsym,amsmath,amsbsy,amsthm,amsxtra,amsgen,graphicx,makeidx,dsfont,xcolor,hyperref,enumitem,longtable}
\newfont{\msbm}{msbm10 at 11pt}
\newcommand {\R} {\mbox{\msbm R}}

\newcommand {\Q} {\mbox{\msbm Q}}

\newcommand {\N} {\mbox{\msbm N}}
\newcommand {\1} {\mathds{1}}
\newcommand {\ind} {\mathds{1}}
\newcommand {\F} {\mathcal{F}}

\newfont{\msbmsm}{msbm10 at 8pt}
\newcommand {\Rsm} {\mbox{\msbmsm R}}

\newcommand {\E} {\mbox{\msbm E}}
\renewcommand {\P} {\mbox{\msbm P}}

\newtheorem{Theo}{Theorem}[section]
\newtheorem{Lemma}[Theo]{Lemma}
\newtheorem{Cor}[Theo]{Corollary}
\newtheorem{Prop}[Theo]{Proposition}

\def\eps{\varepsilon}
\def\Var{\textup{Var}}

\numberwithin{equation}{section}

\begin{document}
\title{A Gaussian particle distribution for branching Brownian motion with an inhomogeneous branching rate}
\author{Matthew I. Roberts\thanks{Supported by a Royal Society University Research Fellowship} \, and Jason Schweinsberg\thanks{Supported in part by NSF Grant DMS-1707953}}
\maketitle

\vspace{-.2in}

\begin{abstract}
Motivated by the goal of understanding the evolution of populations undergoing selection, we consider branching Brownian motion in which particles independently move according to one-dimensional Brownian motion with drift, each particle may either split into two or die, and the difference between the birth and death rates is a linear function of the position of the particle.  We show that, under certain assumptions, after a sufficiently long time, the empirical distribution of the positions of the particles is approximately Gaussian.  This provides mathematically rigorous justification for results in the biology literature indicating that the distribution of the fitness levels of individuals in a population over time evolves like a Gaussian traveling wave. 
\end{abstract}



\section{Introduction}\label{intro_sec}

An important problem in evolutionary biology is to understand how the fitness of individuals in a population increases over time as a result of beneficial mutations.  Results in the biology and physics literature indicate that in large populations, if individuals acquire beneficial mutations at a constant rate, then the overall fitness level of the population increases at a constant rate, known as the rate of adaptation, while the empirical distribution of the fitness levels of individuals in the population becomes approximately Gaussian.  That is, the empirical distribution of the fitness levels of individuals in the population evolves over time like a Gaussian traveling wave.  The idea of modeling the fitness distribution by a traveling wave goes back at least to the work of Tsimring, Levine, and Kessler \cite{tlk96}.  Later works discussing the Gaussian shape for the traveling wave include \cite{beer07, brw08, df07, f13, nh13, psk10, rbw08, rc05}.

Although the idea that the fitness distribution evolves as a Gaussian traveling wave is well established in the biology and physics literature, the mathematically rigorous work on this problem has been considerably more limited. The main aim of this paper is to provide a first rigorous analysis in which a non-degenerate Gaussian traveling wave is observed in this context. Before we go into the details of our own results, we give a brief overview of the existing mathematical literature.

A standard mathematical model involves a population of fixed size $N$ in which each individual independently acquires beneficial mutations at the constant rate $\mu$.  Beneficial mutations increase an individual's fitness by $s$, so that an individual that has acquired $k$ beneficial mutations, which we call a type $k$ individual, has fitness $\max\{0, 1 + s(k - m(t))\}$, where $m(t)$ is the mean number of mutations of the individuals in the population at time $t$.  Each individual independently dies at rate one, and when an individual dies, the parent of the new individual that is born is chosen at random from the population with probability proportional to fitness.  A number of authors have studied models very similar to this one.  Yu, Etheridge, and Cuthbertson \cite{yec10} and Kelly \cite{mk1} obtained rigorous results concerning the rate of adaptation for a very similar model, but did not establish a Gaussian shape for the fitness distribution.  Durrett and Mayberry \cite{dm11} considered, for a closely related model, the case in which $s$ is constant and the mutation rate is $N^{-\alpha}$, where $0 < \alpha < 1$.  They rigorously established traveling wave behavior.  However, they considered mutation rates that are small enough that the number of distinct types present in the population at a typical time is a constant that does not tend to infinity with $N$, which means the traveling wave does not have a Gaussian shape.  Schweinsberg \cite{schI} considered slightly faster mutation rates, so that the mutation rate tends to zero more slowly than any power of $N$.  This work essentially made rigorous the heuristics developed by Desai and Fisher \cite{df07}.  For the range of parameter values considered in \cite{schI}, the traveling wave exhibits Gaussian-like tail behavior, in the sense that the logarithm of the ratio of the number of individuals with $\ell$ more mutations than average to the number of individuals with an average number of mutations is proportional to $-\ell^2$.  However, at a typical time, most individuals have the same number of mutations, which means that the empirical distribution of the fitnesses of individuals in the population is actually converging to a point mass, rather than to a Gaussian distribution.  Up to this point, as far as we know, the empirical distribution of the fitnesses of individuals in the population in this model has not been rigorously shown to converge to a Gaussian distribution for any range of values of the parameters $\mu$ and $s$.

For the fitness distribution to be approximately Gaussian, the mutation rate needs to be large enough that one type does not dominate the population at a typical time.  This corresponds to the high speed regime considered in \cite{f13}; see also \cite{gd13}.  We therefore consider a scenario in which the rate of beneficial mutations is large, but the additional selective benefit resulting from each mutation is small.  This is the idea behind the so-called infinitesimal model in quantitative genetics, which goes back to the early work of Fisher \cite{f18}.  See Barton, Etheridge, and V\'eber~\cite{bev17} for a recent mathematical treatment of the infinitesimal model and an extensive survey of the relevant biology literature.  They establish conditions under which the values of various quantitative traits within a family are approximately normally distributed, but emphasize that their results do not imply that the trait values across the entire population are approximately normally distributed.

If individuals acquire many mutations, each having a small effect on fitness, then the fitness of an individual over time will evolve like a continuous-time random walk, which, after being scaled to have mean zero, can be approximated by Brownian motion.  We will therefore consider a model in which the fitness of an individual moves over time according to Brownian motion.  We will allow the offspring of individuals to evolve independently, rather than imposing a fixed population size.  However, by adding a negative drift to the Brownian motion and choosing the initial conditions carefully, we can work in a nearly critical regime in which the number of individuals stays the same order of magnitude on the time scale of interest, rather than dying out quickly or growing exponentially.  Consequently, we believe that our results will be relevant for understanding populations with a fixed size, as we discuss briefly in Section~\ref{discrete_sec}.

\subsection{The Model}

The above considerations lead us to consider the following branching Brownian motion process, which is the model that we will study throughout the rest of the paper.  Because we aim to prove a limit theorem, we will consider a sequence of processes indexed by $n$.  We begin with some configuration of particles at time zero, which may depend on $n$.  Each particle independently moves according to one-dimensional Brownian motion with drift $-\rho_n$, where $\rho_n > 0$.  Also, any particle at the location $x$ independently dies at rate $d_n(x)$, and splits into two particles at rate $b_n(x)$, where
\begin{equation}\label{bndn}
b_n(x)-d_n(x) = \beta_n x
\end{equation}
for some $\beta_n > 0$. In particular, note that the birth and death rates are the same for particles at the origin. This model is very similar to the model studied by Neher and Hallatschek \cite{nh13}.  The main results of this paper are mathematically rigorous versions of some of the results in \cite{nh13}, and some of the results in the high speed regime in \cite{f13}.

As indicated above, we view this process as modeling a population undergoing selection.  With this interpretation, particles represent individuals in a population, and the position of the particle corresponds to the fitness level of the individual. 

\subsection{Main Results}

Before stating our main results, we will need to introduce some assumptions and some notation.  Given two sequences of positive real numbers $(a_n)_{n=1}^{\infty}$ and $(b_n)_{n=1}^{\infty}$, we will write $a_n \lesssim b_n$ if $a_n/b_n$ is bounded above by a positive constant, and $a_n \ll b_n$ if $\lim_{n \rightarrow \infty} a_n/b_n = 0$.  We will use the symbols $\gtrsim$ and $\gg$ likewise.  We also write $a_n \asymp b_n$ if $a_n/b_n$ is bounded both above and below by positive constants.  We write that $a_n$ is $O(1)$ if the sequence $(a_n)_{n=1}^{\infty}$ is bounded and $o(1)$ if $\lim_{n \rightarrow \infty} a_n = 0$.  We will also use $O(1)$ and $o(1)$ for random sequences that are uniformly bounded above by deterministic sequences that are $O(1)$ and $o(1)$ respectively.

We will make the crucial assumption that
\begin{equation}\label{A1}
\lim_{n \rightarrow \infty} \frac{\rho_n^3}{\beta_n} = \infty.
\end{equation}
We will also assume that 
\begin{equation}\label{A2}
\lim_{n \rightarrow \infty} \rho_n = 0,
\end{equation}
and that, in addition to (\ref{bndn}), there exists $\Delta\in(0,1)$ such that
\begin{equation}\label{A3}
d_n(x)\ge \Delta \hspace{4mm} \text{for all } x\in\R, n\in\N \hspace{4mm} \text{and} \hspace{4mm} b_n(x)\le 1/\Delta \hspace{4mm} \text{for all } x\le 1/\beta_n, n\in\N.
\end{equation}
The assumptions (\ref{A1}), (\ref{A2}), and (\ref{A3}) will be in effect throughout the rest of the paper, even when they are not explicitly mentioned.
Note that (\ref{bndn}) and (\ref{A3}) are satisfied, for example, if $d_n(x) = 1$ and $b_n(x) = 1 + \beta_n x$ for all $x \geq -1/\beta_n$, while $d_n(x) = -\beta_n x$ and $b_n(x) = 0$ for $x < -1/\beta_n$.

Beyond Section \ref{intro_sec}, to lighten notation, we will drop the subscripts and write $\rho$ and $\beta$ in place of $\rho_n$ and $\beta_n$.  However, it is important for the reader to keep in mind that these parameters do depend on $n$.

We will also need to consider the Airy function $$Ai(x) = \frac{1}{\pi} \int_0^{\infty} \cos \bigg( \frac{y^3}{3} + xy \bigg) \: dy.$$  The Airy function $Ai$ has an infinite sequence of zeros $(\gamma_k)_{k=1}^{\infty}$ which satisfy $ \dots < \gamma_2 < \gamma_1 < 0$.  The Airy function and particularly the quantity $\gamma_1$ will play an important role in what follows.  It is known (see table 9.9.1 in \cite{dlmf}) that to three decimal places,
\begin{equation}\label{zerovalue}
\gamma_1 \approx -2.338.
\end{equation}

We will let $N_n(t)$ denote the total number of particles in the system at time $t$, and we will let $X_{1,n}(t) \geq X_{2,n}(t) \geq \dots \geq X_{N_n(t),n}(t)$ denote the locations of the particles at time $t$. We imagine our system drawn with time on the vertical axis, so that the maximal particle is the \emph{right-most}. Let
\begin{equation}\label{Lndef}
L_n = \frac{\rho_n^2}{2 \beta_n} - (2 \beta_n)^{-1/3} \gamma_1.
\end{equation}
Note that $\rho_n^2/\beta_n \ll \beta_n^{-1/3}$ by (\ref{A1}), which means that $\rho_n^2/2 \beta_n$ is the dominant term in (\ref{Lndef}).
Our particles will generally be to the left of $L_n$, and we call the area near $L_n$ the \emph{right edge}.
Note that since $\gamma_1<0$, $L_n$ is to the right of $\rho_n^2/2\beta_n$. Define
\begin{equation}\label{Yntdef}
Y_n(t) = \sum_{i=1}^{N_n(t)} e^{\rho_n X_{i,n}(t)}
\end{equation}
and
\begin{equation}\label{Zntdef}
Z_n(t) = \sum_{i=1}^{N_n(t)} e^{\rho_n X_{i,n}(t)} Ai((2 \beta_n)^{1/3}(L_n - X_{i,n}(t)) + \gamma_1) \1_{\{X_{i,n}(t) < L_n\}}.
\end{equation}
While the form of $Z_n(t)$ may seem mysterious at this point, this turns out to be a natural measure of the ``size" of the process at time $t$.  It turns out that, if we modify the process by killing particles that reach $L_n$, then $(Z_n(t), t \geq 0)$ is a martingale.

In addition to the assumptions (\ref{A1}), (\ref{A2}), and (\ref{A3}) on the parameters, we will make two assumptions on the initial configuration of particles at time zero.  We will assume that for all $\eps > 0$, there is a $\delta > 0$ such that for sufficiently large $n$,
\begin{equation}\label{Zasm}
\P\bigg( \delta \cdot \frac{\beta_n^{1/3}}{\rho_n^3} e^{\rho_n L_n} \leq Z_n(0) \leq \frac{1}{\delta} \cdot \frac{\beta_n^{1/3}}{\rho_n^3} e^{\rho_n L_n} \bigg) > 1 - \eps.
\end{equation}
In other words, the sequence of random variables $(\rho_n^3 \beta_n^{-1/3} e^{-\rho_n L_n} Z_n(0))_{n=1}^{\infty}$ and its reciprocal sequence are tight. Since $\rho_n L_n \approx \rho_n^3/2\beta_n\to\infty$ by (\ref{A1}), we have $(\beta_n^{1/3}/\rho_n) e^{\rho_n L_n} \rightarrow \infty$.  It then follows that when (\ref{Zasm}) holds, $Z_n(0)$ is large because $1/\rho_n^2 \rightarrow \infty$ by (\ref{A2}).

We will also assume that 
\begin{equation}\label{Yasm}
\rho_n^2 e^{-\rho_n L_n} Y_n(0) \rightarrow_p 0,
\end{equation}
where $\rightarrow_p$ denotes convergence in probability as $n \rightarrow \infty$.  Roughly speaking, the condition (\ref{Yasm}) ensures that the contribution to $Z_n(t)$ for small values of $t$ will not be dominated by the descendants of a single particle at time zero, nor will it be dominated by particles that are far from $L_n$ at time zero. The two conditions \eqref{Zasm} and \eqref{Yasm} cannot be satisfied by a single particle at any location.  However, they are satisfied, for example, by letting $(u_n)_{n=1}^{\infty}$ be any sequence satisfying $\rho_n^{-1}\ll u_n\le \beta_n^{-1/3}$, and starting with $e^{\rho_n u_n}/u_n\rho_n^3$ particles all located at $L_n-u_n$.

We are now ready to state our main result, establishing the Gaussian shape for the distribution of particles on a suitable time scale.  Here and throughout the paper, $\delta_y$ denotes the unit mass at $y$, and $\Rightarrow$ denotes convergence in distribution for random elements of the Polish space of probability measures on $\R$ endowed with the weak topology.

\begin{Theo}\label{GaussThm}
Suppose the assumptions (\ref{A1}), (\ref{A2}), (\ref{A3}), (\ref{Zasm}) and (\ref{Yasm}) hold.  Suppose
\begin{equation}\label{rhobetagauss}
\frac{\rho_n^{2/3}}{\beta_n^{8/9}} \ll t_n - \frac{\rho_n}{\beta_n} \lesssim \frac{\rho_n}{\beta_n}.
\end{equation}
For $t > 0$, define the random probability measure
\begin{equation}\label{zetadef}
\zeta_n(t) = \frac{1}{N_n(t)} \sum_{i=1}^{N_n(t)} \delta_{X_{i,n}(t) \sqrt{\beta_n/\rho_n}}
\end{equation}
if $N_n(t) \geq 1$, and set $\zeta_n(t) = \delta_0$ if $N_n(t) = 0$.
Let $\mu$ be the standard normal distribution.  Then $\zeta_n(t_n) \Rightarrow \mu$ as $n \rightarrow \infty$.
\end{Theo}

This result shows that the empirical distribution of the particle locations shortly after time $\rho_n/\beta_n$ is approximately Gaussian, and stays that way at least for times of the order $\rho_n/\beta_n$.  We will see later that it takes time approximately $\rho_n/\beta_n$ for descendants of particles near $L_n$ to start to reach the origin, which is the reason why the Gaussian shape does not show up until after a waiting time of approximately $\rho_n/\beta_n$.

Note that this Gaussian distribution has mean zero.  However, if we considered a translation of the model in which particles move according to branching Brownian motion with no drift, with time-dependent birth and death rates satisfying $b(x,t)-d(x,t)=\beta_n(x - \rho_n t)$, then the Gaussian particle distribution at time $t_n$ would be centered at $\rho_n t_n$, giving rise to the Gaussian traveling wave behavior discussed above.  Therefore, the drift parameter $\rho_n$ can be interpreted in this model as the speed at which the traveling wave advances.  Because the fitness of the population, as measured by the branching rate, increases by $\beta_n$ whenever the traveling wave advances by one unit, the fitness of the population increases over time at rate $v_n = \beta_n \rho_n$.

Note also that from the scaling in (\ref{zetadef}), the empirical distribution of particles has standard deviation approximately $\sqrt{\rho_n/\beta_n}$, and therefore has variance $\rho_n/\beta_n$.  Because the fitness of a particle in increases by $\beta_n$ when the position of a particle increases by one, it follows that the variance of the fitness distribution is $\sigma_n^2 = \beta_n^2 (\rho_n/\beta_n) = v_n$, in agreement with Fisher's Fundamental Theorem of Natural Selection~\cite{f30}.

Because, as we will see, some particles with unusually high fitness account for nearly all of the offspring that are still alive at a time $\rho_n/\beta_n$ in the future, it is also of interest to understand the empirical distribution of particles close to the right edge.  To do this, we consider an empirical measure in which a particle at $x$ is weighted by $e^{\rho_n x}$.  This leads to the following result.

\begin{Theo}\label{EdgeThm}
Suppose the assumptions (\ref{A1}), (\ref{A2}), (\ref{A3}), (\ref{Zasm}) and (\ref{Yasm}) hold.  Suppose
\begin{equation}\label{tcond}
\beta_n^{-2/3} \bigg( \log \bigg(\frac{\rho_n}{\beta_n^{1/3}} \bigg) \bigg)^{1/3} \ll t_n \lesssim \frac{\rho_n}{\beta_n}.
\end{equation}
For $t > 0$, define the random probability measure
\begin{equation}\label{xidef}
\xi_n(t) = \frac{1}{Y_n(t)} \sum_{i=1}^{N_n(t)} e^{\rho_n X_{i,n}(t)} \delta_{(2 \beta_n)^{1/3}(L_n - X_{i,n}(t))}
\end{equation}
if $N_n(t) \geq 1$, and set $\xi_n(t) = \delta_0$ if $N_n(t) = 0$.
Let $\nu$ be the probability measure on $(0,\infty)$ with probability density function
\begin{equation}\label{hdef}
h(y) = \frac{Ai(y + \gamma_1)}{\int_0^{\infty} Ai(z + \gamma_1) \: dz}.
\end{equation}
Then $\xi_n(t_n) \Rightarrow \nu$.
\end{Theo}

Note that $\beta_n^{-2/3} \ll \rho_n/\beta_n$ by (\ref{A1}), so the particles near the right edge reach this limiting configuration on a much faster time scale than the time it takes for the Gaussian shape to emerge near the origin.  Note also that Theorem \ref{EdgeThm} describes particles that are within $O(\beta_n^{-1/3})$ of $L_n$, which in turn is within $O(\beta_n^{-1/3})$ of $\rho_n^2/2 \beta_n$.  Because $\rho_n^2/2 \beta_n \gg \sqrt{\rho_n/\beta_n}$ by (\ref{A1}), it follows that these particles near the right edge whose configuration is described by Theorem \ref{EdgeThm} are far away from the bulk of the Gaussian distribution of particles described by Theorem \ref{GaussThm}.

The appearance of the Airy function here is not a surprise.  The Airy function appeared in the early work on traveling waves by Tsimring, Levine, and Kessler \cite{tlk96}.  It also arises in the work of Neher and Hallatschek \cite{nh13}, who studied essentially the same model that we are considering in this paper, as well as in \cite{ckl05}, where one of the equations that is central to the work of Neher and Hallatschek \cite{nh13} had previously been studied.  Fisher \cite{f13} arrived at an expression analogous to (\ref{Lndef}) involving the largest zero of the Airy function for the difference in fitness between the fittest individual and an individual of average fitness (see equation (50) of \cite{f13}).  The Airy function also arises in the expression for the position of the right-most particle in branching Brownian motion with inhomogeneous variance, as shown in \cite{mz16}.

\subsection{A heuristic analysis based on large deviations}\label{ldsec}

Branching Brownian motion with an inhomogeneous branching rate was also studied in \cite{bbhhr, jhsh09}, where the authors considered the case in which a particle at $x$ branches at a rate proportional to $|x|^p$, where $0 < p < 2$.  The techniques of proof that we will use in the present paper are quite different from those used in \cite{bbhhr, jhsh09}.  However, while the large deviations techniques used in \cite{bbhhr} are not sufficiently precise to prove the main results of this paper, a heuristic calculation based on these techniques provides insight into the behavior of the process and helps to explain the motivation for our proof strategy.  We therefore summarize this calculation here, even though it is not logically necessary for understanding the rest of the paper.

Let $T$ be a large time, and let $f: [0, T] \rightarrow \R$.  According to results in \cite{bbhhr}, if the process starts with one particle at $f(0)$, then the expected number of particles that stay close to the function $f$ through time $T$ can be approximated by
\begin{equation}\label{mainexp}
\exp \bigg( \int_0^T \Big(\beta_n f(u) - \frac{1}{2} (f'(u) + \rho_n)^2\Big) \: du \bigg).
\end{equation}
This is because the birth rate minus the death rate for particles that do manage to follow near $f$ will be approximately $\beta_n f(u)$ at time $u$, and the probability that a Brownian motion with drift $-\rho_n$ manages to follow near $f$ is roughly $\exp(-\frac{1}{2}\int_0^T (f'(u)+\rho_n)^2\:du)$ by Schilder's theorem.

If the process starts with one particle, then it is possible that the process will die out almost immediately.  However, assuming this does not happen, the actual number of such particles will be reasonably omparable to the expected number provided that the integral in (\ref{mainexp}), when evaluated from $0$ to $t$, is nonnegative for all $t \in [0, T]$. Otherwise, at the point that the integral becomes negative, the expected number of particles will be exponentially small and Markov's inequality entails that with high probability no particles will manage to follow the trajectory.

Suppose the function $f$ is constant on $[0, T]$.  Then the integrand in (\ref{mainexp}) is positive if $f(u) > \rho_n^2/2 \beta_n$, suggesting that if we begin with one particle above $\rho_n^2/2 \beta_n$, the number of descendants of this particle will grow exponentially.  The integrand in (\ref{mainexp}) is negative if $f(u) > \rho_n^2/2 \beta_n$, suggesting that if we begin with one particle below $\rho^2/2 \beta$, then no particle can remain close to this level.  An equilibrium is reached when $f(u) = \rho_n^2/2\beta_n$ for all $u \in [0, T]$, in which case the integrand in (\ref{mainexp}) is zero.  Therefore, if we start with one particle at the position $\rho_n^2/2 \beta_n$ at time zero (and the descendants of this particle do not die out quickly), then the right-most particle is likely to stay near the position $\rho_n^2/2 \beta_n$.  This calculation explains why we defined $L_n$ to be close to $\rho_n^2/2 \beta_n$.

We now consider the trajectory $f_z$ followed by particles that are near $z$ at time $T$, starting from one particle at the position $\rho_n^2/2 \beta_n$ at time zero.  According to Theorem 7 of \cite{bbhhr}, these particles initially stay close to their initial position, so we have $f_z(t) = \rho_n^2/2 \beta_n$ for $t \in [0, t_z]$ for some $t_z$.  Then for $u \in [t_z, T]$, equation (18) of \cite{bbhhr} implies that $f_z''(u) = - \beta_n$, which means $f_z(u) = a + bu - (\beta_n/2) u^2$ for $u \in [t_z, T]$, with the conditions $f_z(T) = z$, $f_z(t_z) = \rho_n^2/2\beta_n$, and $f_z'(t_z) = 0$, for some real numbers $a$ and $b$.  Solving, we get
\begin{equation}
f_z(u) = \frac{\rho_n^2}{2 \beta_n} - \frac{\beta_n}{2} (u - t_z)^2, \hspace{.5in}
t_z = T - \sqrt{\frac{2}{\beta_n} \bigg(\frac{\rho_n^2}{2 \beta_n} - z \bigg)} \label{tz}.
\end{equation}
This means that the number of particles near $z$ at time $T$ is approximately $\exp(g(z))$, where
\begin{equation}\label{gdef}
g(z) = \int_{t_z}^T \Big( \beta_n f_z(u) - \frac{1}{2}(f_z'(u) + \rho_n)^2 \Big) \: du = \frac{\rho_n^3}{2 \beta_n} - \rho_n z - \frac{2 \sqrt{2 \beta_n}}{3} \bigg( \frac{\rho_n^2}{2 \beta_n} - z \bigg)^{3/2}.
\end{equation}
We can then calculate $$g(0) = \frac{\rho_n^3}{6 \beta_n}, \hspace{.2in} g'(0) = 0, \hspace{.2in}g''(0) = - \frac{\beta_n}{\rho_n},$$
which means that for small $z$, a Taylor expansion gives
\begin{equation}\label{gapprox}
g(z) \approx \frac{\rho_n^3}{6 \beta_n} - \frac{\beta_n z^2}{2 \rho_n}.
\end{equation}
We then see that at time $T$, the empirical distribution of particles should be approximately normal with mean zero and variance $\rho_n/\beta_n$, consistent with Theorem \ref{GaussThm}.  For this reasoning to be valid, the standard deviation $\sqrt{\rho_n/\beta_n}$ needs to be much smaller than the distance between zero and the right-most particle, which is $\rho_n^2/2\beta_n$.  This is indeed the case when (\ref{A1}) holds.
Note also that $T - t_0 = \rho_n/\beta_n$, which means that it takes time approximately $\rho_n/\beta_n$ for the descendants of particles near $\rho_n^2/2 \beta_n$ to concentrate near the origin.  This explains why, in Theorem \ref{GaussThm}, the Gaussian shape emerges only after the process has evolved for time $\rho_n/\beta_n$.

Finally, note that although the bulk of the distribution of particles is Gaussian, we see different behavior near the right edge.  In particular, writing $y = \rho_n^2/2 \beta_n - z$, from \eqref{gdef} we have $$g(z) = \rho_n y - \frac{2 \sqrt{2 \beta_n}}{3} y^{3/2},$$ and since the first term dominates for small $y$, we see an exponential decay of rate $\rho_n$ for the particle density near the right edge.  This is partially explained by Theorem \ref{EdgeThm}, which describes the configuration of particles that are within a distance of order $\beta_n^{-1/3}$ from the right edge at $L_n$.

\subsection{Additional connections to previous work}

Branching Brownian motion has previously been used to model populations undergoing selection in \cite{bbs, bdmm1, bdmm2, dfpsl19, m16}.  However, in these works, it was assumed that the branching rate of a particle does not depend on the position of the particle.  Instead, to model selection, particles are killed when they drift too far to the left, that is, when their fitness gets too low.  This model leads to substantially different behavior.  The empirical distribution of particles is not approximately Gaussian but rather most particles end up close to the left edge.  For example, for the model studied in \cite{bbs} in which particles are killed when they reach the origin, when the system has $N$ particles that are in a nearly stable configuration, the density of particles near $y$ will be roughly proportional to $e^{-\sqrt{2} y} \sin(\sqrt{2} \pi y/\log N)$.  See \cite{bbs3} for a precise formulation and proof of this result.  Nevertheless, the techniques of proof used in the present paper are very similar to those used in \cite{bbs, bbs3}.

The main results of \cite{bbs} state that when the branching rate is constant but particles are killed upon reaching the origin, under suitable initial conditions, the number of particles over time behaves like Neveu's continuous-state branching process, while the genealogy of the particles can be described by the Bolthausen-Sznitman coalescent.  The result concerning the genealogy of the particles had previously been predicted in \cite{bdmm1, bdmm2}.  We believe that similar results should hold for the process studied in the present paper.  Indeed, \eqref{zmeaneq} and Lemma \ref{2momprop} below closely resemble Lemmas 11 and 12 of \cite{bbs} respectively, which are two of the key steps in proving the convergence to Neveu's continuous-state branching process.  Also, it was established in \cite{nh13} by nonrigorous methods that the genealogy of the particles for the process studied in the present paper should be described by the Bolthausen-Sznitman coalescent, following an initial time period in which coalescence does not occur because particles sampled from the bulk of the distribution at time $t$ will most likely be descended from distinct ancestors near the right edge of the distribution at time $t - \rho/\beta$.  However, in the present paper, we focus on establishing the Gaussian shape for the empirical distribution of particles, and we defer consideration of the genealogy of the particles to a future work.

Beckman \cite{beck19} also considered branching Brownian motion in which the branching rate depends on the position of the particle.  She assumed that at time zero, there are $N$ particles placed independently according to some density, and established a hydrodynamic limit for the evolution of the empirical distribution of particles over time as $N \rightarrow \infty$.  An important difference between the work in \cite{beck19} and the present paper is that in \cite{beck19}, time is not rescaled, so the results essentially pertain to what happens at a fixed time $t$ in the limit as $N \rightarrow \infty$, whereas here we consider times of the order $\rho/\beta$, which tends to infinity.  Also, while the work in \cite{beck19} was likewise motivated by the consideration of evolving populations, the results in \cite{beck19} have been established only for the case in which the branching rate of the particles is a bounded function of the position.  Groisman, Jonckheere, and Martinez \cite{gjm19} likewise obtained a hydrodynamic limit for a model in which, at the time of a branching event, two particles are chosen at random and the particle on the left jumps to the location of the particle on the right, effectively giving a higher branching rate to particles further to the right.

\subsection{Implications for discrete population models}\label{discrete_sec}

We briefly return here to the discrete population model mentioned at the beginning of the paper, in which there is a fixed population of size $N$ and, at rate $\mu_N$, each individual acquires beneficial mutations that increase the individual's fitness by $s_N$.  As noted above, we believe that our results may provide non-rigorous insight into the behavior of this discrete model.  In this section, we explain the correspondence that we expect should hold between the discrete population model and the branching Brownian motion model. We emphasize, however, that the results in this paper do not rigorously establish any results for the discrete population model.  Extending the results in this way is expected to require overcoming considerable technical obstacles, and we consider this to be an important topic for future work. 

In the discrete model, the evolution of an individual's fitness over time is a continuous-time random walk, which evolves by making steps of size $s_N$ at rate $\mu_N$. If $\mu_N$ is sufficiently large and the selective advantage $s_N$ resulting from an individual mutation is sufficiently small, then after suitable scaling, it should therefore be possible to approximate the evolution of an individual's fitness over time by Brownian motion with drift. We therefore expect that, if the parameters are tuned in the right way, then our branching Brownian motion model should give a good approximation to the discrete model. Below we show how to choose the parameters to match with the discrete model, and we then expect that similar results to those that we establish for the branching Brownian motion model should also be true for the discrete model. In particular, we expect the fitness distribution of the population to evolve like a Gaussian traveling wave. 

Furthermore, because the convergence of random walks to Brownian motion is not sensitive to the step size distribution of the random walk (assuming finiteness of the second moment), we believe that this correspondence should extend to discrete models in which the fitness change resulting from a mutation is random, and mutations could be deleterious as well as beneficial.  See \cite{f13} and \cite{grbhd12} for detailed nonrigorous work on the evolution of populations in which mutations have a random effect on fitness.

Recall that although the number of particles varies over time in the branching Brownian motion model, the parameters and initial conditions can be chosen so that the size of the population stays the same order of magnitude on the time scale of interest.  We will therefore compare the discrete population model with $N$ individuals to branching Brownian motion with the parameters indexed by $N$ and chosen so that the number of particles is within a constant multiple of $N$.  To see how the parameters match up, note that in the discrete population model, the standard deviation of the number of mutations that an individual gets in one time unit is $\sqrt{\mu_N}$, so the standard deviation of the fitness change of an individual in one time unit is $s_N \sqrt{\mu_N}$.  Because the standard deviation of the position of a Brownian motion with drift after one time unit equals one, the standard deviation of the fitness change of a particle in one time unit in the branching Brownian motion model is $\beta_N$.  This leads to the correspondence
\begin{equation}\label{betadef}
\beta_N = s_N \sqrt{\mu_N}.
\end{equation}
Furthermore, (\ref{mainNn}) below states that for the branching Brownian motion model, if (\ref{Zasm}) and (\ref{Yasm}) hold, the number of particles in the system at later times is of the order $\beta_N^{1/3} \rho_N^{-3} e^{\rho_N L_N} e^{-\rho_N^3/3 \beta_N}$, which suggests the correspondence
\begin{equation}\label{rhodef}
N = \frac{\beta_N^{1/3}}{\rho_N^3} \exp \bigg( \frac{\rho_N^3}{6 \beta_N} - \rho_N (2 \beta_N)^{-1/3} \gamma_1\bigg).
\end{equation}
That is, given the biological parameters $N$, $\mu_N$, and $s_N$, to model the population using branching Brownian motion, we can choose $\beta_N$ to satisfy (\ref{betadef}) and then choose the drift $\rho_N$ so that (\ref{rhodef}) holds.

If (\ref{rhodef}) holds, then the assumption (\ref{A1}) is equivalent to the condition that $N \gg \beta_N^{-2/3}$, which if (\ref{betadef}) holds is equivalent to the condition that
\begin{equation}\label{biocond}
N^3 \mu_N s_N^2 \rightarrow \infty.
\end{equation}
Note that (\ref{biocond}) fails to hold when $\mu_N$ and $s_N$ are both of the order $1/N$, in which case the population can be studied using a classical diffusion approximation.  The quantity $N^3 \mu_N s_N^2$ also figures prominently in the work of Good, Walczak, Neher, and Desai \cite{gwnd}.  They considered the case in which $N \mu_N \rightarrow \infty$, $N s_N \rightarrow 0$, and $N^3 \mu_N s_N^2 \rightarrow c \in (0, \infty)$, which led to what they called the ``fine-grained coalescent".  Their parameter regime would correspond to our model with $\rho_N^3/\beta_N \rightarrow c \in (0, \infty)$, which entails weaker selection than what we consider in this paper.

To prove results rigorously for the discrete population model, it would be necessary to extend the results in this paper for branching Brownian motion to the case of branching random walks, in which the mutations would correspond to random walk steps.  Furthermore, it would be necessary to adapt the analysis to the case of a model of fixed population size, which presents technical challenges because individuals in the population no longer evolve independently.  As noted in \cite{f13}, this will lead to some complicated feedback between the behavior of the fittest individuals and the behavior of the bulk of the population.  Similar technical obstacles were overcome successfully in \cite{m16} by Maillard, who showed that many of the ideas developed in \cite{bbs} for branching Brownian motion with particles killed at the origin could be carried over to the so-called $N$-BBM process, in which the left-most particle is killed at the time of a branching event to keep the number of particles fixed.  We do not pursue these matters further in this paper.

\subsection{Index of notation}

Below is an index of some of the notation that is used repeatedly throughout the paper.

\begin{longtable}{ll}
$A$ & Constant used to adjust where particles are killed.  See (\ref{LAdef}). \\
$b_n(x)$ & Rate at which a particle at $x$ splits into two particles. \\
$d_n(x)$ & Death rate for a particle at $x$. \\
$h$ & Density of limiting particle configuration near right edge.  See (\ref{hdef}). \\
$H_A(t)$ & Defined to equal $L_A - \beta t^2/9$.  See (\ref{HAdef}). \\
$K_A(t)$ & Defined to equal $L_A - \beta t^2/66$.  See (\ref{Kdef}). \\
$\ell(t)$ & Defined to equal $\beta t^2/33$.  See (\ref{Kdef}). \\
$L = L_N$ & Position at which particles are often killed, defined by $L = \rho^2/2 \beta - (2 \beta)^{-1/3} \gamma_1$. \\
$L_A$ & Defined by $L_A = L - A/\rho$.  See (\ref{LAdef}). \\
$N_n(t)$ & Number of particles at time $t$. \\
$p_t(x,y)$ & Density for process started with one particle at $x$ at time zero. See (\ref{ptxy}). \\
$p_t^{\ell}(x,y)$ & Density for the process if particles are killed at $\ell$. \\
$q_t(x,y)$ & Density for the process if the drift is removed. \\
$r_x^{\ell}(u,v)$ & Expected number of particles killed at $\ell$ between times $u$ and $v$. \\
${\tilde r}_x^{\ell}(t)$ & Rate at which particles hit $\ell$ at time $t$. \\
$t_{x,A}$ & Time defined in (\ref{txdef}). \\
$X_{i,n}(t)$ & Location of the $i$th particle at time $t$, when particles are ordered by position. \\
$Y_n(t)$ & Weighted sum in which a particle at $x$ contributes $e^{\rho x}$.  See (\ref{Yntdef}). \\
$Y_{n,A}^*(t)$ & Similar to $Y_n(t)$ but counting only particles that stay below $L_A$.  See (\ref{Ynstardef}). \\
$Z_n(t)$ & Measure of the ``size" of the process at time $t$, defined in (\ref{Zntdef}). \\
$Z_{n,A}(t)$ & Similar to $Z_n(t)$ but counting only particles below $L_A$ at time $t$.  See (\ref{ZnAdef}). \\
$Z_{n,A}^*(t)$ & Similar to $Z_{n,A}(t)$ but counting only particles that stay below $L_A$.  See (\ref{Znstardef}). \\
$z_{n,A}(x)$ & Contribution to $Z_{n,A}(t)$ from one particle at $x$.  See (\ref{znAdef}). \\
$\alpha(x)$ & Defined to be $Ai((2 \beta)^{1/3}x + \gamma_1)$. \\
$\beta = \beta_n$ & Selection parameter, defined so that $b_n(x) - d_n(x) = \beta_n x$. \\
$\gamma_k$ & $k$th-largest zero of the Airy function $Ai$. \\
$\zeta_n(t)$ & Scaled empirical distribution of particle locations at time $t$, defined in (\ref{zetadef}). \\
$\xi_n(t)$ & Distribution of mass at time $t$, when a particle at $x$ contributes $e^{\rho x}$.  See (\ref{xidef}). \\
$\rho = \rho_n$ & The Brownian particles have a drift of $-\rho$. \\
\end{longtable}

\section{Outline of the Proofs}

The proofs of Theorems \ref{GaussThm} and \ref{EdgeThm} rely on a combination of first moment estimates based on the many-to-one lemma, second moment estimates to control fluctuations, and careful truncation arguments.  We record in this section some of the intermediate results that are important for the argument and defer the more technical proofs until later sections.

We first introduce here some more notation. As mentioned in Section \ref{intro_sec}, from now on we will drop the subscript $n$ from much of our notation, for example writing $\rho$ and $\beta$ in place of $\rho_n$ and $\beta_n$, and writing $L$ in place of $L_n$.  We emphasize again that is important for the reader to keep in mind that these parameters do depend on $n$. We have only excluded the subscripts to lighten the burden of notation.

When the initial configuration consists of a single particle at the location $x$, we denote probabilities and expectations by $\P_x$ and $\E_x$. 
For real numbers $A$, we define
\begin{equation}\label{LAdef}
L_A = \frac{\rho^2}{2 \beta} - (2 \beta)^{-1/3} \gamma_1 - \frac{A}{\rho}.
\end{equation}
Note that $L_A$ depends on $n$, although again we omit the subscript.  Note also that $L_0 = L$. The terms in $L_A$ are listed in order of size: the dominant term remains $\rho^2/2\beta$, and there is an adjustment of order $\beta^{-1/3}$ as in $L$; we will see why $L$ has this form later, in \eqref{ptKeq}. The difference between $L_A$ and $L$ is only in the term of smallest order, $A/\rho$.

We also introduce the shorthand notation 
\begin{equation}\label{alphadef}
\alpha(x) = Ai((2 \beta)^{1/3} x + \gamma_1).
\end{equation}
For $A \in \R$ and $x \in \R$, we define
\begin{equation}\label{znAdef}
z_{n,A}(x) = e^{\rho x} Ai((2 \beta)^{1/3}(L_A - x) + \gamma_1)\1_{\{x < L_A\}} = e^{\rho x} \alpha(L_A - x)\1_{\{x < L_A\}}.
\end{equation}
For $A \in \R$ and $t \geq 0$, we define 
\begin{equation}\label{ZnAdef}
Z_{n,A}(t) = \sum_{i=1}^{N_n(t)} e^{\rho X_{i,n}(t)} Ai((2 \beta)^{1/3}(L_A - X_{i,n}(t)) + \gamma_1)\1_{\{X_{i,n}(t) < L_A\}} = \sum_{i=1}^{N_n(t)} z_{n,A}(X_{i,n}(t)).
\end{equation}
Note that $Z_n(t) = Z_{n,0}(t)$.  More generally, given $A \in \R$ and a function $\varphi: (-\infty, L_A) \rightarrow \R$, we define
\begin{equation}\label{Vdef}
V_{\varphi, n, A}(t) = \sum_{i=1}^{N_n(t)} e^{\rho X_{i,n}(t)} \varphi(X_{i,n}(t)) \1_{\{X_{i,n}(t) < L_A\}}.
\end{equation}

We introduce $L_A$ above because it will sometimes be necessary to consider a modification of the process in which particles are killed when they reach $L_A$.  In this case, it will be important to keep track of how various constants depend on $A$.  Given two sequences $(a_n)_{n=1}^{\infty}$ and $(b_n)_{n=1}^{\infty}$, when we write $a_n \lesssim b_n$ or $b_n \gtrsim a_n$, the ratio $a_n/b_n$ must be bounded above by a positive constant that does not depend on $A$. We write $a_n \asymp b_n$ to mean that both $a_n\lesssim b_n$ and $a_n\gtrsim b_n$ hold.  Throughout the rest of the paper, $C_k$ for a nonnegative integer $k$ will denote a fixed positive constant.  The value of $C_k$ may not depend on $n$ or $A$ and does not change from one occurrence to the next.

\subsection{The empirical distribution of particles}

A large part of our proofs will consist of detailed first and second moment estimates, allowing us to approximate the density of particles in different regions of space. We will detail these in sections \ref{density_unkilled_sec} to \ref{estimates_LA_sec} below. Putting these moment estimates aside, there are three main steps required to prove Theorems \ref{GaussThm} and \ref{EdgeThm}.  The first step requires considering the configuration of particles near the right edge.  As long as the initial configuration of particles satisfies (\ref{Zasm}) and (\ref{Yasm}), a short time later the particles near the right edge should settle into a relatively stable configuration, described by the density $h$ defined in (\ref{hdef}).  For particles within a distance of order $\beta^{-1/3}$ from the right edge, it should take a time of order $\beta^{-2/3}$ for the particles to reach this relatively stable configuration.  For technical reasons, to make it easier to control particles that drift a bit further than order $\beta^{-1/3}$ away from the right edge, we give the particles a bit more time to move into this configuration and establish the result below for times much larger than $\beta^{-2/3} \log(\rho/\beta^{1/3})$.  The proof of Proposition \ref{rtconfigprop}, whose statement is identical to Theorem~\ref{EdgeThm} except that we insist that $t_n\ll\rho/\beta$, uses techniques similar to those used in \cite{bbs3}, and is given in section \ref{rtedgesec}.

\begin{Prop}\label{rtconfigprop}
Suppose the initial configuration of particles satisfies (\ref{Zasm}) and (\ref{Yasm}).  Suppose
\begin{equation}\label{tcond1}
\beta^{-2/3} \bigg( \log \bigg(\frac{\rho}{\beta^{1/3}} \bigg) \bigg)^{1/3} \ll t_n \ll \frac{\rho}{\beta}.
\end{equation}
For $t \geq 0$, define $\xi_n(t)$ as in (\ref{xidef}).  Let $\nu$ be the probability measure on $(0,\infty)$ with probability density function $h$ defined in (\ref{hdef}).  Then $\xi_n(t_n) \Rightarrow \nu$ as $n \rightarrow \infty$.
\end{Prop}

The second step involves considering the configuration of particles near the origin.  As long as the initial configuration of particles satisfies (\ref{Zasm}) and (\ref{Yasm}), the particles near the origin at time approximately $\rho/\beta$, most of which will have been descended from particles that were near $L$ at time zero, should be approximately in a Gaussian configuration.  Note that (\ref{mainNn}) below, while not strictly required for the proof of Theorem \ref{GaussThm}, provides useful insight into the behavior of the process, as it says that the number of particles in the system at time approximately $\rho/\beta$ is determined, to a high degree of precision, by the value of $Z_n(0)$.  The proof of Proposition~\ref{gaussianprop}, which also uses techniques similar to those used in \cite{bbs3}, is given in section \ref{gausssec}.

\begin{Prop}\label{gaussianprop}
Suppose the initial configuration of particles satisfies (\ref{Zasm}) and (\ref{Yasm}).  Suppose 
\begin{equation}\label{tcond2}
\frac{\rho^{2/3}}{\beta^{8/9}} \ll t_n - \frac{\rho}{\beta} \ll \frac{\rho}{\beta}.
\end{equation}
Then, for all $\kappa > 0$, we have
\begin{equation}\label{mainNn}
\lim_{n \rightarrow \infty} \P \bigg(\frac{(1 - \kappa)}{(Ai'(\gamma_1))^2} e^{-\rho^3/3 \beta} Z_n(0) \leq N_n(t_n) \leq \frac{(1 + \kappa)}{(Ai'(\gamma_1))^2} e^{-\rho^3/3 \beta} Z_n(0) \bigg) = 1.
\end{equation}
For $t \geq 0$, define $\zeta_n(t)$ as in (\ref{zetadef}).  Let $\mu$ be the standard normal distribution.  Then $\zeta_n(t_n) \Rightarrow \mu$ as $n \rightarrow \infty$.
\end{Prop}

The third step is to show that if (\ref{Zasm}) and (\ref{Yasm}) hold, then these conditions will still hold with high probability at later times, which are order $\rho/\beta$ in the future.  This is established in the following proposition, which is proved in section \ref{YZsec}.

\begin{Prop}\label{YZmain}
Suppose the initial configuration of particles satisfies (\ref{Zasm}) and (\ref{Yasm}).  Suppose the times $t_n$ are chosen so that $$\lim_{n \rightarrow \infty} \frac{\beta t_n}{\rho} = \tau \in (0, \infty).$$  Then, with probability tending to one as $n \rightarrow \infty$, the conditions (\ref{Zasm}) and (\ref{Yasm}) hold with $Z_n(t_n)$ and $Y_n(t_n)$ in place of $Z_n(0)$ and $Y_n(0)$ respectively.
\end{Prop}

We now show how Propositions \ref{rtconfigprop}, \ref{gaussianprop}, and \ref{YZmain} imply Theorems \ref{GaussThm} and \ref{EdgeThm}.  Because these proofs involve subsequence arguments, we will return to writing $\rho_n$ and $\beta_n$ instead of $\rho$ and $\beta$ to avoid confusion.

\begin{proof}[Proof of Theorem \ref{GaussThm}]
It suffices to show that every subsequence $(n_j)_{j=1}^{\infty}$ contains a further subsequence $(n_{j_k})_{k=1}^{\infty}$ for which $\zeta_{n_{j_k}}(t_{n_{j_k}}) \Rightarrow \mu$ as $k \rightarrow \infty$.  By (\ref{rhobetagauss}), the sequence $$\bigg(\frac{\beta_n t_n}{\rho_n} - 1 \bigg)_{n=1}^{\infty}$$ is bounded and non-negative for large $n$.  Therefore, given a subsequence $(n_j)_{j=1}^{\infty}$, we can choose a further subsequence $(n_{j_k})_{k=1}^{\infty}$ for which $$\lim_{k \rightarrow \infty} \frac{\beta_{n_{j_k}} t_{n_{j_k}}}{\rho_{n_{j_k}}} - 1 = \tau \in [0, \infty).$$  If $\tau = 0$, then (\ref{tcond2}) holds along this subsequence, and it follows immediately from Proposition~\ref{gaussianprop} that $\zeta_{n_{j_k}}(t_{n_{j_k}}) \Rightarrow \mu$ as $k \rightarrow \infty$.  Suppose instead that $\tau > 0$.  Choose any sequence $(s_n)_{n=1}^{\infty}$ for which (\ref{tcond2}) holds with $s_n$ in place of $t_n$, and let $u_n = t_n - s_n$.  Note that $$\lim_{k \rightarrow \infty} \frac{\beta_{n_{j_k}} u_{n_{j_k}}}{\rho_{n_{j_k}}} = \tau \in (0, \infty),$$ so Proposition \ref{YZmain} implies that (\ref{Zasm}) and (\ref{Yasm}) hold with $Z_{n_{j_k}}(u_{n_{j_k}})$ and $Y_{n_{j_k}}(u_{n_{j_k}})$ in place of $Z_n(0)$ and $Y_n(0)$.  We can therefore apply the Markov property at time $u_{n_{j_k}}$, followed by Proposition \ref{gaussianprop} with $s_n$ in place of $t_n$, to see that $\zeta_{n_{j_k}}(t_{n_{j_k}}) \Rightarrow \mu$.
\end{proof}

\begin{proof}[Proof of Theorem \ref{EdgeThm}]
The proof is very similar to the proof of Theorem \ref{GaussThm}.  It suffices to show that every subsequence $(n_j)_{j=1}^{\infty}$ contains a further subsequence $(n_{j_k})_{k=1}^{\infty}$ for which $\xi_{n_{j_k}}(t_{n_{j_k}}) \Rightarrow \nu$ as $k \rightarrow \infty$.  By (\ref{tcond}), the sequence $(\beta_n t_n/\rho_n)_{n=1}^{\infty}$ is bounded.  Therefore, given a subsequence $(n_j)_{j=1}^{\infty}$, we can choose a further subsequence $(n_{j_k})_{k=1}^{\infty}$ for which $$\lim_{k \rightarrow \infty} \frac{\beta_{n_{j_k}} t_{n_{j_k}}}{\rho_{n_{j_k}}} = \tau \in [0, \infty).$$  If $\tau = 0$, then (\ref{tcond1}) holds along this subsequence, and it follows from Proposition \ref{rtconfigprop} that $\xi_{n_{j_k}}(t_{n_{j_k}}) \Rightarrow \nu$.  Suppose instead that $\tau > 0$.  Choose any sequence $(s_n)_{n=1}^{\infty}$ for which (\ref{tcond1}) holds with $s_n$ in place of $t_n$, and let $u_n = t_n - s_n$.  Note that $$\lim_{k \rightarrow \infty} \frac{\beta_{n_{j_k}} u_{n_{j_k}}}{\rho_{n_{j_k}}} = \tau \in (0, \infty),$$ so Proposition \ref{YZmain} implies that (\ref{Zasm}) and (\ref{Yasm}) hold with $Z_{n_{j_k}}(u_{n_{j_k}})$ and $Y_{n_{j_k}}(u_{n_{j_k}})$ in place of $Z_n(0)$ and $Y_n(0)$.  We can therefore apply the Markov property at time $u_{n_{j_k}}$, followed by Proposition \ref{rtconfigprop} with $s_n$ in place of $t_n$, to see that $\xi_{n_{j_k}}(t_{n_{j_k}}) \Rightarrow \nu$.
\end{proof}

\subsection{The density for the unkilled process}\label{density_unkilled_sec}

As mentioned above, a large amount of the work in our proofs involves moment estimates that allow us to bound the number of particles in certain regions of space. We begin with first moment calculations.

We denote by $p_t(x,y)$ the density for the process, which means that if there is one particle at $x$ at time zero, then the expected number of particles in the Borel set $D$ at time $t$ is given by $$\int_D p_t(x,y) \: dy.$$  To calculate the density, we can invoke the many-to-one lemma, which is proved, for example, in \cite{hh09}.  To do this, we first compute the density for the process when $\rho = 0$, which we denote by $q_t(x,y)$.  Let $(B_t)_{t \geq 0}$ be one-dimensional Brownian motion started at $B_0 = x$.  The many-to-one lemma states that if $f: \R \rightarrow \R$ is a nonnegative measurable function, then
\begin{equation}\label{many1}
\E_x \bigg[ \sum_{k=1}^{N(t)} f(X_{k,n}(t)) \bigg] = \E_x \bigg[ \exp \bigg( \int_0^t \beta B_u \: du \bigg) f(B_t) \bigg].
\end{equation}
Consequently, the density for the process without drift can be read from formula 1.8.7 on page~141 of \cite{bosa}, which yields
$$q_t(x,y)dy = \E_x \bigg[ \exp \bigg( \int_0^t \beta B_s \: ds \bigg) ; B_t \in dy \bigg] = \frac{1}{\sqrt{2 \pi t}} \exp \bigg( - \frac{(y - x)^2}{2t} + \frac{\beta(y + x) t}{2} + \frac{\beta^2 t^3}{24} \bigg)dy.$$
The drift of $-\rho$ can be added using a standard Girsanov transformation, which implies that
\begin{equation}\label{Girsanoveq}
p_t(x,y) = e^{\rho(x - y)} e^{-\rho^2 t/2} q_t(x,y)
\end{equation}
and therefore
\begin{equation}\label{ptxy}
p_t(x,y) = \frac{1}{\sqrt{2 \pi t}} \exp \bigg( \rho x - \rho y - \frac{(y - x)^2}{2t} - \frac{\rho^2 t}{2} + \frac{\beta(y + x) t}{2} + \frac{\beta^2 t^3}{24} \bigg).
\end{equation}

Integrating (\ref{ptxy}) with respect to $y$ gives
\begin{equation}\label{intpeq}
\int_{-\infty}^{\infty} p_t(x,y) \: dy = \exp \bigg( \beta x t + \frac{\beta^2 t^3}{6} - \frac{\beta \rho t^2}{2} \bigg).
\end{equation}
Alternatively, one can obtain (\ref{intpeq}) by applying the many-to-one formula \eqref{many1} to Brownian motion with drift, which gives
$$\int_{-\infty}^{\infty} p_t(x,y) \: dy = \E_x \bigg[ \exp \bigg( \int_0^t \beta (B_u - \rho u) \: du \bigg) \bigg] = \exp \bigg( - \frac{\beta \rho t^2}{2} \bigg) \E_x \bigg[ \exp \bigg( \int_0^t \beta B_u \: du \bigg) \bigg].$$  The result (\ref{intpeq}) then follows from equation 1.8.3 on page 141 of \cite{bosa}, which states that $$\E_x \bigg[ \exp \bigg( \int_0^t \beta B_u \: du \bigg) \bigg] = \exp \bigg( \beta x t + \frac{\beta^2 t^3}{6} \bigg).$$

\subsection{Brownian motion killed at a random time}\label{salminen_sec}

In order to control accurately the number of particles in certain regions of space, we will need to use moment estimates for a process where some of the particles are killed upon hitting a barrier. To carry out these calculations we will need estimates on Brownian motion killed at a random time, which we collect here.

Suppose that $x>0$ and $(B_t)_{t\ge 0}$ is a Brownian motion started from $x$. We will use an interpretation of the integral $\int_0^t \beta |B_u| du$ as a random clock, following Salminen \cite{ps88}. Let
\[T_0 = \inf\{t\ge 0 : B_t\le 0\},\]
the first time that the Brownian motion hits zero. Let $\mathcal E$ be an independent exponential random variable of parameter $1$, and define
\[\zeta = \inf\Big\{t\ge 0 : \int_0^t \beta |B_u| du > \mathcal E\Big\}.\]
Let
\[\hat X_t = \begin{cases} B_t & \text{if } t< \zeta\wedge T_0,\\
						   \Upsilon & \text{otherwise }\end{cases}\]
where $\Upsilon$ is some graveyard state. Then for any $x > 0$ and any Borel set $A\subseteq(0,\infty)$,
\begin{equation}\label{killedrepresentation}
\P_x(\hat X_t\in A) = \E_x[e^{-\int_0^t \beta B_u du} \1_{\{T_0>t\}}\1_{\{B_t\in A\}}].
\end{equation}
Let $\tau_0 = \inf\{t\ge 0 : \hat X_t = 0\}$, which takes values in $(0, \infty]$. Writing $\hat p_t(x,y)$ for the transition density of $\hat X_t$ and $\pi_x(t)$ for the density of $\tau_0$, we will need the following two results from \cite{ps88} (noting that the value of $\beta$ used in \cite{ps88} differs from ours by a factor of $2$ and the densities in \cite{ps88} are written with respect to twice Lebesgue measure):
\begin{equation}\label{salminen39}
\hat p_t(x,y) = (2\beta)^{1/3} \sum_{k=1}^\infty e^{\beta(2\beta)^{-1/3} \gamma_k t} \frac{Ai((2\beta)^{1/3}x + \gamma_k)Ai((2\beta)^{1/3}y + \gamma_k)}{Ai'(\gamma_k)^2},
\end{equation}
which is Proposition 3.9 of \cite{ps88}, and
\begin{equation}\label{salminen32}
\pi_x(t) = \lim_{y\downarrow 0}\frac{1}{2}\frac{\partial}{\partial y} \hat p_t(x,y),
\end{equation}
which is (3.2) in \cite{ps88}.

\subsection{The density for the process killed at level $\ell$}

Returning to our branching Brownian motion, we will often need to consider a truncated version of the process, in which particles are killed as soon as they surpass some level $\ell \in \R$.  For $x < \ell$ and $y < \ell$, we denote the density for this process by $p_t^{\ell}(x,y)$. Defining
\[T_{\ell} = \inf\{t\ge 0 : B_t \ge \ell\}\]
to be the first hitting time of $\ell$ by our Brownian motion $(B_t)_{t \geq 0}$, by the many-to-one formula and (\ref{Girsanoveq}) we have
\[p_t^{\ell}(x,y)dy = e^{\rho x - \rho y - \rho^2 t/2} \E_x[e^{\int_0^t \beta B_u du} \1_{\{T_{\ell} > t\}}\1_{\{B_t\in dy\}}].\]
Making the transformation $B'_u = \ell-B_u$, and setting $T'_0 = \inf\{t\ge 0 : B'_t \le 0\}$, we see that
\begin{equation}\label{ptl2}
p_t^{\ell}(x,y)dy = e^{\rho x - \rho y - \rho^2 t/2 + \beta \ell t} \E_{\ell-x}[e^{-\int_0^t \beta B'_u du} \1_{\{T'_0 > t\}}\1_{\{\ell-B'_t\in dy\}}].
\end{equation}
Recognising the last expectation as the transition density of the killed Brownian motion from Section \ref{salminen_sec}, from \eqref{salminen39} we have
\begin{align}\label{ptKeq}
p_t^{\ell}(x,y) &= (2 \beta)^{1/3} \sum_{k=1}^{\infty} \frac{e^{(\beta(\ell + (2 \beta)^{-1/3} \gamma_k) - \rho^2/2) t}}{Ai'(\gamma_k)^2} \nonumber \\
&\hspace{.5in} \times e^{\rho x} Ai((2 \beta)^{1/3}(\ell - x) + \gamma_k) e^{-\rho y} Ai((2 \beta)^{1/3}(\ell - y) + \gamma_k).
\end{align}
We will typically take $\ell=L_A$ for some real number $A$.  The exponent $(\beta(\ell + (2 \beta)^{-1/3} \gamma_1) - \rho^2/2) t$ in the leading term in (\ref{ptKeq}) is zero when $\ell = L$, which is why $L$ is the correct level at which to kill particles to keep the number of particles in the system approximately stable over time.

As a consequence of the formula (\ref{ptKeq}) for the density, it is possible to show that for any $A \in \R$, if we consider a modified process in which particles are killed when they reach $L_A$, then for all $t \geq 0$ and $x < L_A$, we have
\begin{equation}\label{zmeaneq}
\E_x[Z_{n,A}(t)] = e^{-A \beta t/\rho} z_{n,A}(x).
\end{equation}
In particular, the process $(Z_n(t), t \geq 0)$ is a martingale when particles are killed at $\ell$. This can be proved using \eqref{ptKeq}, the dominated convergence theorem, and the following orthogonality relation from Section 4.4 of \cite{vsairy}:
\begin{equation}\label{orth}
\int_0^{\infty} Ai(z + \gamma_j) Ai(z + \gamma_k) \: dz = \left\{
\begin{array}{ll}
(Ai'(\gamma_j))^2 & \mbox{ if } j = k \\
0 & \mbox{ otherwise.}
\end{array} \right.
\end{equation}
We will not use the fact that $(Z_n(t), t\ge0)$ is a martingale in the rest of the paper, so we do not include the full proof here, but it may provide insight into why $(Z_n(t), t \geq 0)$ is an important measure of the ``size" of the process.

\subsection{Approximate density formulas}\label{approxden}

We will need to establish some approximations to the density formulas (\ref{ptxy}) and (\ref{ptKeq}).  We will therefore state in this subsection six lemmas, all of which will be proved in section \ref{densec}.  As indicated in section \ref{ldsec}, most particles that are alive at time $t$ will be descended from particles that were near $L$ at time $t - \rho/\beta$.  Therefore, it will be useful to have the following approximation for $p_t(x,y)$, which holds when $x \approx \rho^2/2 \beta$ and $t \approx \rho/\beta$.  Note that here and throughout this subsection and the next two, the time $t$ implicitly depends on $n$.

\begin{Lemma}\label{firstmom}
Write $$t = \frac{\rho}{\beta} - s, \hspace{.5in} x = \frac{\rho^2}{2 \beta} - w.$$  If $|w| \ll \sqrt{\rho/\beta}$ and $0 \leq s \ll \rho^{1/4} \beta^{-3/4}$, then
\begin{align}\label{ptestimate}
p_t(x,y) &= \frac{1}{\sqrt{2 \pi}} \sqrt{ \frac{\beta}{\rho}} \exp \bigg( \rho x - \frac{\rho^3}{3 \beta} - \frac{\beta y^2}{2 \rho} - \frac{\beta^2 s^3}{6} + \beta w s \nonumber \\
&\hspace{1in}+ O \bigg( \frac{s^2 \beta^2 |y|}{\rho} \bigg) + O \bigg( \frac{s \beta^2 y^2}{\rho^2} \bigg) + O \bigg( \frac{\beta |w y|}{\rho} \bigg) + o(1) \bigg).
\end{align}
\end{Lemma}

It follows from (\ref{A1}) that if $|y| \lesssim \sqrt{\rho/\beta}$, then the last four terms inside the exponential in (\ref{ptestimate}) are all $o(1)$.  Among the other terms inside the exponential in (\ref{ptestimate}), only the $-\beta y^2/2 \rho$ term involves $y$.  Therefore, Lemma \ref{firstmom} indicates that for $x$ sufficiently close to $\rho^2/2 \beta$, the empirical distribution of particles at time $t$ closely matches the Gaussian distribution with mean $0$ and variance $\rho/\beta$.

Our remaining approximations pertain to the process in which particles are killed when they reach $L_A$.  As long as $t$ is large enough, and $x$ and $y$ are sufficiently close to $L_A$, the right-hand side of (\ref{ptKeq}) can be approximated by its leading term.  

\begin{Lemma}\label{dapprox4}
Suppose that $A\in\R$, $t \geq 0$, $x < L_A$ and $y < L_A$. Then
\begin{equation}\label{densapprox}
p_t^{L_A}(x,y) = \frac{(2 \beta)^{1/3} e^{-\beta A t/\rho}}{(Ai'(\gamma_1))^2} e^{\rho x} \alpha(L_A - x) e^{-\rho y} \alpha(L_A - y) \big(1 + E_A(t,x,y) \big)
\end{equation}
where $$|E_A(t,x,y)| \lesssim \sum_{k=2}^{\infty} e^{(\gamma_1 - \gamma_k)((2 \beta)^{1/6}((L_A-x)^{1/2} + (L_A-y)^{1/2}) - 2^{-1/3} \beta^{2/3} t)}.$$
In particular, if there exists a strictly positive constant $C$ such that
\begin{equation}\label{dcond}
(2 \beta)^{1/6}\big((L_A-x)^{1/2} + (L_A-y)^{1/2}\big) \leq 2^{-1/3} \beta^{2/3} t - C
\end{equation}
then there is a positive constant $C_0$ such that
\begin{equation}\label{densapproxrough}
p_t^{L_A}(x,y) \leq C_0 \beta^{1/3} e^{-\beta A t/\rho} e^{\rho x} \alpha(L_A - x) e^{-\rho y} \alpha(L_A - y),
\end{equation}
and if 
\begin{equation}\label{dcondbig}
(2 \beta)^{1/6}\big((L_A-x)^{1/2} + (L_A-y)^{1/2}\big) - 2^{-1/3} \beta^{2/3} t \rightarrow -\infty,
\end{equation}
then the error term $E_A(t, x, y)$ in (\ref{densapprox}) is $o(1)$.
\end{Lemma}

We will also need several additional formulas that can be used when (\ref{dcond}) is not satisfied.  Lemma \ref{dapprox1} is most useful when $t \leq \rho^{-2}$, while Lemma \ref{dapprox2} is useful for slightly larger values of $t$, particularly when both $x$ and $y$ are close to $L_A$.  Lemma \ref{dapprox3} is useful when $x$ is close to $L_A$, and $y$ is far enough away from $L_A$ that a particle going from $x$ to $y$ is unlikely to be affected by the right boundary at $L_A$, unless it hits the boundary almost immediately.

\begin{Lemma}\label{dapprox1}
For all $t \geq 0$, $\ell\ge 0$, $x < \ell$, and $y < \ell$, we have
\[p_t^{\ell}(x,y) \le \frac{1}{\sqrt{2\pi t}} \exp\bigg(\rho x - \rho y - \frac{(y-x)^2}{2t} - \frac{\rho^2 t}{2} + \beta \ell t\bigg).\]
\end{Lemma}

\begin{Lemma}\label{dapprox2}
For all $t \geq 0$, $\ell\ge 0$, $x < \ell$, and $y < \ell$, we have
\[p_t^\ell(x,y) \lesssim \frac{(\ell - x)(\ell - y)}{t^{3/2}} \exp \bigg( \rho x - \rho y - \frac{(y - x)^2}{2t} - \frac{\rho^2 t}{2} + \beta \ell t \bigg).\]
\end{Lemma}

\begin{Lemma}\label{dapprox3}
Suppose $A\in\R$, $t \geq 2 \beta^{-2/3}$, $0 \leq L_A - x \lesssim \beta^{-1/3}$, and $y < L_A$.  Then
\begin{align*}
p_t^{L_A}(x,y) &\lesssim \frac{\beta^{1/3} (L_A - x)}{\sqrt{t}} \max\bigg\{1, \frac{1}{\beta^{1/3} t} \Big( L_A - y - \frac{\beta t^2}{2} \Big) \bigg\} \\
&\hspace{.2in}\times \exp \bigg( \rho x - \rho y - \frac{(y - x)^2}{2t} - \frac{\rho^2 t}{2} + \frac{\beta(y + x) t}{2} + \frac{\beta^2 t^3}{24} + \frac{1}{2\beta^{1/3} t} \Big(L_A - y - \frac{\beta t^2}{2} \Big) \bigg).
\end{align*}
\end{Lemma}

The next result gives an estimate for the integral of the density for branching Brownian motion when the particles are killed at $L_A$ and the initial particle at $x$ is close to the right boundary.  The result can be compared to the formula (\ref{intpeq}) for the process without killing.

\begin{Lemma}\label{intpL}
Suppose $$t = \frac{\rho}{\beta} - s,$$ where $0 \leq s \ll t$, and suppose $A\in\R$ and $x < L_A$.  Then
$$\int_{-\infty}^{L_A} p_t^{L_A}(x,y) \: dy \lesssim \beta^{1/3} (L_A - x) \exp \bigg( \beta x t + \frac{\beta^2 t^3}{6} - \frac{\beta \rho t^2}{2} + \beta^{2/3} s \bigg).$$
\end{Lemma}

\subsection{Second moment estimates}

To control the fluctuations in the process, we will need good second moment bounds.  Given $A \in \R$ and $t \geq 0$, let
\begin{equation}\label{Kdef}
l(t) = \frac{\beta t^2}{33}, \hspace{.5in}K_A(t) = L_A - \frac{l(t)}{2} = L_A - \frac{\beta t^2}{66}.
\end{equation}
There is nothing particularly special about these functions. We will need to split the integrals that arise in our second moment bounds into several cases, and these functions will give convenient points at which to split, taking into account condition \eqref{dcond} in Lemma \ref{dapprox4}. It is worth noting that when $t\asymp \rho/\beta$, $l(t)$ is of order $\rho^2/\beta$ which is the same order as $L$, so $K_A(t)$ is really a long way below $L$ (and $L_A$). On the other hand, when $t\asymp \beta^{-2/3}$, which is much smaller than $\rho/\beta$, we have $l(t)\asymp \beta^{-1/3}$, which is the same order as the second-order term in $L$.

Let $\varphi: (-\infty, L_A) \rightarrow \R$ be a measurable function.  Lemma \ref{2momprop} establishes a second moment bound for the quantity $V_{\varphi, n, A}(t)$ defined in (\ref{Vdef}).  This bound will help us to control the fluctuations of the number of particles at time $t$ in the interval $(K_A(t), L_A)$.  The particles outside this interval will be controlled by other methods.  Note that in Lemma \ref{2momprop}, the time $t$ and the function $\varphi$ are implicitly allowed to depend on $n$.

\begin{Lemma}\label{2momprop}
Fix $A \geq 0$, and let $\varphi: (-\infty, L_A) \rightarrow \R$ be a bounded measurable function such that $\varphi(y) = 0$ unless $K_A(t) < y < L_A$.
Consider the process in which there is initially one particle at $x$ and particles are killed when they reach $L_A$.  Suppose $K_A(t) < x < L_A$, and suppose $t \gtrsim \beta^{-2/3}$.  Then
\begin{equation}\label{Vvareq}
\E_x[V_{\varphi, n, A}(t)^2] \lesssim \frac{\beta^{2/3} e^{\rho L_A}}{\rho^4} \big( e^{\rho x} + \beta^{2/3} t z_A(x) \big).
\end{equation}
\end{Lemma}

To prove Theorem \ref{GaussThm} and establish that the empirical distribution of particles is asymptotically Gaussian, we will also need to control the fluctuations in the number of particles close to the origin after a time that is approximately $\rho/\beta$.

\begin{Lemma}\label{mainvarg}
Consider the process in which there is initially one particle at $x$ and particles are killed when they reach $L$.
Write $$t = \frac{\rho}{\beta} - s.$$  Suppose $0 \leq L - x \ll \rho^2/\beta$ and $0 \leq s \ll t$.  Suppose $g: \R \rightarrow \R$ is a bounded measurable function.  Then
\begin{equation}\label{mainvargeq}
\E_x \bigg[ \bigg( \sum_{i=1}^{N(t)} g\big( X_i(t) \sqrt{\beta/\rho} \big) \bigg)^2 \bigg] \lesssim \frac{\beta^{2/3}}{\rho^4} \exp \bigg( \rho x + \rho L - \frac{2 \rho^3}{3 \beta} - \frac{\beta^2 s^3}{3}\bigg).
\end{equation}
\end{Lemma}

\noindent Because the proofs of Lemmas \ref{2momprop} and \ref{mainvarg} are rather tedious, we defer them until section \ref{2momproof}.

\subsection{Estimates for the particles that reach $L_A$}\label{estimates_LA_sec}

Fix $A \in \R$.  We estimate here the rate at which particles reach the right boundary at $L_A$, when particles are killed upon hitting this level. We will begin by considering a more general right boundary at $\ell$, before we later specialise to $\ell=L_A$. Suppose we start with one particle at $x<\ell$ and kill particles when they hit level $\ell$.  For $0 \leq u < v$, let $r_x^\ell(u,v)$ be the expected number of particles that hit $\ell$ between times $u$ and $v$.  Let ${\tilde r}_x^\ell(t)$ denote the rate at which particles hit $\ell$ at time $t$, defined by requiring that
$$r_x^\ell(u,v) = \int_u^v {\tilde r}_x^\ell(t) \: dt.$$

Recall that we defined $T_\ell = \inf\{t\ge 0 : B_t \ge \ell\}$, where $(B_t)_{t \geq 0}$ is Brownian motion, which means $\{T_{\ell} \leq t\}$ is the event that the Brownian particle hits $\ell$ before time $t$. To calculate $r_x^\ell(0,t)$, it is more convenient to think of freezing particles at $\ell$ (that is, they no longer move or branch) rather than killing them. By applying the many-to-one formula \eqref{many1} to this system, with $f(y) = \ind_{\{y\ge \ell\}}$, and additionally using (\ref{Girsanoveq}) to incorporate the drift, we get
\begin{equation}\label{rx0t}
r_x^\ell(0,t) = \E_x[e^{\rho x - \rho \ell - \rho^2 T_\ell/2 + \int_0^{T_\ell} \beta B_u du} \1_{\{T_\ell \le t\}}].
\end{equation}
Thus
\[\tilde r_x^\ell(t) dt = e^{\rho x - \rho \ell - \rho^2 t/2} \E_x[e^{\int_0^t \beta B_u du} \1_{\{T_\ell \in dt\}}],\]
and making the transformation $B'_u = \ell- B_u$ with $T'_0 = \inf\{t\ge 0 : B'_t \le 0\}$,
\[\tilde r_x^\ell(t) dt = e^{\rho x - \rho \ell - \rho^2 t/2 + \beta \ell t} \E_{\ell-x}[e^{-\int_0^t \beta B'_u du} \1_{\{T'_0 \in dt\}}].\]
Recognising the last expectation as the density of the hitting time of zero of the killed Brownian motion introduced in Section \ref{salminen_sec}, and recalling the notation of that section, from \eqref{salminen32} we have
\[\tilde r_x^\ell(t) = \frac{1}{2}e^{\rho x - \rho \ell - \rho^2 t/2 + \beta \ell t}\lim_{y\downarrow0}\frac{\partial}{\partial y} \hat p_t(\ell-x,y).\]
However, by comparing \eqref{killedrepresentation} and \eqref{ptl2}, we see that
$$p_t^{\ell}(x, \ell - y) = e^{\rho x + \rho y - \rho \ell - \rho^2 t/2 + \beta \ell t} \hat p_t(\ell - x, y),$$
and because $p_t^{\ell}(x, \ell) = 0$, it follows that 
\begin{equation}\label{rxformula}
\tilde r_x^\ell(t) = \frac{1}{2} \lim_{y \downarrow 0} \frac{\partial}{\partial y} e^{-\rho y} p_t^{\ell}(x, \ell - y) =
-\frac{1}{2} \lim_{y \uparrow \ell} \frac{\partial}{\partial y} p_t^{\ell}(x, y).
\end{equation}

Since we know that $p_t^\ell(x,\ell)=0$ and that the limit on the right-hand side of (\ref{rxformula}) exists, if $p_t^{\ell}(x, \ell - h) \leq dh$ for sufficiently small $h > 0$, then ${\tilde r}_x^\ell(t) \leq d$.  Therefore, the estimates on $p_t^{\ell}(x, \ell - h)$ from Lemmas \ref{dapprox4} and \ref{dapprox2} imply the following corollary.

\begin{Cor}\label{rtedge}
For all $t \geq 0$, $\ell>0$ and $x < \ell$, we have
\begin{equation}\label{rtdef}
{\tilde r}_x^{\ell}(t) \lesssim \frac{\ell - x}{t^{3/2}} \exp \bigg(\rho x - \rho \ell - \frac{(\ell - x)^2}{2t} - \frac{\rho^2 t}{2} + \beta \ell t \bigg).
\end{equation}
If $\ell=L_A$ for some fixed $A\in\R$ and, in addition, there exists a positive constant $C>0$ such that
\begin{equation}\label{Rtcond}
(2 \beta)^{1/6}(L_A - x)^{1/2} \leq 2^{-1/3} \beta^{2/3} t - C,
\end{equation}
then
\begin{equation}\label{rtlarge}
{\tilde r}_x^{L_A}(t) \lesssim \beta^{2/3} e^{-\beta A t/\rho} e^{\rho x} \alpha(L_A - x) e^{-\rho L_A}.
\end{equation}
\end{Cor}

We will also use the following lemma to estimate the number of descendants of a particle at $x$ that reach $L_A$.  This estimate involves bounding separately the expected number of descendants that hit $L_A$ during an initial time period, for which we can use the bound in (\ref{rtdef}), and the number of descendants that hit $L_A$ later, for which the bound in (\ref{rtlarge}) is valid.  This result will be proved in section \ref{rtsec}.

\begin{Lemma}\label{Rxlem}
For $A \in \R$, $x < L_A$, and $0\le s<t$, if we define $A_- = \max\{-A, 0\}$, then
\[r_x^{L_A}(s,t) \lesssim e^{\rho x} e^{-\rho L_A} e^{-\beta^2 s^3/9} + (t - s) e^{-\rho L_A} \beta^{2/3} z_A(x) e^{\beta A_- t/\rho}.\]
\end{Lemma}

Finally, we return to the original process in which particles are not killed when they reach $L_A$.  Lemma \ref{Lsurvive} below shows that the probability that a descendant of a particle that reaches $L_A$ will survive for a reasonably long time is of the order $\rho^2$.  Note that because a particle near $L_A$ has an effective branching rate of $b(x)-d(x) \approx \rho^2/2$, this result is to be expected in view of classical results on the survival probability of Galton-Watson processes.  A complication is that the birth rate changes as the particles move.  Note that in \cite{bbs, bms}, stronger results were obtained related to the number of surviving descendants of particles that reached the right boundary, and these results were essential for establishing convergence to a continuous-state branching process.  However, the weaker result established in Lemma \ref{Lsurvive} will be sufficient for our purposes.  Lemma~\ref{Lsurvive} will also be proved in section \ref{rtsec}.

\begin{Lemma}\label{Lsurvive}
Suppose at time zero, there is a single particle at $x=x_n$, with $\beta^{-1/3}\ll x\ll \beta^{-1}$.  There is a positive constant $C_1$ such that for sufficiently large $n$, the probability that any individual survives until time $C_1/(\beta x)$ is at most $2\beta x/\Delta$, where $\Delta$ is the constant from \eqref{A3}.
\end{Lemma}

\section{Proofs of density approximations}\label{densec}

In this section, we prove the results stated in section \ref{approxden}, which establish approximate formulas for the densities $p_t(x,y)$ and $p_t^{L_A}(x,y)$.  

\subsection{Facts about Airy functions}

In this subsection we collect some facts about Airy functions which will be needed later in the paper.  In particular, Lemma \ref{Airyratlem} below will be important for the proof of Lemma \ref{dapprox4}.

We first record some facts which can be found in \cite{vsairy}.  Below $C_2$, $C_3$, $C_4$, and $C_5$ are positive constants.  Using $\sim$ to indicate that the ratio of the two sides tends to one, we have
\begin{equation}\label{Airyasymp}
Ai(x) \sim \frac{1}{2 \sqrt{\pi} x^{1/4}} e^{-(2/3) x^{3/2}} \hspace{.2in}\mbox{as }x \rightarrow \infty,
\end{equation}
which is (2.45) in \cite{vsairy}.  We will also use that
\begin{equation}\label{zeros}
-\gamma_k \sim C_2 k^{2/3} \hspace{.2in}\mbox{as }k \rightarrow \infty,
\end{equation}
which can be seen from (2.54) and (2.66) in \cite{vsairy}, and
\begin{equation}\label{zerosder}
|Ai'(\gamma_k)| \sim C_3 k^{1/6} \hspace{.2in}\mbox{as }k \rightarrow \infty,
\end{equation}
which can be seen from (2.60) and (2.68) in \cite{vsairy}.  We have
\begin{equation}\label{upairy}
|Ai(x)| \leq C_4 \hspace{.2in}\mbox{for all }x \in \R,
\end{equation}
which can be deduced from (\ref{Airyasymp}), equation (2.49) in \cite{vsairy}, and the continuity of the Airy function.  Finally, we will use that
\begin{equation}\label{airyder}
|Ai'(x)| \lesssim |x|^{1/4} \hspace{.2in}\mbox{for }x \leq -1,
\end{equation}
which follows from (2.50) in \cite{vsairy}, and that
\begin{equation}\label{airyder2}
|Ai'(x)| \leq C_5 \hspace{.2in}\mbox{for all }x \geq -1,
\end{equation}
which can be deduced from equation (2.46) in \cite{vsairy} and the continuity of the derivative of the Airy function.

\begin{Lemma}\label{Airyratlem}
For all $z \geq 0$ and $k \in \N$, we have
\begin{equation}\label{Aikbound}
|Ai(\gamma_k + z)| \lesssim k^{1/6} e^{(\gamma_1 - \gamma_k)z^{1/2}} Ai(\gamma_1 + z).
\end{equation}
\end{Lemma}

\begin{proof}
Let $A$ be a positive constant.  If $0 \leq z \leq A$, then it follows from (\ref{zeros}), along with the facts that $Ai'(\gamma_1) > 0$ and $Ai(z) > 0$ for all $z > \gamma_1$, that
$$|Ai(\gamma_k + z)| \lesssim k^{1/6} z \lesssim k^{1/6} Ai(\gamma_1 + z),$$ which is stronger than (\ref{Aikbound}).

Next, suppose $A \leq z \leq -\gamma_k + A$.  Then $|Ai(\gamma_k + z)| \leq C_4$ by (\ref{upairy}), and $$Ai(\gamma_1 + z) \gtrsim (\gamma_1 + z)^{-1/4} e^{-(2/3)(\gamma_1 + z)^{3/2}}$$ by (\ref{Airyasymp}).  It follows that
\begin{equation}\label{midb1}
|Ai(\gamma_k + z)| \lesssim (\gamma_1 + z)^{1/4} e^{(2/3)(\gamma_1 + z)^{3/2}} Ai(\gamma_1 + z).
\end{equation}
In view of (\ref{zeros}), we have 
\begin{equation}\label{midb2}
(\gamma_1 + z)^{1/4} \leq (\gamma_1 - \gamma_k + A)^{1/4} \lesssim k^{1/6}.
\end{equation}
Also, recalling that $\gamma_1 < 0$,
$$\frac{2}{3}(\gamma_1 + z)^{3/2} = \frac{2}{3} (\gamma_1 + z)(\gamma_1 + z)^{1/2} \leq \frac{2}{3} (\gamma_1 - \gamma_k + A) z^{1/2} = (\gamma_1 - \gamma_k) z^{1/2} + \bigg(\frac{2A}{3} -\frac{\gamma_1 - \gamma_k}{3} \bigg) z^{1/2}.$$
Because $2A - (\gamma_1 - \gamma_k) < 0$ for sufficiently large $k$ and $z \leq -\gamma_k + A$, it follows that there is a positive constant $A'$ such that
\begin{equation}\label{midb3}
\frac{2}{3}(\gamma_1 + z)^{3/2} \leq (\gamma_1 - \gamma_k) z^{1/2} + A'.
\end{equation}
The result (\ref{Aikbound}) when $A \leq z \leq -\gamma_k + A$ follows from (\ref{midb1}), (\ref{midb2}), and (\ref{midb3}).

Finally, suppose $z \geq -\gamma_k + A$.  We can apply (\ref{Airyasymp}) to get
\begin{equation}\label{bigb1}
\frac{Ai(\gamma_k + z)}{Ai(\gamma_1 + z)} \lesssim \bigg( \frac{\gamma_1 + z}{\gamma_k + z} \bigg)^{1/4} e^{(2/3)[(\gamma_1 + z)^{3/2} - (\gamma_k + z)^{3/2}]}.
\end{equation}
The first factor is maximized when $z = -\gamma_k + A$, so using (\ref{zeros}), we get
\begin{equation}\label{bigb2}
\bigg(\frac{\gamma_1 + z}{\gamma_k + z} \bigg)^{1/4} \leq \bigg( \frac{\gamma_1 - \gamma_k + A}{A} \bigg)^{1/4} \lesssim k^{1/6}.
\end{equation}
Also, because $\gamma_1$ and $\gamma_k$ are negative and $\frac{d}{dz} z^{3/2} = \frac{3}{2} z^{1/2}$, we have
\begin{equation}\label{bigb3}
(\gamma_1 + z)^{3/2} - (\gamma_k + z)^{3/2} \leq \frac{3}{2} (\gamma_1 - \gamma_k) z^{1/2}.
\end{equation}
The result (\ref{Aikbound}) follows from (\ref{bigb1}), (\ref{bigb2}), and (\ref{bigb3}).
\end{proof}

\subsection{Proofs of Lemmas \ref{dapprox4}, \ref{dapprox1}, \ref{dapprox2}, and \ref{dapprox3}}

We begin by using Lemma \ref{Airyratlem} to provide the necessary error estimates to prove Lemma \ref{dapprox4}.

\begin{proof}[Proof of Lemma \ref{dapprox4}]
The expression (\ref{densapprox}) with $E_A = 0$ is the $k = 1$ term from (\ref{ptKeq}), as can be seen by recalling (\ref{LAdef}).  Denote by $r_k(t,x,y)$ the ratio of the $k$th term in (\ref{ptKeq}) to the first term, when $\ell = L_A$.  
For all $t \geq 0$, $x \leq L_A$, and $y \leq L_A$, we have
$$r_k(t, x, y) = \frac{(Ai'(\gamma_1))^2}{(Ai'(\gamma_k))^2} \cdot \frac{Ai((2 \beta)^{1/3}(L_A - x) + \gamma_k) Ai((2 \beta)^{1/3}(L_A - y) + \gamma_k)}{Ai((2 \beta)^{1/3}(L_A - x) + \gamma_1) Ai((2 \beta)^{1/3}(L_A - y) + \gamma_1)} \cdot \frac{e^{\beta (2 \beta)^{-1/3} \gamma_k t}}{e^{\beta (2 \beta)^{-1/3} \gamma_1 t}}.$$
Using (\ref{zerosder}) to bound the first factor and Lemma \ref{Airyratlem} to bound the second factor, we get
\begin{align*}
|r_k(x, t, y)| &\lesssim k^{-1/3} \cdot k^{1/3} e^{(\gamma_1 - \gamma_k)(2 \beta)^{1/6}[(L_A - x)^{1/2} + (L_A - y)^{1/2}]} \cdot e^{-(\gamma_1 - \gamma_k) 2^{-1/3} \beta^{2/3} t} \\
&= e^{(\gamma_1 - \gamma_k)((2 \beta)^{1/6}[(L_A - x)^{1/2} + (L_A - y)^{1/2}] - 2^{-1/3} \beta^{2/3} t)},
\end{align*}
which implies (\ref{densapprox}).  We then obtain (\ref{densapproxrough}) and the last conclusion of the lemma by estimating $E_A(t, x, y)$ using (\ref{zeros}).
\end{proof}

The proof of Lemma \ref{dapprox1} uses the fact that no particles go above $\ell$ to bound the branching rate, but otherwise uses the unkilled process. This explains the fact that it is similar to, but not the same as, \eqref{ptxy}.

\begin{proof}[Proof of Lemma \ref{dapprox1}]
Fix $a < b \le \ell$.  Let $N_t(a,b)$ denote the number of particles in $(a, b)$ at time $t$ that have never hit $\ell$.  Let $(B_s)_{s \geq 0}$ be Brownian motion started at $x$.  By applying the many-to-one formula \eqref{many1} to Brownian motion with drift, we get
\begin{align*}
\E_x[N_t(a,b)] &= \E_x\left[e^{\int_0^t \beta (B_s-\rho s) ds} \1_{\{B_s-\rho s<\ell \, \forall s\le t,\,B_t-\rho t\in(a,b)\}}\right]\\
&\le e^{\beta \ell t}\P_x(B_s-\rho s<\ell\,\, \forall s\le t,\,B_t-\rho t\in(a,b))\\
&\le e^{\beta \ell t}\P_x(B_t-\rho t\in(a,b))\\
&=\frac{e^{\beta \ell t}}{\sqrt{2\pi t}} \int_a^b \exp\Big(-\frac{(y-x+\rho t)^2}{2t}\Big) dy\\
&=\frac{e^{\beta \ell t}}{\sqrt{2\pi t}}\int_a^b \exp\Big(\rho x - \rho y - \frac{(y-x)^2}{2t} - \frac{\rho^2 t}{2} \Big) dy
\end{align*}
so
\[p_t^{\ell}(x,y) \le \frac{1}{\sqrt{2\pi t}} \exp\Big(\rho x - \rho y - \frac{(y-x)^2}{2t} - \frac{\rho^2 t}{2} + \beta \ell t  \Big),\]
as claimed.
\end{proof}

The proof of Lemma \ref{dapprox2} uses a trivial bound on the branching rate, but uses the killed process as opposed to the unkilled process.

\begin{proof}[Proof of Lemma \ref{dapprox2}]
We proceed as in the proof of Lemma \ref{dapprox1} but keep the restriction that $B_s-\rho s< \ell$ for all $s \le t$, and apply Girsanov's theorem followed by the reflection principle. This gives
\begin{align*}
\E_x[N_t(a,b)] &= \E_x\left[e^{\int_0^t \beta (B_s-\rho s) ds} \1_{\{B_s-\rho s<\ell\, \forall s\le t,\,B_t-\rho t\in(a,b)\}}\right]\\
&\le e^{\beta \ell t}\P_x(B_s-\rho s<\ell\,\,\, \forall s\le t,\,B_t-\rho t\in(a,b))\\
&= e^{\beta \ell t}\P_0(B_s+\rho s>x-\ell\,\,\, \forall s\le t,\,B_t+\rho t\in(x-b,x-a))\\
&= e^{\beta \ell t - \rho^2 t/2}\E_0[e^{\rho B_t}\1_{\{B_s>x-\ell\,\,\, \forall s\le t,\,B_t\in(x-b,x-a)\}}]\\
&= e^{\beta \ell t - \rho^2 t/2}\int_{x-b}^{x-a} e^{\rho y}\frac{1}{\sqrt{2\pi t}}\left(e^{-y^2/2t} - e^{-(y+2(\ell-x))^2/2t}\right) dy\\
&= e^{\beta \ell t - \rho^2 t/2}\int_a^b e^{\rho x - \rho y}\frac{1}{\sqrt{2\pi t}}\left(e^{-(x-y)^2/2t} - e^{-(x-y+2(\ell-x))^2/2t}\right) dy.
\end{align*}
Applying the elementary bound
\[e^{-A^2/2t} - e^{-(A+2B)^2/2t} = e^{-A^2/2t}(1-e^{-2B(A+B)/t}) \le \frac{2B}{t}(A+B)e^{-A^2/2t}\]
gives
\[\E_x[N_t(a,b)] \le e^{\beta \ell t - \rho^2 t/2}\int_a^b e^{\rho x - \rho y}\frac{2^{1/2}(\ell-x)(\ell-y)}{\pi^{1/2}t^{3/2}}e^{-(x-y)^2/2t} dy\]
and therefore
\[p_t^{\ell}(x,y) \le e^{\beta \ell t - \rho^2 t/2 + \rho x - \rho y}\frac{2^{1/2}(\ell-x)(\ell-y)}{\pi^{1/2}t^{3/2}}e^{-(x-y)^2/2t},\]
which yields the result.
\end{proof}

Finally, the proof of Lemma \ref{dapprox3} is more involved and combines the formula (\ref{ptxy}) with the estimate in Lemma \ref{dapprox2}. 

\begin{proof}[Proof of Lemma \ref{dapprox3}]
Let $s = \beta^{-2/3}$.
We apply the Chapman-Kolmogorov equation at time $s$ and then use Lemma \ref{dapprox2} to bound $p_s^{L_A}(x,z)$ and (\ref{ptxy}) to bound $p_{t-s}(z,y)$, which gives
\begin{align}\label{approx31eq}
p_t^{L_A}(x,y) &= \int_{-\infty}^{L_A} p_s^{L_A}(x,z) p_{t-s}^{L_A}(z,y) \: dz \nonumber \\
&\lesssim \int_{-\infty}^{L_A} \frac{(L_A - x)(L_A - z)}{s^{3/2}} \exp \bigg( \rho x - \rho z - \frac{(z - x)^2}{2s} - \frac{\rho^2 s}{2} + \beta L_A s \bigg) \nonumber \\
&\hspace{.2in}\times \frac{1}{\sqrt{t - s}} \exp \bigg( \rho z - \rho y - \frac{(y - z)^2}{2(t - s)} - \frac{\rho^2(t - s)}{2} + \frac{\beta(z + y)(t-s)}{2} + \frac{\beta^2(t - s)^3}{24} \bigg) \: dz \nonumber \\
&= \frac{L_A - x}{s^{3/2} \sqrt{t - s}} \exp \bigg( \rho x - \rho y - \frac{\rho^2 t}{2} + \beta L_A t - \frac{\beta(L_A - y)(t - s)}{2} + \frac{\beta^2(t - s)^3}{24} \bigg) \nonumber \\
&\hspace{.2in} \times \int_{-\infty}^{L_A} (L_A - z) \exp \bigg( - \frac{(x - z)^2}{2s} - \frac{(y - z)^2}{2(t - s)} - \frac{\beta(L_A - z)(t-s)}{2} \bigg) \: dz.
\end{align}
Now let ${\tilde x} = L_A - x$, ${\tilde y} = L_A - y$, and ${\tilde z} = L_A - z$.  After a few lines of algebra, we obtain
\begin{align*}
\frac{({\tilde x} - {\tilde z})^2}{2s} + \frac{({\tilde z} - {\tilde y})^2}{2(t-s)} + \frac{\beta {\tilde z} (t - s)}{2} &= \frac{t}{2s(t - s)} \bigg({\tilde z} - \bigg( \frac{(t-s){\tilde x}}{t} + \frac{s{\tilde y}}{t} - \frac{\beta s (t-s)^2}{2t} \bigg) \bigg)^2 \\
&\hspace{.2in} + \frac{({\tilde y} - {\tilde x})^2}{2t} - \frac{\beta^2 s (t - s)^3}{8t} + \frac{\beta (t-s)^2 {\tilde x}}{2t} + \frac{\beta s (t-s) {\tilde y}}{2t}.
\end{align*}
Therefore, substituting ${\tilde z}$ in place of $z$ in the integral in (\ref{approx31eq}), we get
\begin{align}\label{approx32eq}
p_t^{L_A}(x,y)& \lesssim \frac{L_A - x}{s^{3/2} \sqrt{t - s}} \exp \bigg( \rho x - \rho y - \frac{(y - x)^2}{2t} - \frac{\rho^2 t}{2} + \beta L_A t - \frac{\beta(t - s)^2(L_A - x)}{2t} \bigg) \nonumber \\
&\hspace{.2in} \times \exp \bigg(- \frac{\beta(L_A - y)(t - s)}{2} - \frac{\beta s (t - s)(L_A - y)}{2t} + \frac{\beta^2(t - s)^3}{24} + \frac{\beta^2 s (t - s)^3}{8t} \bigg) \nonumber \\
&\hspace{.2in} \times \int_0^{\infty} {\tilde z} \exp \bigg( - \frac{t}{2s(t-s)} \bigg({\tilde z} - \bigg( \frac{(t-s){\tilde x}}{t} + \frac{s{\tilde y}}{t} - \frac{\beta s (t-s)^2}{2t} \bigg) \bigg)^2 \bigg) \: d{\tilde z}.
\end{align}
Because we are assuming $L_A - x \lesssim \beta^{-1/3}$ and $s=\beta^{-2/3}\le t$, we have
\begin{equation}\label{approx33eq}
- \frac{\beta(t - s)^2 (L_A - x)}{2t} = - \frac{\beta(L_A - x)t}{2} + \beta s (L_A - x) - \frac{\beta s^2 (L_A - x)}{2t} = - \frac{\beta (L_A - x)t}{2} + O(1).
\end{equation}
Note that
\begin{equation}\label{approx34eq}
- \frac{\beta(L_A - y)(t - s)}{2} - \frac{\beta s (t - s)(L_A - y)}{2t} = - \frac{\beta (L_A - y) t}{2} + \frac{\beta(L_A - y) s^2}{2t},
\end{equation}
and
\begin{equation}\label{approx35eq}
\frac{\beta^2 (t - s)^3}{24} + \frac{\beta^2 s (t - s)^3}{8t} = \frac{\beta^2 t^3}{24} - \frac{\beta^2 s^2 t}{4} + O(1).
\end{equation}
Also, because $s = \beta^{-2/3}$
\begin{equation}\label{approx36eq}
\frac{\beta(L_A - y) s^2}{2t} - \frac{\beta^2 s^2 t}{4} = \frac{1}{2 \beta^{1/3} t} \Big(L_A - y - \frac{\beta t^2}{2} \Big).
\end{equation}
Using (\ref{approx33eq}), (\ref{approx34eq}), (\ref{approx35eq}), and (\ref{approx36eq}) to simplify the exponential terms in (\ref{approx32eq}), we get
\begin{align}\label{approx37eq}
p_t^{L_A}(x,y) &\lesssim \frac{L_A - x}{s^{3/2} \sqrt{t - s}} \nonumber \\
&\hspace{.2in} \times \exp \bigg( \rho x - \rho y - \frac{(y - x)^2}{2t} - \frac{\rho^2 t}{2} + \frac{\beta(y + x) t}{2} + \frac{\beta^2 t^3}{24} + \frac{1}{2\beta^{1/3} t} \Big(L_A - y - \frac{\beta t^2}{2} \Big) \bigg) \nonumber \\
&\hspace{.2in} \times \int_0^{\infty} {\tilde z} \exp \bigg( - \frac{t}{2s(t-s)} \bigg({\tilde z} - \bigg( \frac{(t-s){\tilde x}}{t} + \frac{s{\tilde y}}{t} - \frac{\beta s (t-s)^2}{2t} \bigg) \bigg)^2 \bigg) \: d{\tilde z}
\end{align}

For $z \in \R$, we will write $z_+ = \max\{0, z\}$ for the positive part of $z$.
It is easy to check that if $Z$ has a normal distribution with mean $\mu \in \R$ and standard deviation $\sigma > 0$, then $E[Z_+] \lesssim \mu_+ + \sigma$.  Therefore, 
\begin{align*}
&\int_0^{\infty} {\tilde z} \exp \bigg( - \frac{t}{2s(t-s)} \bigg({\tilde z} - \bigg( \frac{(t-s){\tilde x}}{t} + \frac{s{\tilde y}}{t} - \frac{\beta s (t-s)^2}{2t} \bigg) \bigg)^2 \bigg) \: dz \nonumber \\
&\hspace{1in}\lesssim \sqrt{\frac{s(t-s)}{t}} \bigg( \bigg( \frac{(t-s){\tilde x}}{t} + \frac{s{\tilde y}}{t} - \frac{\beta s (t-s)^2}{2t} \bigg)_+ + \sqrt{\frac{s(t-s)}{t}} \bigg).
\end{align*}
Because $s = \beta^{-2/3}$ and $t - s \asymp t$, we have $\sqrt{s(t-s)/t} \asymp \beta^{-1/3}$.  Using that ${\tilde x} \lesssim \beta^{-1/3}$, we have $$\bigg( \frac{(t-s){\tilde x}}{t} + \frac{s{\tilde y}}{t} - \frac{\beta s (t-s)^2}{2t} \bigg)_+ \leq \bigg( {\tilde x} + \frac{s}{t} \Big( {\tilde y} - \frac{\beta t^2}{2} \Big) + \beta s^2 \bigg)_+ \lesssim \beta^{-1/3} + \frac{s}{t} \Big(L_A - y - \frac{\beta t^2}{2} \Big)_+,$$
and therefore
\begin{align}\label{approx38eq}
&\int_0^{\infty} {\tilde z} \exp \bigg( - \frac{t}{2s(t-s)} \bigg({\tilde z} - \bigg( \frac{(t-s){\tilde x}}{t} + \frac{s{\tilde y}}{t} - \frac{\beta s (t-s)^2}{2t} \bigg) \bigg)^2 \bigg) \: dz \nonumber \\
&\hspace{.5in}\lesssim \beta^{-2/3} + \frac{\beta^{-1/3} s}{t} \Big(L_A - y - \frac{\beta t^2}{2} \Big)_+ \lesssim \beta^{-2/3} \max \bigg\{ 1,\frac{1}{\beta^{1/3} t} \Big(L_A - y - \frac{\beta t^2}{2} \Big) \bigg\}.
\end{align}
Because $s^{-3/2} = \beta$ and $t - s \asymp t$, the lemma follows from (\ref{approx37eq}) and (\ref{approx38eq}).
\end{proof}

\subsection{Proof of Lemma \ref{firstmom}}

From (\ref{ptxy}), we have
\begin{equation}\label{ptformula}
p_t(x,y) = \frac{1}{\sqrt{2 \pi t}} \exp \bigg( \rho x - \rho y - \frac{\rho^2 t}{2} - \frac{(y - x)^2}{2t} + \frac{\beta(y + x) t}{2} + \frac{\beta^2 t^3}{24} \bigg).
\end{equation}
From (\ref{A1}), we have $s \ll t$, and therefore
\begin{equation}\label{pt1}
\frac{1}{\sqrt{2 \pi t}} \exp \bigg( \rho x - \frac{\rho^2 t}{2} \bigg) \sim \frac{1}{\sqrt{2 \pi}} \sqrt{ \frac{\beta}{\rho}} \exp \bigg( \rho x - \frac{\rho^3}{2 \beta} + \frac{\rho^2 s}{2} \bigg).
\end{equation}
Also,
\begin{equation}\label{pt2}
\frac{\beta^2 t^3}{24} = \frac{\rho^3}{24 \beta} - \frac{\rho^2 s}{8} + \frac{\beta \rho s^2}{8} - \frac{\beta^2 s^3}{24}.
\end{equation}
Next, we observe that
\begin{align}\label{pt3}
- \rho y + \frac{\beta (y+x) t}{2} &= - \rho y + \frac{\beta}{2}  \bigg(y + \frac{\rho^2}{2 \beta} - w \bigg) \bigg( \frac{\rho}{\beta} - s \bigg) \nonumber \\
&= - \frac{\rho y}{2} - \frac{\beta y s}{2} + \frac{\rho^3}{4 \beta} - \frac{\rho^2 s}{4} - \frac{\rho w}{2} + \frac{\beta w s}{2}.
\end{align}
Furthermore, we have
\begin{align}\label{quadterm}
- \frac{(y-x)^2}{2t} &= - \frac{\beta}{2 \rho (1 - s \beta/\rho)} \bigg(\frac{\rho^2}{2 \beta} - w - y \bigg)^2 \nonumber \\
&= - \frac{\beta}{2 \rho} \bigg( \frac{\rho^2}{2 \beta} - w - y \bigg)^2 \sum_{k=0}^{\infty} \bigg( \frac{s \beta}{\rho} \bigg)^k \nonumber \\
&= - \frac{\rho^3}{8 \beta} \sum_{k=0}^{\infty} \bigg( \frac{s \beta}{\rho} \bigg)^k  + \frac{\rho(w + y)}{2} \sum_{k=0}^{\infty} \bigg( \frac{s \beta}{\rho} \bigg)^k - \frac{\beta (w + y)^2}{2 \rho} \sum_{k=0}^{\infty} \bigg( \frac{s \beta}{\rho} \bigg)^k.
\end{align}
To estimate the first term in (\ref{quadterm}), we note that the terms in the infinite sum with $k \geq 4$ are small due to (\ref{A1}) and the assumption that $s \ll \rho^{1/4} \beta^{-3/4}$.  Therefore,
\begin{equation}\label{pt4}
- \frac{\rho^3}{8 \beta} \sum_{k=0}^{\infty} \bigg( \frac{s \beta}{\rho} \bigg)^k = - \frac{\rho^3}{8 \beta} - \frac{\rho^2 s}{8} - \frac{\beta \rho s^2}{8} - \frac{\beta^2 s^3}{8} + o(1).
\end{equation}
To estimate the second term in (\ref{quadterm}), we again use (\ref{A1}) and the hypotheses on $|w|$ and $s$ to get
\begin{equation}\label{pt5}
\frac{\rho(w + y)}{2} \sum_{k=0}^{\infty} \bigg( \frac{s \beta}{\rho} \bigg)^k = \frac{\rho (w + y)}{2} + \frac{s \beta (w + y)}{2} + O \bigg( \frac{s^2 \beta^2 |y|}{\rho} \bigg) + o(1).
\end{equation}
For the third term in (\ref{quadterm}), we have
\begin{equation}\label{pt6}
- \frac{\beta (w + y)^2}{2 \rho} \sum_{k=0}^{\infty} \bigg( \frac{s \beta}{\rho} \bigg)^k = - \frac{\beta y^2}{2 \rho} + O \bigg( \frac{s \beta^2 y^2}{\rho^2} \bigg) + O \bigg( \frac{\beta |w y|}{\rho} \bigg) + o(1).
\end{equation}
We can now evaluate the right-hand side of (\ref{ptformula}) by putting together the results in (\ref{pt1}), (\ref{pt2}), (\ref{pt3}), (\ref{pt4}), (\ref{pt5}), and (\ref{pt6}), which yields (\ref{ptestimate}).

\subsection{Proof of Lemma \ref{intpL}}

Because $p_t^{L_A}(x,y) \leq p_t(x,y)$, it follows from (\ref{intpeq}) that it suffices to prove the result when $L_A - x \leq \beta^{-1/3}$, which we will assume for the rest of the proof.
Note that $L_A = \rho^2/2 \beta + O(\beta^{-1/3})$, and $\beta^{-2/3}/t \rightarrow 0$ by (\ref{A1}).  Therefore,
\begin{align*}
\frac{1}{2\beta^{1/3} t} \bigg(L_A - y - \frac{\beta t^2}{2} \bigg) &= \frac{1}{2 \beta^{1/3} t} \bigg( \frac{\rho^2}{2 \beta} + O(\beta^{-1/3}) - y - \frac{\beta}{2} \Big( \frac{\rho}{\beta} - s \Big)^2 \bigg) \\
&= \frac{1}{2 \beta^{1/3} t} \bigg( \rho s - \frac{\beta s^2}{2} - y \bigg) + o(1).
\end{align*}
Combining this estimate with Lemma \ref{dapprox3}, and separating out the terms involving $y$ in the exponential factor, we get
\begin{align*}
p_t^{L_A}(x,y) &\lesssim \frac{\beta^{1/3}(L_A - x)}{\sqrt{t}} \max\bigg\{1, \frac{1}{\beta^{1/3} t} \Big(\rho s - \frac{\beta s^2}{2} - y \Big) \bigg\} \\
&\hspace{.5in} \times \exp \bigg(\rho x - \frac{\rho^2 t}{2} + \frac{\beta x t}{2} + \frac{\beta^2 t^3}{24} + \frac{\rho s}{2 \beta^{1/3} t} - \frac{\beta^{2/3} s^2}{4t} \bigg) \\
&\hspace{1in} \times \exp \bigg(- \frac{(y - x)^2 + 2 \rho y t - \beta y t^2 + \beta^{-1/3} y}{2t} \bigg).
\end{align*}
Note that if $a \in \R$, then $(y - x)^2 + ay = (y - (x - a/2))^2 + ax - a^2/4$.  Applying this result to the second exponential factor above with $$a = 2 \rho t - \beta t^2 + \beta^{-1/3},$$ we get
\begin{align}\label{expnoy}
p_t^{L_A}(x,y) &\lesssim \frac{\beta^{1/3}(L_A - x)}{\sqrt{t}} \max\bigg\{1, \frac{1}{\beta^{1/3} t} \Big(\rho s - \frac{\beta s^2}{2} - y \Big) \bigg\}
\exp \bigg(- \frac{1}{2t} \Big(y - \Big(x - \frac{a}{2} \Big) \Big)^2 \bigg) \nonumber \\
&\hspace{.5in} \times \exp \bigg(\rho x - \frac{\rho^2 t}{2} + \frac{\beta x t}{2} + \frac{\beta^2 t^3}{24} + \frac{\rho s}{2 \beta^{1/3} t} - \frac{\beta^{2/3} s^2}{4t} - \frac{a x}{2t} + \frac{a^2}{8t} \bigg). 
\end{align}
After some tedious but straightforward algebra, we see that the exponential factor on the second line of (\ref{expnoy}) can be written as
\begin{equation}\label{expnoy2}
\exp \bigg( \beta x t + \frac{\beta^2 t^3}{6} - \frac{\beta \rho t^2}{2} \bigg) \exp \bigg(\frac{1}{8 \beta^{2/3} t} + \frac{2 \rho s + 2 \rho t - 2x - \beta s^2 - \beta t^2}{4\beta^{1/3}t} \bigg).
\end{equation}
Now using that $t = \rho/\beta - s$, and that $x = \rho^2/2 \beta + O(\beta^{-1/3})$ by the definition of $L_A$ and the assumption that $L_A - x \leq \beta^{-1/3}$, we can write the second exponential factor in (\ref{expnoy2}) as
\begin{equation}\label{expnoy3}
\exp \bigg(\frac{2 \rho s - 2 \beta s^2}{4 \beta^{1/3} t} + O \Big( \frac{1}{\beta^{2/3} t} \Big) \bigg) = \exp \bigg(\frac{\beta^{2/3} s}{2} + O \Big( \frac{1}{\beta^{2/3} t} \Big) \bigg).
\end{equation}
Substituting the results in (\ref{expnoy2}) and (\ref{expnoy3}) back into (\ref{expnoy}), using that $1/(\beta^{2/3} t) \rightarrow 0$, and integrating with respect to $y$, we get
\begin{align}\label{intptL}
\int_{-\infty}^{L_A} p_t^{L_A}(x,y) \: dy &\lesssim \beta^{1/3} (L_A - x) \exp \bigg( \beta x t + \frac{\beta^2 t^3}{6} - \frac{\beta \rho t^2}{2} + \frac{\beta^{2/3} s}{2} \bigg) \nonumber \\
&\hspace{.1in}\times \int_{-\infty}^{L_A} \frac{1}{\sqrt{t}} \max\bigg\{1, \frac{1}{\beta^{1/3} t} \Big(\rho s - \frac{\beta s^2}{2} - y \Big) \bigg\}
\exp \bigg(- \frac{1}{2t} \Big(y - \Big(x - \frac{a}{2} \Big) \Big)^2 \bigg) \: dy.
\end{align}
The integral in the previous line is bounded above by $$\sqrt{2 \pi} \E \bigg[ \max\bigg\{1, \frac{1}{\beta^{1/3} t} \Big(\rho s - \frac{\beta s^2}{2} - V \Big) \bigg\} \bigg],$$ where $V$ has a normal distribution with mean $x - a/2$ and variance $t$.  This is the same as $$\sqrt{2 \pi} \E[\max\{1, W\}],$$ where $W$ has a normal distribution with variance $$\sigma^2 = \frac{1}{(\beta^{1/3} t)^2} \cdot t = \frac{1}{\beta^{2/3} t} \rightarrow 0$$
and mean
\begin{align*}
\mu &= \frac{1}{\beta^{1/3} t} \bigg( \rho s - \frac{\beta s^2}{2} - \Big( x - \frac{a}{2} \Big) \bigg) \\
&= \frac{1}{\beta^{1/3} t} \bigg( \rho s - \frac{\beta s^2}{2} - \frac{\rho^2}{2 \beta} + \rho t - \frac{\beta t^2}{2} + O(\beta^{-1/3}) \bigg) \\
&= \frac{\rho s - \beta s^2}{\beta^{1/3} t} + O \bigg( \frac{1}{\beta^{2/3} t} \bigg) \\
&= \beta^{2/3} s + o(1).
\end{align*}
We therefore have $\E[\max\{1, W\}] \lesssim \max\{1, \beta^{2/3} s\} \lesssim \exp( \beta^{2/3}s/2).$  Because this expression gives an upper bound for the integral in (\ref{intptL}), the result follows.

\section{Particles that hit $L_A$}\label{rtsec}

In this section, we will prove Lemmas \ref{Rxlem} and \ref{Lsurvive}.  Lemma \ref{Rxlem} pertains to the process in which particles are killed when they reach $L_A$, and provides an estimate for how many particles are killed.  Lemma \ref{Lsurvive} pertains to the process in which particles are allowed to continue after reaching $L_A$.  It provides an estimate for the probability that the descendants of a particle will survive for a significant period of time after the particle reaches $L_A$.  We also prove Lemma \ref{YfromL} below, which bounds the expected contribution to $Y_n(t)$, for small times $t$, from an initial particle at $L_A$.  We begin with the following simple integral estimate.

\begin{Lemma}\label{32int}
If $a > 0$ and $b > 0$, then
$$\int_0^{\infty} \frac{1}{x^{3/2}} e^{-b^2/ax} \: dx = \frac{\sqrt{\pi a}}{b}.$$
\end{Lemma}

\begin{proof}
Make the substitution $y = b^2/ax$ and then recall that $\Gamma(1/2) = \sqrt{\pi}$.
\end{proof}

\begin{proof}[Proof of Lemma \ref{Rxlem}]
Define
\begin{equation}\label{txdef}
t_{x,A} = \sqrt{\frac{2 (L_A - x)}{\beta}} + 2^{1/3}\beta^{-2/3}.
\end{equation}
We first show that
\begin{equation}\label{rxsmall}
r_x^{L_A}(0,t_{x,A}) \lesssim e^{\rho x} e^{-\rho L_A} e^{-\beta^2 t_{x,A}^3/9}.
\end{equation}
Let $0 < \delta < 2^{-1/3}$.  It follows from Corollary \ref{rtedge} and the fact that $\gamma_1 < 0$ that
\begin{align}\label{rxLA}
r_x^{L_A}(0, t_{x,A}) &=\int_0^{t_{x,A}} {\tilde r}_x^{L_A}(u) \: du \nonumber \\
&\lesssim \int_0^{t_{x,A}} \frac{L_A - x}{u^{3/2}} \exp \bigg( \rho x - \rho L_A - \frac{(L_A - x)^2}{2u} - 2^{-1/3} \gamma_1 \beta^{2/3} u - \frac{A \beta u}{\rho} \bigg) \: du \nonumber \\
&\le (L_A - x) \exp \bigg( \rho x - \rho L_A - \frac{(1 - \delta)(L_A - x)^2}{2 t_{x,A}} - 2^{-1/3} \gamma_1 \beta^{2/3} t_{x,A} + \frac{A_- \beta t_{x,A}}{\rho} \bigg) \nonumber \\
&\hspace{.5in}\times \int_0^{t_{x,A}} \frac{1}{u^{3/2}} e^{-\delta (L_A - x)^2/2u} \: du.
\end{align}
It follows from Lemma \ref{32int} that 
$$\int_0^{t_{x,A}} \frac{1}{u^{3/2}} e^{-\delta (L_A - x)^2/2u} \: du \lesssim \frac{1}{L_A - x}.$$  Therefore to prove \eqref{rxsmall} it suffices to show that
\begin{equation}\label{exp5nts}
\exp \bigg( - \frac{(1 - \delta)(L_A - x)^2}{2 t_{x,A}} + \Big(- 2^{-1/3} \gamma_1 + \frac{A_- \beta^{1/3}}{\rho} \Big) \beta^{2/3} t_{x,A} \bigg) \lesssim e^{-\beta^2 t_{x,A}^3/9}.
\end{equation}
We consider two cases.  First, suppose $t_{x,A} \leq \delta^{-1} \beta^{-2/3}$.  Then (\ref{exp5nts}) holds because the left-hand side is bounded above by a positive constant, while the right-hand side is bounded below by a positive constant.  Next, suppose $t_{x,A} \geq \delta^{-1} \beta^{-2/3}$.  It follows from (\ref{txdef}) that
\[\sqrt{2(L_A - x)/\beta} = (1-2^{1/3}\delta)t_{x,A} + 2^{1/3}\delta t_{x,A} - 2^{1/3}\beta^{-2/3} \ge (1-2^{1/3}\delta)t_{x,A},\]
and therefore $$L_A - x \geq \frac{\beta (1 - 2^{1/3} \delta)^2 t_{x,A}^2}{2}.$$  Also, we can bound $A_- \beta^{1/3}/\rho$ from above by $C$ because $\beta^{1/3}/\rho \rightarrow 0$.  Therefore, the left-hand side of (\ref{exp5nts}) is bounded above by
\begin{align*}
&\exp \bigg( - \frac{(1 - \delta)(1 - 2^{1/3} \delta)}{8} \beta^2 t_{x,A}^3 + (-2^{-1/3} \gamma_1 + C) \beta^{2/3} t_{x,A} \bigg) \\
&\hspace{.2in}= \exp \bigg( - \bigg( \frac{(1 - \delta)(1 - 2^{1/3} \delta)}{8} - \frac{1}{9} \bigg) (\beta^{2/3} t_{x,A})^3 + (-2^{-1/3} \gamma_1 + C) \beta^{2/3} t_{x,A} \bigg) e^{-\beta^2 t_{x,A}^3/9}.
\end{align*}
By choosing $\delta$ sufficiently small and using that the function $y \mapsto -ay^3 + by$ is bounded above for all $a > 0$ and $b > 0$, we see that the quantity above can be upper bounded by a constant multiple of $e^{-\beta^2 t_{x,A}^3/9}$.  It follows that
(\ref{exp5nts}) holds, and therefore so does \eqref{rxsmall}. In particular this proves the lemma in the case $s\le t\le t_{x,A}$.

Next, suppose that $t_{x,A} \leq s \leq t$. Note that $t_{x,A}$ has been defined so that (\ref{Rtcond}) holds with equality when $t = t_{x,A}$ and $C = 1$.
If $s \leq u \leq t$, then (\ref{Rtcond}) holds with $u$ in place of $t$.  Therefore, by Corollary \ref{rtedge},
\begin{equation}\label{rxlate}
\int_s^t {\tilde r}_x^{L_A}(u) \: du \lesssim \int_s^t \beta^{2/3} e^{-\beta A u/\rho} e^{\rho x} \alpha(L_A - x) e^{-\rho L_A} \: du \leq (t - s) e^{-\rho L_A} \beta^{2/3} z_A(x) e^{\beta A_- t/\rho},
\end{equation}
which proves the lemma in the case $t_{x,A}\le s \le t$. For the remaining case, when $s\le t_{x,A}\le t$, simply note that
\[r_x^{L_A}(s,t) \le r_x^{L_A}(0,t_{x,A}) + r_x^{L_A}(t_{x,A},t)\]
and combine the bounds from \eqref{rxsmall} and \eqref{rxlate}.
\end{proof}

Before we prove Lemma \ref{Lsurvive}, we show that our process cannot explode in finite time.

\begin{Lemma}\label{noexplosion}
Suppose that, at time zero, there is a single particle at $x$. Then for all $t > 0$, the random variable
\[M = \sup\{X_{i,n}(u) : i\le N_n(u),\, u \le t\}\]
is finite almost surely.
\end{Lemma}

\begin{proof}
Fix $k > x$ and consider a system where particles are frozen (that is, they no longer move or branch) once they hit level $k$. Call the resulting probability measure $\hat\P$. Let $\mathcal A$ be the number of particles that have been frozen by time $t$.  Let $(B_u)_{u \geq 0}$ be a process which, under $\hat\P$, is a Brownian motion with drift $-\rho$ that is frozen upon hitting the level $k$.  By applying the many-to-one lemma (\ref{many1}) under the measure $\hat\P$ with the function $f(y)=\ind_{\{y\ge k\}}$, we get
\[\P_x(M \ge k) = \hat \P_x(\mathcal A\ge 1) \le \hat \E_x[\mathcal A] \le \hat\E_x\Big[\exp\Big(\int_0^t \beta B_u du\Big) \ind_{\{B_t = k\}}\Big].\]
Since $B_u\le k$ for all $u$ under $\hat\P$, we deduce that
\[\P_x(M \ge k) \le e^{\beta k t} \hat\P_x\big( B_t = k \big).\]
Now, the probability that $B_t = k$ is exactly the probability that a Brownian motion with drift $-\rho$ hits $k$ before time $t$, which is smaller than the probability that a Brownian motion with no drift hits $k$ before time $t$. It is well-known that the first hitting time of level $y$ by a standard Brownian motion $(W_t)_{t\ge 0}$ started from $0$ is equal in distribution to $(y/W_1)^2$; thus, using a standard Gaussian approximation,
\[\P_x(M \ge k) \le e^{\beta k t}\P\Big(\frac{(k-x)^2}{W_1^2} \le t \Big) = 2e^{\beta k t}\P\Big(W_1 \ge \frac{k-x}{t^{1/2}}\Big) \lesssim \frac{t^{1/2}}{k-x} \exp\Big(\beta kt - \frac{(k-x)^2}{2t} \Big),\]
which converges to $0$ as $k\to\infty$. The result follows.
\end{proof}

\begin{proof}[Proof of Lemma \ref{Lsurvive}]
Fix $\delta > 0$.  Let $K_m = x(1 + \delta)^m$ for each $m\ge 0$.  Label the initial particle at $x$ to be type $0$.  Whenever a particle reaches $K_m$ for the first time, it becomes type $m$.  The type of a particle is never allowed to decrease, so a particle will have type $m$ if at some time it had an ancestor above $K_m$, but it never had an ancestor above $K_{m+1}$.  When a birth occurs, offspring have the same type as the parent.

For nonnegative integers $m$, let $$\tau_m = \frac{\log(1/\delta)}{\beta x (1 + \delta)^{m+1}},$$ and let $T_m = \sum_{k=0}^m \tau_k$.  Let
\begin{equation}\label{Tdef}
T = \lim_{m \rightarrow \infty} T_m = \sum_{k=0}^{\infty} \tau_k = \frac{\log(1/ \delta)}{\beta x \delta}.
\end{equation}
Let $D_0$ be the event that there is a type $0$ individual alive in the population at time $T_0$.  For positive integers $m$, let $D_m$ be the event that there are type $m$ individuals in the population continuously from time $T_{m-1}$ until time $T_m$.  Let $D$ be the event that some individual survives until time $T$.  We claim that
\begin{equation}\label{DDm}
\P(D) \le \P\Big(\bigcup_{m=0}^{\infty} D_m\Big).
\end{equation}
To see this, note that if $D_0$ fails to occur, then there are no type 0 individuals left at time $T_0$, but there could be individuals that migrated to the right of $K_1$ and became type 1 individuals.  If $D_1$ also fails to occur, then the type 1 individuals must all be gone by time $T_1$, but there could be individuals that became type $2$.  Repeating this argument, we see that if none of the $D_i$ occur, then there cannot be individuals of any type remaining at time $T$. The only further possibility is that there are individuals alive at time $t$ that have had type $j+1$ by time $T_j$, for all $j\in\N$. By Lemma \ref{noexplosion}, this has probability zero.

We therefore aim to bound the probability of $D_m$.  Suppose $m \geq 1$. For any $u\ge 0$, by (\ref{rtdef}),
$$\tilde r_x^{K_m}(u) \lesssim \frac{K_m - x}{u^{3/2}} \exp \bigg( \rho (x - K_m) - \frac{(K_m-x)^2}{2u} - \frac{\rho^2 u}{2} + \beta K_m u \bigg).$$
Note that if $0 \leq u \leq T$, then
$$\beta K_m u \leq \beta x (1+\delta)^m \frac{\log(1/\delta)}{\beta x\delta}  = \frac{(1+\delta)^m\log(1/\delta)}{\delta}.$$
Therefore, using also that $\rho^2 u/2 \geq 0$, we have
$${\tilde r}_{x}^{K_m}(u) \lesssim \frac{K_m - x}{u^{3/2}} \exp \bigg( \frac{(1 + \delta)^m \log(1/\delta)}{\delta} + \rho (x - K_m)\bigg) \exp \bigg( - \frac{(x - K_m)^2}{2u} \bigg).$$  It follows from this bound and Lemma \ref{32int} that the expected number of particles to hit $K_m$ by time $T_{m-1}$, if particles are killed upon reaching $K_m$, is at most
\begin{align}\label{hitKm}
\int_0^{T_{m-1}} {\tilde r}_x^{K_m}(u) \: du &\lesssim \exp \bigg( \frac{(1 + \delta)^m \log(1/\delta)}{\delta} + \rho (x - K_m)\bigg) \nonumber \\
&= \exp \bigg( \frac{(1 + \delta)^{m} \log(1/\delta)}{\delta} + \rho x - \rho x (1 + \delta)^m \bigg).
\end{align}

Now suppose, for $m \geq 1$, a particle reaches $K_m$ before time $T_{m-1}$.  For $m = 0$, we can consider instead the particle at $x$ at time zero.  Particles of type $m$ have positions $x\le K_{m+1}$ and therefore effective branching rate $b(x)-d(x)\le\beta K_{m+1}$.  For a continuous-time branching process in which each individual gives birth at rate $\lambda$ and dies at rate $\mu$, it is well-known that the probability that the population survives for at least time $t$ is given by
\[\frac{\lambda - \mu}{\lambda - \mu e^{-(\lambda - \mu)t}} = \frac{\lambda - \mu}{\lambda - \mu + \mu(1 - e^{-(\lambda - \mu)t})}.\]
This formula can be deduced, for example, from results in Section 5 in Chapter III of \cite{athney}.

For any $\mu>0$ and $t>0$, one may check that the derivative of the function $z\mapsto\frac{z + \mu(1 - e^{-zt})}{z}$
is always negative and therefore the function is decreasing in $z$. Thus the function
\[z\mapsto \frac{z}{z + \mu(1 - e^{-zt})}\]
is increasing in $z$. Applying this with $z=\lambda-\mu\le \beta K_{m+1}$, $\mu\ge\Delta$ (from \eqref{A3}) and $t=\tau_m$, we see that the probability that a particle that reaches $K_m$ before time $T_{m-1}$ has 
descendants of type $m$ alive in the population at time $T_m$ is bounded above by
\begin{equation}\label{survivep1}
\frac{\beta K_{m+1}}{\beta K_{m+1} + \Delta(1 - e^{-\beta K_{m+1} \tau_m})} \le \frac{\beta K_{m+1}}{\Delta(1 - e^{-\beta K_{m+1} \tau_m})}.
\end{equation}
By our choice of $\tau_m$, this equals
\begin{equation}\label{GWsurvive}
\frac{\beta x(1+\delta)^{m+1}}{\Delta(1-\delta)}.
\end{equation}
When $m = 0$, it follows that
$$\P(D_0) \leq \frac{\beta x (1 + \delta)}{\Delta(1 - \delta)}.$$
When $m \geq 1$, it follows from (\ref{hitKm}) and the estimate in (\ref{GWsurvive}) that
$$\P(D_m) \leq \frac{\beta x (1 + \delta)^{m+1}}{\Delta(1 - \delta)} \exp \bigg( (1 + \delta)^m \frac{\log(1/\delta)}{\delta} +\rho x - \rho x (1+\delta)^m \bigg).$$
Therefore, using (\ref{DDm}),
\begin{equation}\label{PD}
\P(D) \leq \frac{\beta x (1 + \delta)}{\Delta(1 - \delta)} + \sum_{m=1}^{\infty} \frac{\beta x (1 + \delta)^{m+1}}{\Delta(1 - \delta)} \exp \bigg( (1 + \delta)^m \frac{\log(1/\delta)}{\delta} +\rho x - \rho x (1+\delta)^m \bigg).
\end{equation}
Let $\Gamma_m$ denote the $m$th term in the infinite sum in (\ref{PD}).  Because $x\gg \beta^{-1/3}\gg \rho^{-1}$ as $n \rightarrow \infty$, we see that for any fixed $\delta$, we have $\Gamma_1 \rightarrow 0$ and $\sup_{m \geq 1} \Gamma_{m+1}/\Gamma_m \rightarrow 0$.  It follows that the infinite sum in (\ref{PD}) tends to zero, and therefore as long as $\delta$ is chosen to be small enough that $(1 + \delta)/(1 - \delta) < 2$, we have $\P(D) \leq 2\beta x/\Delta$ for sufficiently large $n$.  In view of (\ref{Tdef}), it follows that the conclusion of the lemma holds with $C_1 = \log(1/\delta)/\delta$.
\end{proof}

\begin{Lemma}\label{YfromL}
If $0 < t \lesssim \rho^{-2}$, then $$\int_{-\infty}^{\infty} p_t(L_A, y) e^{\rho y} \: dy \lesssim e^{\rho L_A}.$$
\end{Lemma}

\begin{proof}
By (\ref{ptxy}),
$$\int_{-\infty}^{\infty} p_t(L_A, y) e^{\rho y} \: dy = \int_{-\infty}^{\infty} \frac{1}{\sqrt{2 \pi t}} \exp \bigg( \rho L_A - \frac{(y - L_A)^2}{2t} - \frac{\rho^2 t}{2} + \frac{\beta (L_A + y)t}{2} + \frac{\beta^2 t^3}{24} \bigg) \: dy.$$
Because $t \lesssim \rho^{-2}$, the terms $\rho^2t/2$, $\beta L_A t/2$, and $\beta^2 t^3/24$ are all bounded above by positive constants.  It follows that
$$\int_{-\infty}^{\infty} p_t(L_A, y) e^{\rho y} \: dy \lesssim e^{\rho L_A} \int_{-\infty}^{\infty} \frac{1}{\sqrt{2 \pi t}} e^{\beta y t/2} e^{-(y - L_A)^2/2t} \: dy = e^{\rho L_A} \cdot e^{\beta t L_A/2} e^{\beta^2 t^3/8} \lesssim e^{\rho L_A},$$
as claimed.
\end{proof}

\section{Proof of Proposition \ref{rtconfigprop}}\label{rtedgesec}

In this section, we will prove Proposition \ref{rtconfigprop}, which gives a precise description of the density of particles near the right edge.  Because it will sometimes be necessary to condition on the initial configuration of particles, we will define $({\cal F}_t, t \geq 0)$ to be the natural filtration associated with the branching Brownian motion process.

We begin with the following lemma, which states that when (\ref{Yasm}) holds, particles that start out close to $L$ can be neglected because they will not have descendants surviving for very long.

\begin{Lemma}\label{aboveL}
Suppose (\ref{Yasm}) holds.  Fix $A \in \R$, and let $C_1>0$ be the constant from Lemma \ref{Lsurvive}.  Then the probability that some particle that is to the right of $L_A$ at time zero has a descendant alive in the population at time $2C_1 \rho^{-2}$ tends to zero as $n \rightarrow \infty$.
\end{Lemma}

\begin{proof}
First note that any particle that is to the right of $\rho/\beta$ at time zero contributes at least $e^{\rho^2/\beta}\gg \rho^{-2}e^{\rho L}$ to $Y_n(0)$. It therefore follows from \eqref{Yasm} that the probability that any particle is to the right of $\rho/\beta$ at time zero tends to zero as $n\to\infty$, and we can restrict our attention to particles that start to the left of $\rho/\beta$.

Fix $\eps>0$. Say that a particle at time $0$ ``survives'' if it has a descendant alive at time $2C_1\rho^{-2}$.  That is, we say the particle survives if the family initiated by the particle lasts until at least time $2C_1 \rho^{-2}$.  Let $\mathcal S$ be the number of particles whose positions at time $0$ are between $L_A$ and $\rho/\beta$ and who survive; and define $E_0 = \E[\mathcal S|\F_0]$. Then, using the conditional Markov inequality,
\begin{align*}
\P(\mathcal S\ge 1) = \E\big[\P(\mathcal S\ge 1 |\F_0)\big] &= \E\big[\P(\mathcal S\ge 1 |\F_0)\ind_{\{E_0\le \eps\}}\big] + \E\big[\P(\mathcal S\ge 1 |\F_0)\ind_{\{E_0> \eps\}}\big]\\
&\le \E\big[\E[\mathcal S | \F_0] \ind_{\{E_0\le \eps\}}\big] + \E\big[1\cdot\ind_{\{E_0>\eps\}}\big]\\
&\le \eps + \P(E_0>\eps).
\end{align*}
Since $\eps>0$ was arbitrary, it therefore suffices to show that $\P(E_0>\eps)\to 0$ as $n\to\infty$.

Consider a particle $i$ with $X_i(0)\in[L_A,\rho/\beta]$.
Lemma \ref{Lsurvive} gives an upper bound for the probability that this particle has a descendant alive at time $C_1/\beta X_i(0)$.  Because $X_i(0)\ge L_A$ and $L_A \geq \rho^2/2 \beta$ for large $n$, we have $C_1/\beta X_i(0) \leq 2C_1 \rho^{-2}$ for large $n$.  Therefore, for large $n$, the bound in Lemma \ref{Lsurvive} is an upper bound for the probability that this particle has a descendant alive until time $2C_1 \rho^{-2}$.  Thus, by Lemma \ref{Lsurvive}, a particle $i$ with $X_i(0)\in[L_A,\rho/\beta]$ survives with probability at most $2\beta X_i(0)/\Delta$.  It follows that
\begin{align}
E_0 &\le \sum_{i=1}^{N_n(0)}\ind_{\{X_i(0)\in[L_A,\rho/\beta]\}}\frac{2\beta X_i(0)}{\Delta}\nonumber\\
&= \frac{2}{\Delta}\sum_{i=1}^{N_n(0)}\ind_{\{X_i(0)\in[L_A,\rho/\beta]\}}\beta(X_i(0)-L) + \frac{2 L}{\Delta}\sum_{i=1}^{N_n(0)}\beta\ind_{\{X_i(0)\in[L_A,\rho/\beta]\}}.\label{E0bound1}
\end{align}
We now use the elementary bound $x\le e^x$ together with \eqref{A1} to say that, for large $n$,
\[\beta(X_i(0)-L) \le \rho^2\cdot\rho(X_i(0)-L) \le \rho^2 e^{\rho X_i(0) - \rho L}\]
which, combined with \eqref{E0bound1}, gives
\[E_0 \le \frac{2\rho^2}{\Delta}\sum_{i=1}^{N_n(0)} e^{\rho X_i(0) - \rho L} + \frac{2\beta L}{\Delta}\sum_{i=1}^{N_n(0)}\ind_{\{X_i(0)\ge L_A\}}.\]
The first sum on the right-hand side tends to $0$ in probability by \eqref{Yasm}. The second also tends to $0$ in probability, since
\[\P\Big(\frac{2\beta L}{\Delta}\sum_{i=1}^{N_n(0)}\ind_{\{X_i(0)\ge L_A\}} \ge \eps \Big) \le \P\Big(\beta L \sum_{i=1}^{N_n(0)} e^{\rho X_i(0) - \rho L_A} \ge \frac{\Delta\eps}{2}\Big) = \P\Big(\beta L e^{-\rho L} Y_n(0) \ge \frac{\Delta\eps e^{-A}}{2}\Big)\]
which converges to $0$ as $n\to\infty$ by \eqref{Yasm}. This completes the proof.
\end{proof}

We now introduce some additional notation.  Recall from (\ref{Kdef}) that $K_A(t) = L_A - \beta t^2/66$.  We also define
\begin{equation}\label{HAdef}
H_A(t) = L_A - \frac{\beta t^2}{9},
\end{equation}
so that $H_A(t) < K_A(t) < L_A.$  We set $H(t) = H_0(t)$ and $K(t) = K_0(t)$.  Note that if $H_A(t_n) \leq x < L_A$ and $H_A(t_n) \leq y < L_A$, then as long as $t_n \gg \beta^{-2/3}$, the condition (\ref{dcondbig}) holds, and therefore Lemma \ref{dapprox4} can be used to estimate $p_{t_n}^{L_A}(x,y)$.  When $K_A(t_n) < x < L_A$ and $K_A(t_n) < y < L_A$, Lemma \ref{2momprop} can be used for second moment bounds.  Other methods are needed to control the contribution from particles to the left of $H_A(t)$.  The next lemma will be very useful in this regard.

\begin{Lemma}\label{lowxylem}
If $x \leq H_A(t_n)$ and $t_n \gg \beta^{-2/3}$, then
\begin{equation}\label{explowx}
\int_{-\infty}^{L_A} p_{t_n}^{L_A}(x,y) e^{\rho y} \: dy \ll e^{\rho x} e^{-\beta^2 t_n^3/73}.
\end{equation}
If $x < L_A$ and $t_n \gg \beta^{-2/3}$, then
\begin{equation}\label{explowy}
\int_{-\infty}^{H_A(t_n)} p_{t_n}^{L_A}(x,y) e^{\rho y} \: dy \ll e^{\rho x} e^{-\beta^2 t_n^3/73}.
\end{equation}
If $x \leq H_A(t_n)$, $0 \leq \zeta \leq \beta t_n/2$, and $\beta^{-2/3} \ll t_n \ll \rho/\beta$, then
\begin{equation}\label{explowx2}
\int_{-\infty}^{L_A} p_{t_n}^{L_A}(x,y) e^{(\rho - \zeta) y} \: dy \ll e^{\rho x} e^{-\beta^2 t_n^3/73}.
\end{equation}
If $x < L_A$, $0 \leq \zeta \leq \beta t_n/2$, and $\beta^{-2/3} \ll t_n \ll \rho/\beta$, then
\begin{equation}\label{explowy2}
\int_{-\infty}^{H_A(t_n)} p_{t_n}^{L_A}(x,y) e^{(\rho - \zeta) y} \: dy \ll e^{\rho x} e^{-\beta^2 t_n^3/73}.
\end{equation}
\end{Lemma}

\begin{proof}
Suppose $x \leq L_A - \frac{1}{9} \beta t_n^2$.  By (\ref{ptxy}) and the fact that $p_{t_n}^{L_A}(x,y) \leq p_{t_n}(x,y)$, we have
\begin{align*}
\int_{-\infty}^{L_A} p_{t_n}^{L_A}(x,y) e^{\rho y} \: dy &\leq \int_{-\infty}^{L_A} \frac{1}{\sqrt{2 \pi t_n}} \exp \bigg(\rho x - \frac{(y - x)^2}{2t} - \frac{\rho^2 t_n}{2} + \frac{\beta(y + x)t_n}{2} + \frac{\beta^2 t_n^3}{24} \bigg) \: dy \\
&\leq \exp \bigg( \rho x - \frac{\rho^2 t_n}{2} + \beta L_A t_n - \frac{\beta^2 t_n^3}{18} + \frac{\beta^2 t_n^3}{24} \bigg) \int_{-\infty}^{\infty} \frac{1}{\sqrt{2 \pi t_n}} e^{-(y - x)^2/2t_n} \: dy \\
&= \exp \bigg(\rho x - 2^{-1/3} \gamma_1 \beta^{2/3} t_n - \frac{A \beta t_n}{\rho} - \frac{\beta^2 t_n^3}{72} \bigg).
\end{align*} 
Because $\beta^2 t_n^3 \gg \beta^{2/3} t_n$ and $\beta^2 t_n^3 \gg \beta t_n/\rho$, it follows that (\ref{explowx}) holds.

To establish (\ref{explowy}), we use the same argument.  Instead of having $x \leq H_A(t_n)$ and $y \leq L_A$, we now have $x \leq L_A$ and $y \leq H_A(t_n)$.  However, the resulting bound on $\beta(y + x)t_n/2$ is the same, and the rest of the calculation proceeds identically.

To prove (\ref{explowx2}) and (\ref{explowy2}), it follows from (\ref{explowx}) and (\ref{explowy}) that we only need to consider the portion of the integral when $y < 0$.  Using (\ref{ptxy}) along with the fact that $\zeta \leq \beta t_n/2$ and therefore $\frac{1}{2}\beta y t_n - \zeta y \le 0$ when $y \le 0$, we have
\begin{align*}
&\int_{-\infty}^0 p_{t_n}^{L_A}(x,y) e^{(\rho - \zeta)y} \: dy \\
&\hspace{.5in}\leq  \int_{-\infty}^0 \frac{1}{\sqrt{2 \pi t_n}} \exp \bigg(-\zeta y + \rho x - \frac{(y - x)^2}{2t_n} - \frac{\rho^2 t_n}{2} + \frac{\beta(y + x)t_n}{2} + \frac{\beta^2 t_n^3}{24} \bigg) \: dy \\
&\hspace{.5in}\leq \exp \bigg(\rho x - \frac{\rho^2 t_n}{2} + \frac{\beta L_A t_n}{2} + \frac{\beta^2 t_n^3}{24} \bigg) \int_{-\infty}^0 \frac{1}{\sqrt{2 \pi t_n}} e^{-(y - x)^2/2t_n} \: dy \\
&\hspace{.5in}\leq \exp \bigg( \rho x - \frac{\rho^2 t_n}{4} - 2^{-4/3} \gamma_1 \beta^{2/3} t_n - \frac{A \beta t_n}{2 \rho} + \frac{\beta^2 t_n^3}{24} \bigg).
\end{align*}
Because $t_n \ll \rho/\beta$, we have $\rho^2 t_n \gg \beta^2 t_n^3$, which is sufficient to conclude (\ref{explowx2}) and (\ref{explowy2}).
\end{proof}

\begin{proof}[Proof of Proposition \ref{rtconfigprop}]
Let $g: \R \rightarrow [0, \infty)$ be a bounded measurable function, and let $$\|g\| = \sup_{x \in \Rsm} |g(x)|.$$  On the event that $N_n(t) \geq 1$, let
\begin{equation}\label{Phidef}
\Phi_n(g) = Y_n(t_n) \int_{-\infty}^{\infty} g(x) \:\xi_n(t_n)(dx) = \sum_{i=1}^{N_n(t_n)} e^{\rho X_{i,n}(t_n)} g\big((2 \beta)^{1/3}(L - X_{i,n}(t_n)\big).
\end{equation}
Otherwise, let $\Phi_n(g) = 0$.
Let $\Phi_n(1)$ be the value of $\Phi_n(g)$ when $g(x) = 1$ for all $x$.  Then, when $N_n(t) \geq 1$, we have $$\int_{-\infty}^{\infty} g(x) \: \xi_n(t_n)(dx) = \frac{\Phi_n(g)}{\Phi_n(1)}.$$
We will show that for all $\kappa > 0$, we have
\begin{align}\label{Pphi}
&\lim_{n \rightarrow \infty} \P \bigg( \frac{1 - \kappa}{(Ai'(\gamma_1))^2} \bigg( \int_0^{\infty} g(z) Ai(\gamma_1 + z) \: dz \bigg) Z_n(0) < \Phi_n(g) \nonumber \\
&\hspace{1.5in}< \frac{1 + \kappa}{(Ai'(\gamma_1))^2} \bigg( \int_0^{\infty} g(z) Ai(\gamma_1 + z) \: dz \bigg) Z_n(0)  \bigg) = 1.
\end{align}
It will then follow that for all $\kappa > 0$, we have
$$\lim_{n \rightarrow \infty} \P \bigg( \frac{1 - \kappa}{1 + \kappa} \int_0^{\infty} g(z)h(z) \: dz \leq \int_{-\infty}^{\infty} g(x) \: \xi_n(t_n)(dx) \leq \frac{1 + \kappa}{1 - \kappa} \int_0^{\infty} g(z)h(z) \: dz \bigg) = 1.$$
That is, as $n \rightarrow \infty$ we have $$\int_{-\infty}^{\infty} g(x) \: \xi_n(t_n)(dx) \rightarrow_p \int_0^{\infty} g(x) \: \nu(dx),$$ which by, for example, Theorem 16.16 of \cite{kall} will imply the statement of the proposition.

It therefore remains to prove (\ref{Pphi}).  We will estimate $\Phi_n(g)$ by dividing it into seven pieces, depending mostly on the location of the particle at time $t_n$ and the location of the ancestral particle at time zero.  For $i \in \{1, \dots, N_n(t)\}$ and $s \in [0, t]$, let $a_{i,n}(s,t)$ be the location at time $s$ of the ancestor of the $i$th particle at time $t$.  We partition the particles at time $t_n$ into the following seven subsets:
\begin{align*}
S_{1,n} &= \big\{i: a_{i,n}(s, t_n) \geq L \mbox{ for some }s \in \big[0, t_n - 2C_1 \rho^{-2}\big]\big\}, \\
S_{2,n} &= \big\{i \notin S_{1,n}: a_{i,n}(s,t_n) \geq L \mbox{ for some }s \in \big[t_n - 2C_1 \rho^{-2}, t_n\big] \big\}, \\
S_{3,n} &= \big\{i \notin S_{1,n} \cup S_{2,n}: a_{i,n}(0,t_n) < H(t_n) \big\}, \\
S_{4,n} &= \big\{i \notin S_{1,n} \cup S_{2,n} \cup S_{3,n}:  X_{i,n}(t_n) < H(t_n) \big\}, \\
S_{5,n} &= \big\{i \notin S_{1,n} \cup S_{2,n} \cup S_{3,n} \cup S_{4,n}: H(t_n) \leq a_{i,n}(0,t_n) \leq K(t_n) \big\}, \\
S_{6,n} &= \big\{i \notin S_{1,n} \cup S_{2,n} \cup S_{3,n} \cup S_{4,n} \cup S_{5,n}: H(t_n) \leq X_{i,n}(t_n) \leq K(t_n) \big\}, \\
S_{7,n} &= \big\{i \notin S_{1,n} \cup S_{2,n}: K(t_n) < a_{i,n}(0, t_n) < L \mbox{ and } K(t_n) < X_{i,n}(t_n) < L \big\}.
\end{align*}
For $j \in \{1, \dots, 7\}$, let $$\Phi_{j,n}(g) = \sum_{i \in S_{j,n}} e^{\rho X_{i,n}(t_n)} g\big((2 \beta)^{1/3}(L - X_{i,n}(t_n)\big),$$ and note that $\Phi_n(g) = \Phi_{1,n}(g) + \dots + \Phi_{7,n}(g)$.  We will show that the first six terms contribute little to the sum, while the seventh is highly concentrated around its mean.

The first term $\Phi_{1,n}(g)$ accounts for the contributions of particles that reach $L$ before time $t_n - 2C_1 \rho^{-2}$.  By Lemma \ref{aboveL}, with probability tending to one as $n \rightarrow \infty$, no particles above $L$ at time zero will have descendants alive at time $t_n$.
Consider the process in which particles are killed upon reaching $L$, and let $R_n(s,t)$ be the number of particles killed between time $s$ and time $t$.  By Lemma \ref{Rxlem},
$$\E[R_n(0,t_n - 2C_1 \rho^{-2})|{\cal F}_0] \lesssim e^{-\rho L}\big(Y_n(0) + t_n \beta^{2/3} Z_n(0)\big).$$
Therefore, using the assumptions (\ref{Zasm}) and (\ref{Yasm}) and the fact that $t_n \ll \rho/\beta$, we obtain that as $n \rightarrow \infty$,
$$\rho^2 \E[R_n(0, t_n - 2C_1 \rho^{-2})|{\cal F}_0] \lesssim \rho^2 e^{-\rho L} Y_n(0) + \rho^2 e^{-\rho L} \beta^{2/3} t_n Z_n(0) \rightarrow_p 0.$$  In view of Lemma \ref{aboveL}, it follows that with probability tending to one as $n \rightarrow \infty$, no particle that hits $L$ before time $t_n - 2C_1 \rho^{-2}$ has descendants alive in the population at time $t_n$.  That is, we have
\begin{equation}\label{Pphi1}
\lim_{n \rightarrow \infty} \P(\Phi_{1,n}(g) = 0) = 1.
\end{equation}

The term $\Phi_{2,n}(g)$ accounts for particles that reach $L$ between times $t_n - 2C_1 \rho^{-2}$ and $t_n$, which means they may have descendants surviving at time $t_n$ even if they will not have descendants surviving for a long time.  We again apply Lemma \ref{Rxlem}. Noting that $\beta^2 t_n^3 - \beta^2 (t_n - 2C_1 \rho^{-2})^3 \rightarrow 0$ because $t_n \ll \rho/\beta$, we get
$$\E[R_n(t_n - 2C_1 \rho^{-2}, t_n)|{\cal F}_0] \lesssim e^{-\rho L}\big(Y_n(0) e^{-\beta^2 t_n^3/9} + \rho^{-2} \beta^{2/3} Z_n(0)\big).$$  Therefore, by Lemma \ref{YfromL},
$$\E[\Phi_{2,n}(g)|{\cal F}_0] \lesssim \|g\| \big(Y_n(0) e^{-\beta^2 t_n^3/9} + \rho^{-2} \beta^{2/3} Z_n(0)\big).$$  It follows that
\begin{equation}\label{Ephi2}
\E\bigg[ \frac{\rho^3 e^{-\rho L}}{\beta^{1/3}} \cdot \Phi_{2,n}(g)\Big|{\cal F}_0\bigg] \lesssim \frac{\rho^3 e^{-\rho L}}{\beta^{1/3}} Y_n(0) e^{-\beta^2 t_n^3/9} + \frac{\rho^3 e^{-\rho L}}{\beta^{1/3}} \cdot \frac{\beta^{2/3}}{\rho^2} Z_n(0).
\end{equation}
Because (\ref{A1}) implies that $\beta^{2/3}/\rho^2 \rightarrow 0$ and (\ref{tcond1}) implies that
\begin{equation}\label{2nbound}
\frac{\rho}{\beta^{1/3}} e^{-\beta^2 t_n^3/9} \rightarrow 0,
\end{equation}
it follows from (\ref{Zasm}), (\ref{Yasm}), and (\ref{Ephi2}) that as $n \rightarrow \infty$,
$$\E\bigg[ \frac{\rho^3 e^{-\rho L}}{\beta^{1/3}} \cdot \Phi_{2,n}(g)\Big|{\cal F}_0\bigg] \rightarrow_p 0$$ and therefore, by the conditional Markov's inequality,
\begin{equation}\label{Pphi2}
\frac{\rho^3 e^{-\rho L}}{\beta^{1/3}} \cdot \Phi_{2,n}(g) \rightarrow_p 0.
\end{equation}

We now consider $\Phi_{3,n}(g)$ and $\Phi_{4,n}(g)$.  By (\ref{explowx}),
\begin{equation}\label{Ephi3}
\E\bigg[ \frac{\rho^3 e^{-\rho L}}{\beta^{1/3}} \cdot \Phi_{3,n}(g)\Big|{\cal F}_0\bigg] \ll \frac{\rho^3 e^{-\rho L}}{\beta^{1/3}} \|g\| Y_n(0) e^{-\beta^2 t_n^3/73},
\end{equation}
and we obtain the identical result for $\Phi_{4,n}(g)$ using (\ref{explowy}) in place of (\ref{explowx}).  Therefore, applying the conditional Markov's inequality as we did for $\Phi_{2,n}(g)$, we get
\begin{equation}\label{Pphi34}
\frac{\rho^3 e^{-\rho L}}{\beta^{1/3}} \cdot (\Phi_{3,n}(g) + \Phi_{4,n}(g)) \rightarrow_p 0.
\end{equation}

To handle the remaining terms, write
\begin{align*}
Z_n'(0) &= \sum_{i=1}^{N_n(0)} e^{\rho X_{i,n}(0)} \alpha(L - X_{i,n}(0)) \1_{\{X_{i,n}(0) \leq K(t_n) \}}, \\
Z_n^*(0) &= \sum_{i=1}^{N_n(0)} e^{\rho X_{i,n}(0)} \alpha(L - X_{i,n}(0)) \1_{\{K(t_n) < X_{i,n}(0) < L\}} = Z_n(0) - Z_n'(0).
\end{align*}
If $x \leq K(t_n)$, then $(2 \beta)^{1/3}(L - x) + \gamma_1 \geq \frac{1}{66} \cdot 2^{1/3} \beta^{4/3} t^2 + \gamma_1.$
Therefore, by (\ref{Airyasymp}), there exists a positive constant $C_6$ such that if $x \leq K(t_n)$, then
$$\alpha(L - x) \leq e^{-C_6 \beta^2 t_n^3}.$$ It follows that $Z'_n(0) \leq Y_n(0) e^{-C_6 \beta^2 t_n^3},$ so by (\ref{Yasm}), (\ref{A1}), and the reasoning that led to (\ref{2nbound}), we have
\begin{equation}\label{Znprime}
\frac{\rho^3}{\beta^{1/3}} e^{-\rho L} Z_n'(0) \rightarrow_p 0
\end{equation}
as $n \rightarrow \infty$.  Also, if $H(t_n) \leq x < L$ and $H(t_n) \leq y < L$, then because (\ref{dcondbig}) holds, we
can estimate $p_{t_n}^{L}(x,y)$ using Lemma \ref{dapprox4}.  It follows from Lemma \ref{dapprox4} and the boundedness of $g$ that
\begin{align*}
\E[\Phi_{5,n}(g)|{\cal F}_0] &\lesssim \beta^{1/3} \bigg( \int_{H(t_n)}^{L} e^{\rho y} \cdot e^{-\rho y} \alpha(L - y) \: dy \bigg) Z_n'(0) \\
&\leq \beta^{1/3} \bigg( \int_{H(t_n)}^{L} Ai \big( (2 \beta)^{1/3}(L - y) + \gamma_1 \big) \: dy \bigg) Z_n'(0) \\
&\lesssim Z_n'(0).
\end{align*}
Combining this result with (\ref{Znprime}) and the conditional Markov's Inequality, we get
\begin{equation}\label{Pphi5}
\frac{\rho^3 e^{-\rho L}}{\beta^{1/3}} \cdot \Phi_{5,n}(g) \rightarrow_p 0.
\end{equation}

We can bound $\Phi_{6,n}(g)$ by making a similar calculation.  This time, we are considering descendants of the initial particles that contribute to $Z_n^*(0)$ rather than $Z_n'(0)$, and requiring the particles to end up between $H(t_n)$ and $K(t_n)$ at time $t_n$.  Therefore, making the substitution $z = (2 \beta)^{1/3}(L - y)$, we get
\begin{align}\label{phi6int}
\E[\Phi_{6,n}(g)|{\cal F}_0] &\lesssim \beta^{1/3} \bigg( \int_{H(t_n)}^{K(t_n)} Ai \big( (2 \beta)^{1/3}(L - y) + \gamma_1 \big) \: dy \bigg) Z_n^*(0) \nonumber \\
&\lesssim \bigg( \int_{(2 \beta)^{1/3} \cdot \frac{1}{66} \beta t_n^2}^{(2 \beta)^{1/3} \cdot \frac{1}{9} \beta t_n^2} Ai(\gamma_1 + z) \: dz \bigg) Z_n^*(0).
\end{align}
Because $\beta^{4/3} t_n^2 \rightarrow \infty$, the integral in (\ref{phi6int}) tends to zero as $n \rightarrow \infty$.  Therefore, in view of (\ref{Zasm}),
\begin{equation}\label{Pphi6}
\frac{\rho^3 e^{-\rho L}}{\beta^{1/3}} \cdot \Phi_{6,n}(g) \rightarrow_p 0.
\end{equation}

Finally, we consider the term $\Phi_{7,n}(g)$.  Using Lemma \ref{dapprox4} in the first step, making the substitution $z = (2 \beta)^{1/3}(L - y)$ in the second step, and using that $\beta^{4/3} t_n^2 \rightarrow \infty$ in the third step, we obtain
\begin{align*}
\E[\Phi_{7,n}(g)|{\cal F}_0] &= \frac{(2 \beta)^{1/3}}{(Ai'(\gamma_1))^2} \bigg( \int_{K(t_n)}^{L} g \big( (2 \beta)^{1/3}(L - y) \big)\alpha(L-y)\: dy\bigg) Z_n^*(0)(1+o(1)) \\
&= \frac{1}{(Ai'(\gamma_1))^2} \bigg( \int_0^{(2 \beta)^{1/3} \cdot \frac{1}{66} \beta t_n^2} g(z) Ai(\gamma_1 + z) \: dz \bigg) Z_n^*(0)(1 + o(1)) \\
&= \frac{1}{(Ai'(\gamma_1))^2} \bigg( \int_0^{\infty} g(z) Ai(\gamma_1 + z) \: dz \bigg) Z_n^*(0)(1 + o(1)).
\end{align*}
Therefore, in view of (\ref{Zasm}), we have that for all $\eta > 0$,
\begin{equation}\label{EPhi7}
\lim_{n \rightarrow \infty} \P \bigg( \bigg| \E[\Phi_{7,n}(g)|{\cal F}_0] - \frac{1}{(Ai'(\gamma_1))^2} \bigg( \int_0^{\infty} g(z) Ai(\gamma_1 + z) \: dz \bigg) Z_n^*(0) \bigg| > \frac{\eta \beta^{1/3} e^{\rho L}}{\rho^3} \bigg) = 0.
\end{equation}
Moreover, using the independence of the descendants of different particles along with Lemma~\ref{2momprop}, we get
\begin{align*}
\mbox{Var}(\Phi_{7,n}(g)|{\cal F}_0) &= \sum_{i=1}^{N_n(0)} \mbox{Var}_{X_{i,n}(0)}(\Phi_{7,n}(g)) \1_{\{K(t_n) < X_{i,n}(0) < L\}} \\
&\leq \sum_{i=1}^{N_n(0)} \E_{X_{i,n}(0)}[\Phi_{7,n}(g)^2] \1_{\{K(t_n) < X_{i,n}(0) < L\}} \\ 
&\lesssim \frac{\beta^{2/3} e^{\rho L}}{\rho^4} \big( Y_n(0) + \beta^{2/3} t_n Z_n^*(0) \big).
\end{align*}
Therefore, by the conditional Chebyshev's Inequality, for all $\eta > 0$ we have
\begin{align*}
\P\bigg( | \Phi_{7,n}(g) - \E[\Phi_{7,n}(g)|{\cal F}_0] | > \frac{\eta \beta^{1/3} e^{\rho L}}{\rho^3} \bigg|{\cal F}_0 \bigg) &\leq \frac{\rho^2 e^{-\rho L}}{\eta^2} \big( Y_n(0) + \beta^{2/3} t_n Z_n^*(0) \big).
\end{align*}
The first term on the right-hand side converges in probability to zero by (\ref{Yasm}), and because $t_n \ll \rho/\beta$, the second term on the right-hand side converges in probability to zero by (\ref{Zasm}).  It follows that
\begin{equation}\label{VPhi7}
\lim_{n \rightarrow \infty} \P\bigg( | \Phi_{7,n}(g) - \E[\Phi_{7,n}(g)|{\cal F}_0] | > \frac{\eta \beta^{1/3} e^{\rho L}}{\rho^3} \bigg) = 0.
\end{equation}
Combining (\ref{Znprime}), (\ref{EPhi7}) and (\ref{VPhi7}) gives that for all $\eta > 0$,
\begin{equation}\label{Pphi7}
\lim_{n \rightarrow \infty} \P \bigg( \bigg| \Phi_{7,n}(g) - \frac{1}{(Ai'(\gamma_1))^2} \bigg( \int_0^{\infty} g(z) Ai(\gamma_1 + z) \: dz \bigg) Z_n(0) \bigg| > \frac{3 \eta \beta^{1/3} e^{\rho L}}{\rho^3} \bigg) = 0.
\end{equation}
Finally, combining (\ref{Pphi1}), (\ref{Pphi2}), (\ref{Pphi34}), (\ref{Pphi5}), (\ref{Pphi6}), and (\ref{Pphi7}), we get that for all $\eta > 0$,
\begin{equation}\label{mainphi}
\lim_{n \rightarrow \infty} \P \bigg( \bigg| \Phi_n(g) - \frac{1}{(Ai'(\gamma_1))^2} \bigg( \int_0^{\infty} g(z) Ai(\gamma_1 + z) \: dz \bigg) Z_n(0) \bigg| > \frac{4 \eta \beta^{1/3} e^{\rho L}}{\rho^3} \bigg) = 0.
\end{equation}
The result (\ref{Pphi}), and therefore the statement of the proposition, now follows from (\ref{Zasm}).
\end{proof}

\section{Proof of Proposition \ref{gaussianprop}}\label{gausssec}

In this section, we prove Proposition \ref{gaussianprop}, which shows that when (\ref{Zasm}) and (\ref{Yasm}) hold, the empirical distribution of particles at time approximately $\rho/\beta$ is asymptotically Gaussian.  We begin by proving the following simple lemma concerning the Airy function.

\begin{Lemma}\label{airyr3}
We have
$$\lim_{r \rightarrow \infty} e^{-r^3/3} \int_0^{\infty} e^{r(\gamma_1 + z)} Ai(\gamma_1 + z) \: dz = 1.$$
\end{Lemma}

\begin{proof}
Although this proof is elementary it does require considering various cases. For a lower bound, begin by noting that for $r$ large,
\begin{equation}\label{airyr3lb}
\int_0^{\infty} e^{r(\gamma_1 + z)} Ai(\gamma_1 + z) \: dz = \int_{\gamma_1}^{\infty} e^{rz} Ai(z) \: dz \ge \int_{r^2-r^{2/3}}^{r^2+r^{2/3}} e^{rz} Ai(z) \: dz.
\end{equation}
Recall from \eqref{Airyasymp} that as $x\to\infty$,
\[Ai(x) \sim \frac{1}{2 \sqrt{\pi} x^{1/4}} e^{-(2/3) x^{3/2}} \hspace{.2in}\mbox{as }x \rightarrow \infty,\]
and therefore
\begin{align*}
\int_{r^2-r^{2/3}}^{r^2+r^{2/3}} e^{rz} Ai(z) \: dz &\sim \frac{1}{2\sqrt\pi} \int_{r^2-r^{2/3}}^{r^2+r^{2/3}} e^{rz}\frac{1}{z^{1/4}} e^{-2z^{3/2}/3} \: dz\\
&= \frac{1}{2\sqrt\pi} \int_{-r^{2/3}}^{r^{2/3}} \frac{1}{(r^2+y)^{1/4}}e^{r^3 + ry - 2(r^2+y)^{3/2}/3} \: dy.
\end{align*}
Now, for $y\in[-r^{2/3},r^{2/3}]$, we have $(r^2+y)^{1/4}\sim r^{1/2}$ and $(r^2+y)^{3/2} = r^3+3ry/2 + 3y^2/8r + o(1)$. Substituting these estimates into the above, we have
\[\int_{r^2-r^{2/3}}^{r^2+r^{2/3}} e^{rz} Ai(z) \: dz \sim \frac{1}{2\sqrt\pi} \int_{-r^{2/3}}^{r^{2/3}} \frac{1}{r^{1/2}}e^{r^3/3 - y^2/4r} \: dy = e^{r^3/3}\int_{-r^{2/3}}^{r^{2/3}} \frac{1}{\sqrt{4\pi r}} e^{-y^2/4r} \: dy.\]
The last integral is easily recognised as the probability that a Gaussian random variable of mean $0$ and variance $2r$ falls between $-r^{2/3}$ and $r^{2/3}$, which converges to $1$ as $r\to\infty$. Thus we have shown that
\begin{equation}\label{airyr3asymp}
\int_{r^2-r^{2/3}}^{r^2+r^{2/3}} e^{rz} Ai(z) \: dz \sim e^{r^3/3}
\end{equation}
and, combining this with \eqref{airyr3lb}, we see that the lower bound claimed in the statement holds.

For the upper bound, by \eqref{airyr3asymp}, it suffices to show that
\begin{equation}\label{airyr3ubcond}
\int_{\gamma_1}^{r^2-r^{2/3}} e^{rz} Ai(z) \: dz + \int_{r^2+r^{2/3}}^\infty e^{rz} Ai(z) \: dz \ll e^{r^3/3}.
\end{equation}
Since $Ai(z)$ is bounded (see \eqref{upairy}), we have
\[\int_{\gamma_1}^r e^{rz} Ai(z) \: dz \lesssim re^{r^2} \ll e^{r^3/3}.\]
When $z> r$, we can again use \eqref{Airyasymp} to obtain that
\[\int_r^{r^2-r^{2/3}} e^{rz} Ai(z) \: dz \sim \frac{1}{2\sqrt\pi}\int_r^{r^2-r^{2/3}} e^{rz}\frac{1}{z^{1/4}} e^{-2z^{3/2}/3} \: dz = \frac{1}{2\sqrt\pi}\int_{r-r^2}^{-r^{2/3}} \frac{e^{r^3+ry-2(r^2+y)^{3/2}/3}}{(r^2+y)^{1/4}} \: dy.\]
We now use the fact that
\[(1+x)^{3/2} \ge 1 + 3x/2 + x^2/4 \hspace{4mm} \text{ for all } x\in(-1,4)\]
and therefore
\begin{equation}\label{airyr3ub}
(r^2+y)^{3/2} \ge r^3 + 3ry/2 + y^2/4r \hspace{4mm} \text{ for all } y\in(-r^2,4r^2)
\end{equation}
to see that
\[\frac{1}{2\sqrt\pi}\int_{r-r^2}^{-r^{2/3}} \frac{e^{r^3+ry-2(r^2+y)^{3/2}/3}}{(r^2+y)^{1/4}} \: dy \le \frac{1}{2\sqrt\pi}\int_{r-r^2}^{-r^{2/3}} \frac{e^{r^3/3-y^2/6r}}{(r^2+y)^{1/4}} \: dy \lesssim r^2 e^{r^3/3 - r^{1/3}/6} \ll e^{r^3/3}.\]
When $z\in [r^2+r^{2/3},4r^2]$, we follow a very similar route: using \eqref{Airyasymp} and \eqref{airyr3ub} in exactly the same way as above, we have
\[\int_{r^2+r^{2/3}}^{4r^2} e^{rz} Ai(z) \: dz \sim \frac{1}{2\sqrt\pi}\int_{r^2+r^{2/3}}^{4r^2} e^{rz}\frac{1}{z^{1/4}} e^{-2z^{3/2}/3} \: dz \le \frac{1}{2\sqrt\pi}\int_{r^{2/3}}^{4r^2} \frac{e^{r^3/3-y^2/6r}}{(r^2+y)^{1/4}} \: dy \ll e^{r^3/3}.\]
Finally, for $z>4r^2$, again by \eqref{Airyasymp} we have
\[\int_{4r^2}^\infty e^{rz} Ai(z) \: dz \sim \frac{1}{2\sqrt\pi}\int_{4r^2}^\infty e^{rz}\frac{1}{z^{1/4}} e^{-2z^{3/2}/3} \: dz\]
and since for $z>4r^2$ we have $z^{3/2} > 2rz$, this is at most
\[\int_{4r^2}^\infty e^{-rz/3} dz \ll e^{r^3/3}.\]
Combining our four estimates on the integrals over the regions $z\in[\gamma_1,r]$, $z\in(r,r^2-r^{2/3}]$, $z\in[r^2+r^{2/3},4r^2]$ and $z>4r^2$, we obtain \eqref{airyr3ubcond} and therefore the proof is complete.
\end{proof}

Let $\eta > 0$.  By Lemma \ref{airyr3}, we can choose $C_7$ to be sufficiently large that
\begin{equation}\label{C1choice}
\bigg(1 - \frac{\eta}{2} \bigg) e^{C_7^3/6} \leq \int_0^{\infty} e^{2^{-1/3} C_7 (\gamma_1 + y)} Ai(\gamma_1 + y) \: dy \leq (1 + \eta) e^{C_7^3/6}
\end{equation}
holds, and also
\begin{equation}\label{C1choice2}
e^{-C_7^3/6} \cdot \frac{(1 + \eta) \sqrt{2 \pi}}{(Ai'(\gamma_1))^2} \int_0^{\infty} Ai(\gamma_1 + y) \: dy < \eta.
\end{equation}
We can then choose $C_8$ to be sufficiently large that
\begin{equation}\label{C2choice}
\int_{2^{1/3}C_8}^{\infty} e^{2^{-1/3} C_7 (\gamma_1 + y)} Ai(\gamma_1 + y) \: dy < \frac{\eta}{2} e^{C_7^3/6}.
\end{equation}
By (\ref{C1choice}) and (\ref{C2choice}), we have
\begin{equation}\label{C1C2}
(1 - \eta) e^{C_7^3/6} \leq \int_0^{2^{1/3} C_8} e^{2^{-1/3} C_7 (\gamma_1 + y)} Ai(\gamma_1 + y) \: dy \leq (1 + \eta) e^{C_7^3/6}.
\end{equation}

We now establish the following lemma, which shows that a certain functional of the process after a short time $t_n$ which satisfies (\ref{tcond1}) is well approximated by a constant mutiple of $Z_n(0)$.  To prove this result, we use the results established in the previous section for the configuration of particles near the right edge.

\begin{Lemma}\label{rhobetas}
Suppose (\ref{Zasm}), (\ref{Yasm}), and (\ref{tcond1}) hold.  Let $\eta > 0$, and choose positive constants $C_7$ and $C_8$ as above so that (\ref{C1C2}) holds.  Let
$$\Upsilon_n = \exp \bigg( \frac{C_7 \rho^2}{2 \beta^{2/3}} - \frac{C_7^3}{6} \bigg) \sum_{i=0}^{N_n(t_n)} e^{(\rho - C_7 \beta^{1/3}) X_{i,n}(t_n)} \1_{\{L - C_8 \beta^{-1/3} < X_{i,n}(t_n) < L\}}.$$
Then
$$\lim_{n \rightarrow \infty} \P \bigg( \frac{1 - 2\eta}{(Ai'(\gamma_1))^2} Z_n(0) \leq \Upsilon_n \leq  \frac{1 + 2\eta}{(Ai'(\gamma_1))^2} Z_n(0) \bigg) = 1.$$
\end{Lemma}

\begin{proof}
Define the function $g$ by $g(y) = e^{2^{-1/3} C_7 y}$ if $0 < y < 2^{1/3} C_8$ and $g(y) = 0$ otherwise.  Define $\Phi_n(g)$ as in (\ref{Phidef}).  Then
\begin{align}\label{sumphi}
&\sum_{i=0}^{N_n(t_n)} e^{(\rho - C_7 \beta^{1/3}) X_{i,n}(t_n)} \1_{\{L - C_8 \beta^{-1/3} < X_{i,n}(t_n) < L\}} \nonumber \\
&\hspace{.5in}= e^{-C_7 \beta^{1/3} L} \sum_{i=0}^{N_n(t_n)} e^{\rho X_{i,n}(t_n)} e^{C_7 \beta^{1/3} (L - X_{i,n}(t_n))} \1_{\{L - C_8 \beta^{-1/3} < X_{i,n}(t_n) < L\}} \nonumber \\
&\hspace{.5in}= e^{-C_7 \rho^2 \beta^{-2/3}/2} e^{2^{-1/3}C_7 \gamma_1} \sum_{i=0}^{N_n(t_n)} e^{\rho X_{i,n}(t_n)} g((2 \beta)^{1/3}(L - X_{i,n}(t_n))) \nonumber \\
&\hspace{.5in}= e^{-C_7 \rho^2 \beta^{-2/3}/2} e^{2^{-1/3}C_7 \gamma_1} \Phi_n(g).
\end{align}
By (\ref{Pphi}), as $n \rightarrow \infty$, we have
\begin{equation}\label{phiz}
\frac{\Phi_n(g)}{Z_n(0)} \rightarrow_p \frac{1}{(Ai'(\gamma_1))^2} \int_0^{2^{1/3} C_8} e^{2^{-1/3} C_7 y} Ai(\gamma_1 + y) \: dy.
\end{equation}
It follows from (\ref{sumphi}) and (\ref{phiz}) that
\begin{align}\label{ratioint}
&\frac{e^{C_7 \rho^2 \beta^{-2/3}/2}}{Z_n(0)} \sum_{i=0}^{N_n(t_n)} e^{(\rho - C_7 \beta^{1/3}) X_{i,n}(t_n)} \1_{\{L - C_8 \beta^{-1/3} < X_{i,n}(t_n) < L\}} \nonumber \\
&\hspace{1.5in}\rightarrow_p \frac{1}{(Ai'(\gamma_1))^2} \int_0^{2^{1/3} C_8} e^{2^{-1/3} C_7 (\gamma_1 + y)} Ai(\gamma_1 + y) \: dy.
\end{align}
The result follows from (\ref{ratioint}) and (\ref{C1C2}).
\end{proof}

\begin{proof}[Proof of Proposition \ref{gaussianprop}]
Let $g: \R \rightarrow [0, \infty)$ be a nonzero bounded continuous function, and let $$\|g\| = \sup_{x \in \Rsm} |g(x)|.$$  
Write $$I(g) = \int_{-\infty}^{\infty} g(x) \mu(dx) = \int_{-\infty}^{\infty} g(y) \cdot \frac{1}{\sqrt{2 \pi}} e^{-y^2/2} \: dy.$$  Let
$$\Psi_n(g) = \sum_{i=1}^{N_n(t_n)} g \Big( X_{i,n}(t_n) \sqrt{\beta/\rho} \Big)$$
We will show that for all $\kappa > 0$, we have
\begin{equation}\label{mainPhin}
\lim_{n \rightarrow \infty} \P \bigg( \frac{1 - \kappa}{(Ai'(\gamma_1))^2} e^{-\rho^3/3\beta} Z_n(0) I(g) \leq \Psi_n(g) \leq \frac{1 + \kappa}{(Ai'(\gamma_1))^2} e^{-\rho^3/3\beta} Z_n(0) I(g) \bigg) = 1.
\end{equation}
Because $\Psi_n(1) = N_n(t_n)$, the result (\ref{mainNn}) will follow immediately.  Also, if $N_n(t) \geq 1$, then
$$\int_{-\infty}^{\infty} g(x) \: \zeta_n(t_n)(dx) = \frac{\Psi_n(g)}{\Psi_n(1)}.$$  Because $I(1) = 1$, it will follow from (\ref{mainPhin}) that as $n \rightarrow \infty$, we have
$$\int_{-\infty}^{\infty} g(x) \: \zeta_n(t_n)(dx) \rightarrow_p \int_{-\infty}^{\infty} g(x) \: \mu(dx),$$ 
which by, for example, Theorem 16.16 of \cite{kall} is enough to imply that $\zeta_n(t_n) \Rightarrow \mu$.

It remains, then, to prove (\ref{mainPhin}).
Let $\eta > 0$, and recall the definitions of the positive constants $C_7$ and $C_8$ from before the statement of Lemma \ref{rhobetas}.  Let $C_9$ be a positive constant chosen large enough that if $Z$ has a standard normal distribution, then
\begin{equation}\label{C3cond}
P(|Z| > C_9) < \eta.
\end{equation}
Let $s_n = C_7 \beta^{-2/3}$, and let $$u_n = t_n - \frac{\rho}{\beta} + s_n.$$  Because the times $t_n$ satisfy (\ref{tcond2}), we have
\begin{equation}\label{unbound}
\beta^{-2/3} \bigg( \log \bigg(\frac{\rho}{\beta^{1/3}} \bigg) \bigg)^{1/3} \ll \frac{\rho^{2/3}}{\beta^{8/9}} \ll u_n \ll \frac{\rho}{\beta},
\end{equation}
and therefore the configuration of particles at time $u_n$ satisfies the conclusions of Proposition~\ref{rtconfigprop}.

To prove (\ref{mainPhin}), we will follow the trajectories of the particles between times $u_n$ and $t_n$.  Recalling (\ref{HAdef}), we first
partition the particles at time $u_n$ into the following four subsets:
\begin{align*}
G_{1,n} &= \{i: X_{i,n}(u_n) \geq L\}, \\
G_{2,n} &= \{i: X_{i,n}(u_n) \leq H(u_n) \}, \\
G_{3,n} &= \big\{i: H(u_n) < X_{i,n}(u_n) \leq L - C_8 \beta^{-1/3} \big\}, \\
G_{4,n} &= \big\{i: L - C_8 \beta^{-1/3} < X_{i,n}(u_n) < L \big\}.
\end{align*}
We then partition the particles at time $t_n$ into six subsets.  Recall that $a_{i,n}(s,t)$ is the position of the ancestor at time $s$ of the $i$th particle at time $t$.  We will denote by $k_{i,n}(s,t)$ the index of this ancestor, which means $X_{k_{i,n}(s,t), n}(s) = a_{i,n}(s,t)$.
For $j \in \{1, 2, 3\}$, we define $$S_{j,n} = \{i: k_{i,n}(u_n, t_n) \in G_{j,n}\}.$$  We also define
\begin{align*}
S_{4,n} &= \big\{i: k_{i,n}(u_n, t_n) \in G_{4,n} \mbox{ and } X_{i,n}(t_n) \notin \big[ -C_9 \sqrt{\rho/\beta}, C_9 \sqrt{\rho/\beta} \big]\big\}, \\
S_{5,n} &= \big\{i: k_{i,n}(u_n, t_n) \in G_{4,n}, \: i \notin S_{4,n}, \mbox{ and }a_{i,n}(s, t_n) \geq L \mbox{ for some }s \in (u_n, t_n) \big\}, \\
S_{6,n} &= \big\{i: k_{i,n}(u_n, t_n) \in G_{4,n} \mbox{ and }i \notin S_{4,n} \cup S_{5,n} \big\}.
\end{align*}
For $j \in \{1, \dots, 6\}$, let $$\Psi_{j,n}(g) = \sum_{i \in S_{j,n}} g \Big( X_{i,n}(t_n) \sqrt{\beta/\rho} \Big).$$  Then $\Psi_n(g) = \Psi_{1,n}(g) + \dots + \Psi_{6,n}(g)$.  We will show that with high probability, the values of $\Psi_{j,n}(g)$ are small for $j \in \{1, \dots, 5\}$.  The dominant contribution comes from $\Psi_{6,n}(g)$, which is highly concentrated around its expectation.

The term $\Psi_{1,n}(g)$ accounts for the particles that are above $L$ at time $u_n$. 
To bound this term we can use the argument leading to (\ref{Pphi1}), with $u_n$ in place of $t_n - 2C_1 \rho^{-2}$, to see that with probability tending to one as $n \rightarrow \infty$, no particle that either starts above $L$ or reaches $L$ before time $u_n$ has descendants alive past time $u_n + 2C_1 \rho^{-2}$.  Because $u_n + 2C_1 \rho^{-2} \leq t_n$ for sufficiently large $n$, it follows that
\begin{equation}\label{Ppsi1}
\lim_{n \rightarrow \infty} \P(\Psi_{1,n}(g) = 0) = 1.
\end{equation}

We next consider $\Psi_{2,n}(g)$, which accounts for particles that are below $H(u_n)$ at time $u_n$.  If there is a particle at $x$ at time $u_n$, then by (\ref{intpeq}), the expected number of descendants of this particle alive at time $t_n$ is
\begin{equation}\label{prePsi}
\exp \bigg( \beta x (t_n - u_n) + \frac{\beta^2 (t_n - u_n)^3}{6} - \frac{\beta \rho (t_n - u_n)^2}{2} \bigg).
\end{equation}
Using that $t_n - u_n = (\rho/\beta) - s_n$, after a few lines of algebra we get that the expression in (\ref{prePsi}) is equal to
\begin{equation}\label{EPsi15}
\exp \bigg( (\rho - \beta s_n)x - \frac{\rho^3}{3 \beta} + \frac{\rho^2 s_n}{2} - \frac{\beta^2 s_n^3}{6} \bigg).
\end{equation}
It follows that
\begin{equation}\label{EPsi2}
\E[\Psi_{2,n}(g)|{\cal F}_{u_n}] \leq \|g\| \exp \bigg( - \frac{\rho^3}{3 \beta} + \frac{\rho^2 s_n}{2} - \frac{\beta^2 s_n^3}{6} \bigg) \sum_{i \in G_{2,n}} e^{(\rho - \beta s_n)X_{i,n}(u_n)}.
\end{equation}
Note that (\ref{unbound}) implies that $s_n \ll u_n$ and therefore $\beta s_n \leq \beta u_n/2$ for sufficiently large $n$, so it follows from (\ref{explowy2}) that $$\E \bigg[  \sum_{i \in G_{2,n}} e^{(\rho - \beta s_n)X_{i,n}(u_n)} \Big| {\cal F}_0 \bigg] \ll e^{-\beta^2 u_n^3/73} Y_n(0),$$
and therefore $$\E[\Psi_{2,n}(g)|{\cal F}_0] \ll Y_n(0) \exp \bigg( - \frac{\rho^3}{3 \beta} + \frac{\rho^2 s_n}{2} - \frac{\beta^2 s_n^3}{6} - \frac{\beta^2 u_n^3}{73} \bigg).$$  Because $u_n \gg \rho^{2/3}/\beta^{8/9}$ by (\ref{unbound}), we have $\rho^2 s_n \ll \beta^2 u_n^3$.  Therefore,
\begin{equation}\label{EPsi74}
\E[\Psi_{2,n}(g)|{\cal F}_0] \ll Y_n(0) \exp \bigg( - \frac{\rho^3}{3 \beta} - \frac{\beta^2 u_n^3}{74} \bigg).
\end{equation}
Combining (\ref{EPsi74}) with (\ref{Zasm}) and (\ref{Yasm}) along with the conditional Markov's inequality, we obtain
\begin{equation}\label{Ppsi2}
\lim_{n \rightarrow \infty} \P\Big(\Psi_{2,n}(g) > \eta Z_n(0) e^{-\rho^3/3 \beta} \Big) = 0.
\end{equation}

The reasoning leading to (\ref{EPsi2}) gives
\begin{equation}\label{EPsi3}
\E[\Psi_{3,n}(g)|{\cal F}_{u_n}] \leq \|g\| \exp \bigg( - \frac{\rho^3}{3 \beta} + \frac{\rho^2 s_n}{2} - \frac{\beta^2 s_n^3}{6} \bigg) \sum_{i \in G_{3,n}} e^{(\rho - \beta s_n)X_{i,n}(u_n)}.
\end{equation}
We can write $G_{3,n} = G_{3,n}^* + G_{3,n}^{**}$, where for $i \in G_{3,n}$, we say $i \in G_{3,n}^*$ if $a_{i,n}(0, u_n) \leq H(u_n)$ and $i \in G_{3,n}^{**}$ otherwise.  By (\ref{explowx2}),
\begin{equation}\label{G*}
\E \bigg[  \sum_{i \in G_{3,n}^*} e^{(\rho - \beta s_n)X_{i,n}(u_n)} \bigg| {\cal F}_0 \bigg] \ll e^{-\beta^2 u_n^3/73} Y_n(0).
\end{equation}
When $H(u_n) < x < L$ and $H(u_n) < y < L$, we can estimate $p_{u_n}^{L}(x,y)$ using Lemma \ref{dapprox4}, and the error term $E_0(u_n, x, y)$ will be $o(1)$.  Therefore, for any constant $C_{10} > 2^{1/3}/(Ai'(\gamma_1))^2$, we have for sufficiently large $n$,
\begin{align*}
\E \bigg[  \sum_{i \in G_{3,n}^{**}} e^{(\rho - \beta s_n)X_{i,n}(u_n)} \bigg| {\cal F}_0 \bigg] &= \sum_{i=1}^{N_n(0)} \1_{\{H(u_n) < X_{i,n}(0) < L\}}\int_{H(u_n)}^{L - C_8 \beta^{-1/3}} e^{(\rho - \beta s_n) y} p_{u_n}^{L}(X_{i,n}(0), y) \: dy \\
&\leq C_{10} Z_n(0) \int_{H(u_n)}^{L - C_8 \beta^{-1/3}} e^{-\beta s_n y} \beta^{1/3} \alpha(L - y) \: dy.
\end{align*}
We now make the substitution $z = (2 \beta)^{1/3}(L - y)$ and recall (\ref{C2choice}) and the fact that $s_n = C_7 \beta^{-2/3}$ to get
\begin{align}\label{G**}
\E \bigg[  \sum_{i \in G_{3,n}^{**}} e^{(\rho - \beta s_n)X_{i,n}(u_n)} \bigg| {\cal F}_0 \bigg] &\leq C_{10} Z_n(0) \int_{2^{1/3} C_8}^{\infty} e^{-\beta s_n (L - (2 \beta)^{-1/3} z)} Ai(\gamma_1 + z) \: dz \nonumber \\
&= C_{10} Z_n(0) e^{-\rho^2 s_n/2} \int_{2^{1/3} C_8}^{\infty} e^{2^{-1/3} C_7 (\gamma_1 + z)} Ai(\gamma_1 + z) \: dz \nonumber \\
&\leq \frac{C_{10} \eta}{2} \cdot Z_n(0) e^{-\rho^2 s_n/2} e^{C_7^3/6}.
\end{align}
It follows from (\ref{EPsi3}), (\ref{G*}), and (\ref{G**}), using the reasoning that led to (\ref{EPsi74}) to handle the first term, that for sufficiently large $n$,
$$\E[\Psi_{3,n}(g)|{\cal F}_0] \leq \|g\| e^{-\rho^3/3 \beta} \bigg( Y_n(0) e^{-\beta^2 u_n^3/74} + \frac{C_{10} \eta}{2} \cdot Z_n(0) \bigg).$$
Therefore, using (\ref{Zasm}) and (\ref{Yasm}) along with the conditional Markov's Inequality,
\begin{equation}\label{Ppsi3}
\limsup_{n \rightarrow \infty} \P\Big(\Psi_{3,n}(g) > \eta^{1/2} Z_n(0) e^{-\rho^3/3 \beta} \Big) \leq C_{10} \|g\| \eta^{1/2}.
\end{equation}

It remains to consider the particles between $L - C_8 \beta^{-1/3}$ and $L$ at time $u_n$.  Applying Lemma~\ref{firstmom} with $s=s_n=\rho/\beta-(t_n-u_n)$, there is an ${\cal F}_{u_n}$-measurable random variable $\theta_n$ tending uniformly to zero as $n \rightarrow \infty$ such that
\begin{align}\label{firstPsi56}
\E[\Psi_{5,n}(g) + \Psi_{6,n}(g)|{\cal F}_{u_n}] &= (1 + \theta_n) \sum_{i \in G_{4,n}} \int_{-C_9 \sqrt{\rho/\beta}}^{C_9 \sqrt{\rho/\beta}} \sqrt{\frac{\beta}{2 \pi \rho}} g \bigg( y \sqrt{\frac{\beta}{\rho}} \bigg) \nonumber \\
&\hspace{.2in}\times \exp \bigg( \rho X_{i,n}(u_n) - \frac{\rho^3}{3 \beta} - \frac{\beta y^2}{2 \rho} - \frac{\beta^2 s_n^3}{6} + \beta\Big(\frac{\rho^2}{2\beta} - X_{i,n}(u_n)\Big)s_n \bigg) \: dy \nonumber \\
&= (1 + \theta_n) \bigg( \sum_{i \in G_{4,n}} e^{(\rho - \beta s_n)X_{i,n}(u_n)} \bigg) \nonumber \\
&\hspace{.2in}\times \exp \bigg( \frac{\rho^2 s_n}{2} - \frac{\rho^3}{3 \beta} - \frac{\beta^2 s_n^3}{6} \bigg) \int_{-C_9}^{C_9} g(y) \cdot \frac{1}{\sqrt{2 \pi}} e^{-y^2/2} \: dy.
\end{align}
Because $s_n = C_7 \beta^{-2/3}$, Lemma \ref{rhobetas} applied at time $u_n$ tells us that 
\begin{equation}\label{mainYZ}
\lim_{n \rightarrow \infty} \P \bigg( \frac{(1 - 2\eta) Z_n(0)}{(Ai'(\gamma_1))^2} \leq  \exp \bigg( \frac{\rho^2 s_n}{2} - \frac{\beta^2 s_n^3}{6} \bigg) \sum_{i \in G_{4,n}} e^{(\rho - \beta s_n)X_{i,n}(u_n)} \leq \frac{(1 + 2\eta) Z_n(0)}{(Ai'(\gamma_1))^2} \bigg) = 1.
\end{equation}
Because $g$ is nonnegative and $\theta_n \rightarrow 0$ uniformly as $n \rightarrow \infty$, it follows that
\begin{equation}\label{Psi56upper}
\lim_{n \rightarrow \infty} \P \bigg( \E[\Psi_{5,n}(g) + \Psi_{6,n}(g)|{\cal F}_{u_n}] \leq \frac{1 + 3 \eta}{(Ai'(\gamma_1))^2} e^{-\rho^3/3 \beta} Z_n(0) I(g) \bigg) = 1.
\end{equation}
To obtain the corresponding lower bound, note that by (\ref{C3cond}),
$$\int_{-C_9}^{C_9} g(y) \frac{1}{\sqrt{2 \pi}} e^{-y^2/2} \: dy \geq \int_{-\infty}^{\infty} g(y) \frac{1}{\sqrt{2 \pi}} e^{-y^2/2} \: dy - \eta \|g\| = I(g) \bigg(1 - \frac{\eta \|g\|}{I(g)}\bigg).$$ Also using again that $\theta_n \rightarrow 0$ uniformly as $n \rightarrow \infty$, we get
\begin{equation}\label{Psi56lower}
\lim_{n \rightarrow \infty} \P \bigg(\E[\Psi_{5,n}(g) + \Psi_{6,n}(g)|{\cal F}_{u_n}] \geq \frac{1 - 3 \eta}{(Ai'(\gamma_1))^2} \Big(1 - \frac{\eta \|g\|}{I(g)}\Big) e^{-\rho^3/3 \beta} Z_n(0) I(g)  \bigg) = 1.
\end{equation}

The term $\Psi_{5,n}(g)$ accounts for the particles that reach $L$ between times $u_n$ and $t_n$. We now bound the contribution from this term individually. Take $v\in[0,t_n-u_n]$, and recall the definition of ${\tilde r}^L_x(v)$ from the beginning of section \ref{estimates_LA_sec}.
From Corollary~\ref{rtedge} and the fact that $(\rho^2/2) - \beta L = 2^{-1/3} \beta^{2/3} \gamma_1$, there is a positive constant $C_{11}$ such that
\begin{equation}\label{rxn}
{\tilde r}_x^L(v) \leq \frac{C_{11}(L - x)}{v^{3/2}} \exp \bigg(\rho x - \rho L - \frac{(L - x)^2}{2v} - 2^{-1/3} \beta^{2/3} \gamma_1 v \bigg).
\end{equation}
Now let $m_n(v)$ denote the expected number of descendants in the population at time $t_n$ of a particle that reaches $L$ at time $u_n + v$.  It follows from (\ref{intpeq}) that 
$$m_n(v) = \exp \bigg( \beta L (t_n - u_n - v) + \frac{\beta^2 (t_n - u_n - v)^3}{6} - \frac{\beta \rho (t_n - u_n - v)^2}{2} \bigg).$$
Because $t_n - u_n = (\rho/\beta) - s_n$, a short computation gives
$$\frac{\beta^2 (t_n - u_n - v)^3}{6} - \frac{\beta \rho (t_n - u_n - v)^2}{2} = \frac{\rho^2(v + s_n)}{2} - \frac{\rho^3}{3 \beta} - \frac{\beta^2(v + s_n)^3}{6}.$$  It follows that
\begin{equation}\label{vnu}
m_n(v) = \exp \bigg( \rho L + 2^{-1/3} \beta^{2/3} \gamma_1 (v + s_n) - \frac{\rho^3}{3 \beta} - \frac{\beta^2(v + s_n)^3}{6} \bigg).
\end{equation}
Therefore, using that $\gamma_1 < 0$ and $(v + s_n)^3 \geq s_n^3$, we obtain from (\ref{rxn}) and (\ref{vnu}) that
$${\tilde r}_x^L(v)m_n(v) \leq \frac{C_{11}(L - x)}{v^{3/2}} \exp \bigg(\rho x - \frac{(L - x)^2}{2v} - \frac{\rho^3}{3 \beta} - \frac{\beta^2 s_n^3}{6} \bigg).$$
We now integrate over $v$ and apply Lemma \ref{32int} to see that if there is one particle at $x$ at time $u_n$, then the expected number of particles alive at time $t_n$ whose trajectory crosses $L$ between times $u_n$ and $t_n$ is bounded above by 
\begin{align*}
&C_{11}(L - x) \exp \bigg( \rho x - \frac{\rho^3}{3 \beta} - \frac{\beta^2 s_n^3}{6}\bigg) \int_0^{t_n - u_n}\hspace{-0.5mm} v^{-3/2} e^{-(L - x)^2/2v} \: dv \leq C_{11} \sqrt{2 \pi} \exp \bigg( \rho x - \frac{\rho^3}{3 \beta} - \frac{C_7^3}{6} \bigg).
\end{align*}  
It follows that $$\E[\Psi_{5,n}(g)|{\cal F}_{u_n}] \leq C_{11} \sqrt{2\pi} \|g\| e^{-C_7^3/6} e^{-\rho^3/3 \beta} Y_n(u_n).$$
By applying (\ref{Pphi}) with the constant function that always takes the value $1$ in place of $g$, we have
\begin{equation}\label{Ybd}
\lim_{n \rightarrow \infty} \P \bigg( Y_n(u_n) < \frac{1 + \eta}{(Ai'(\gamma_1))^2} \bigg( \int_0^{\infty} Ai(\gamma_1 + z) \: dz \bigg) Z_n(0) \bigg) = 1.
\end{equation}
Therefore, because $C_7$ was chosen to satisfy (\ref{C1choice2}),
\begin{equation}\label{EPsi6}
\lim_{n \rightarrow \infty} \P \bigg( \E[\Psi_{5,n}(g)|{\cal F}_{u_n}] \geq C_{11} \|g\| \eta Z_n(0) e^{-\rho^3/3 \beta} \bigg) = 0.
\end{equation}
Combining this result with the conditional Markov's Inequality, we get
\begin{equation}\label{Ppsi6}
\limsup_{n \rightarrow \infty} \P\Big( \Psi_{5,n}(g) > \eta^{1/2} Z_n(0) e^{-\rho^3/3 \beta} \Big) \leq C_{11} \|g\| \eta^{1/2}.
\end{equation}

We now consider $\Psi_{6,n}(g)$.  The upper bound (\ref{Psi56upper}) on the expectation still holds when $\Psi_{5,n}(g) + \Psi_{6,n}(g)$ is replaced by $\Psi_{6,n}(g)$.  To get a lower bound on the expectation, let
$$f(\eta) = (1 - 3 \eta) \bigg(1 - \frac{\eta \|g\|}{I(g)} \bigg) - \frac{(Ai'(\gamma_1))^2 C_{11} \|g\| \eta}{I(g)}$$ and then
combine (\ref{Psi56lower}) with (\ref{EPsi6}) to get
\begin{equation}\label{Psi5lower}
\lim_{n \rightarrow \infty} \P \bigg(\E[\Psi_{6,n}(g)|{\cal F}_{u_n}] \geq \frac{f(\eta)}{(Ai'(\gamma_1))^2} e^{-\rho^3/3 \beta} Z_n(0) I(g)  \bigg) = 1.
\end{equation}
We now need to control the fluctuations.  By Lemma \ref{mainvarg},
\begin{align*}
\Var(\Psi_{6,n}(g)|{\cal F}_{u_n}) &\leq \sum_{i \in G_{4,n}} \E_{X_i(u_n)} \bigg[ \bigg( \sum_{j \in S_{6,n}} g \Big( X_{j,n}(t) \sqrt{\beta/\rho} \Big) \bigg)^2 \bigg] \\
&\lesssim \sum_{i \in G_{4,n}} \frac{\beta^{2/3}}{\rho^4} \exp \bigg( \rho X_i(u_n) + \rho L - \frac{2 \rho^3}{3 \beta} - \frac{\beta^2 s_n^3}{3} \bigg) \\
&\leq \frac{\beta^{2/3}}{\rho^4} \exp \bigg(\rho L - \frac{2 \rho^3}{3 \beta} \bigg) Y_n(u_n).
\end{align*}
Therefore,
\begin{equation}\label{condcheb}
\P \big( |\Psi_{6,n}(g) - E[\Psi_{6,n}(g)|{\cal F}_{u_n}]| > \eta e^{-\rho^3/3 \beta} Z_n(0) \big| {\cal F}_{u_n} \big) \lesssim \frac{\beta^{2/3} e^{\rho L}}{\eta^2 \rho^4} \cdot \frac{Y_n(u_n)}{Z_n(0)^2}.
\end{equation}
It follows from (\ref{Zasm}) and (\ref{A1}) that as $n \rightarrow \infty$,
\begin{equation}\label{Zsize}
\frac{\beta^{2/3} e^{\rho L}}{\rho^4} \cdot \frac{1}{Z_n(0)} \rightarrow_p 0.
\end{equation}
It follows from (\ref{Ybd}) and (\ref{Zsize}) that the right-hand side of (\ref{condcheb}) converges in probability to zero as $n \rightarrow \infty$.  Combining this observation with (\ref{Psi56upper}) and (\ref{Psi5lower}), we get that for all $\eta > 0$,
\begin{equation}\label{Ppsi5}
\lim_{n \rightarrow \infty} \P \bigg( \bigg| \Psi_{6,n}(g) - \frac{e^{-\rho^3/3\beta}}{(Ai'(\gamma_1))^2} Z_n(0) I(g) \bigg| > \Big( \eta + \frac{(1 - f(\eta)) I(g)}{(Ai'(\gamma_1))^2} \Big) e^{-\rho^3/3 \beta} Z_n(0) \bigg) = 1.
\end{equation}

The only term remaining to be considered is $\Psi_{4,n}(g)$, which accounts for particles that end up outside the interval $[-C_9 \sqrt{\rho/\beta}, C_9 \sqrt{\rho/\beta}]$ at time $t_n$.  Letting $1$ denote the constant function whose value is always equal to one, we have by (\ref{EPsi15}),
$$\E[\Psi_{4,n}(1) + \Psi_{5,n}(1) + \Psi_{6,n}(1)|{\cal F}_{u_n}]  = \exp \bigg( - \frac{\rho^3}{3 \beta} + \frac{\rho^2 s_n}{2} - \frac{\beta^2 s_n^3}{6} \bigg) \sum_{i \in G_{4,n}} e^{(\rho - \beta s_n) X_{i,n}(u_n)}.$$
Also, using (\ref{firstPsi56}), if $Z$ has a standard normal distribution, then
\begin{align*}
&\E[\Psi_{5,n}(1) + \Psi_{6,n}(1)|{\cal F}_{u_n}] \\
&\hspace{.3in} = (1 + \theta_n) \exp \bigg( \frac{\rho^2 s_n}{2} - \frac{\rho^3}{3 \beta} - \frac{\beta^2 s_n^3}{6} \bigg) \bigg( \sum_{i \in G_{4,n}} e^{(\rho - \beta s_n)X_{i,n}(u_n)} \bigg) (1 - P(|Z| > C_9)).
\end{align*}
Combining these two results and using (\ref{C3cond}) gives
\begin{align*}
\E[\Psi_{4,n}(g)|{\cal F}_{u_n}] &\leq \|g\| \E[\Psi_{4,n}(1)|{\cal F}_{u_n}] \\
&= \|g\| \big( \E[\Psi_{4,n}(1) + \Psi_{5,n}(1) + \Psi_{6,n}(1)|{\cal F}_{u_n}] - \E[\Psi_{5,n}(1) + \Psi_{6,n}(1)|{\cal F}_{u_n}] \big) \\
&\leq \|g\| (|\theta_n| + \eta) \exp \bigg( \frac{\rho^2 s_n}{2} - \frac{\rho^3}{3 \beta} - \frac{\beta^2 s_n^3}{6} \bigg) \bigg( \sum_{i \in G_{4,n}} e^{(\rho - \beta s_n)X_{i,n}(u_n)} \bigg).
\end{align*}
Using the upper bound in (\ref{mainYZ}), and the fact that $\theta_n \rightarrow 0$ uniformly as $n \rightarrow \infty$, it follows
$$\lim_{n \rightarrow \infty} \P \bigg( \E[\Psi_{4,n}(g)|{\cal F}_{u_n}] \leq \frac{\|g\| (2 \eta)(1 + 2 \eta)}{(Ai'(\gamma_1))^2} e^{-\rho^3/3 \beta} Z_n(0) \bigg) = 1.$$  Therefore, by the conditional Markov's Inequality,
\begin{equation}\label{Ppsi4}
\limsup_{n \rightarrow \infty} \P \Big( \Psi_{4,n}(g) > \eta^{1/2} Z_n(0) e^{-\rho^3/3 \beta} \Big) \leq \frac{\|g\| (2 \eta^{1/2})(1 + 2 \eta)}{(Ai'(\gamma_1))^2}.
\end{equation}

Because $\eta > 0$ was arbitrary, we can now obtain (\ref{mainPhin}) from (\ref{Ppsi1}), (\ref{Ppsi2}), (\ref{Ppsi3}), (\ref{Ppsi6}), (\ref{Ppsi5}), and (\ref{Ppsi4}).  The proposition follows.
\end{proof}

\section{Proof of Proposition \ref{YZmain}}\label{YZsec}

To prove Proposition \ref{YZmain}, we essentially show that $Y_n(t_n)$ and $Z_n(t_n)$ remain of the order $\beta^{1/3} \rho^{-3} e^{\rho L}$ after the process has evolved for a time that is of the order $\rho/\beta$.  For the upper bounds, only truncated first moment estimates, in combination with Markov's inequality, are needed.  However, killing particles when they hit $L$ is not sufficient because doing so would kill some particles whose descendants would otherwise contribute significantly to the process at time $t_n$.  Therefore, we instead have to kill particles when they hit $L_{A}$ for $A<0$, and then move the barrier further away as time increases---that is, make $A$ more negative as a function of time---thereby reducing the number of particles that hit the wall.

Obtaining a lower bound on $Z_n(t_n)$ requires second moment estimates.  To obtain adequate second moment estimates, the value of $A$ needs to be chosen so that the wall at $L_A$ moves closer to the origin at regular intervals.  An alternative to this approach would be to follow the techniques in \cite{bbs, bms}, which would likely yield the stronger result that $(Z_n(t), t \geq 0)$ converges to a continuous-state branching process.  However, the simpler arguments given here are sufficient for our purposes.

\subsection{Upper bounds on contributions of particles remaining below $L_A$}

Recall that $a_{i,n}(s,t)$ is the position of the ancestor at time $s$ of the $i$th particle at time $t$, and 
\[z_{n,A}(x) = e^{\rho x}Ai((2\beta)^{1/3}(L_A-x)+\gamma_1)\ind_{\{x<L_A\}}.\]
For $A<0$, define
\begin{equation}\label{Ynstardef}
Y^*_{n,A}(t) = \sum_{i=1}^{N_n(t)} e^{\rho X_{i,n}(t)} \1_{\{a_{i,n}(s,t)<L_A \,\forall s\le t\}}
\end{equation}
and
\begin{equation}\label{Znstardef}
Z^*_{n,A}(t) = \sum_{i=1}^{N_n(t)} z_{n,A}(X_{i,n}(t)) \1_{\{a_{i,n}(s,t)<L_A \,\forall s\le t\}}
\end{equation}
so that $Y^*_{n,A}(t)$ and $Z^*_{n,A}(t)$ only count particles that have remained below the level $L_A$ for all times $s\le t$. Before we introduce our moving barrier, a large part of our upper bound follows relatively easily from estimates on $Y^*_{n,A}(t_n)$ and $Z^*_{n,A}(t_n)$ for fixed $A<0$.

The following fact about the Airy function will help us to compare the values of $z_{n,A}(x)$ for different values of $A$.

\begin{Lemma}\label{AiryZZA}
If $x > 0$ and $1/2 < r < 1$, then $Ai(x + \gamma_1) \leq 2 Ai(rx + \gamma_1)$. 
\end{Lemma}

\begin{proof}
We consider three cases.  First, suppose $0 < x \leq -\gamma_1$.
Because $Ai(x) > 0$ for all $x > \gamma_1$ and the Airy function solves the differential equation $Ai''(z) = z Ai(z)$, the second derivative $Ai''(z)$ is negative for all $z \in (\gamma_1, 0)$.  Therefore, since $r < 1$, the average value of the derivative of the Airy function between $\gamma_1$ and $\gamma_1 + x$ is less than the average value of the derivative of the Airy function between $\gamma_1$ and $\gamma_1 + rx$.  The conclusion of the lemma follows from this observation because $r \geq 1/2$.

Next, let $a_1' < 0$ be the largest zero of the derivative of the Airy function, which is also the point at which the Airy function attains its maximum.
Suppose $x > -\gamma_1$ but $r x < a_1' - \gamma_1$.  Then, we can apply the result from the previous case with $a_1' - \gamma_1$ in place of $x$ to see that $Ai(x + \gamma_1) \leq Ai(a_1') \leq 2 Ai(rx + \gamma_1).$

Finally, suppose $rx \geq a_1' - \gamma_1$.  Because the Airy function is decreasing on $(a_1', \infty)$, we have $Ai(x + \gamma_1) \leq Ai(rx + \gamma_1)$, which implies the conclusion of the lemma.
\end{proof}

We now check that, at time $0$, the value of $Z_{n,A}(0)$, or equivalently $Z^*_{n,A}(0)$, cannot be much larger than that of $Z_n(0)$.

\begin{Lemma}\label{ZAlem}
Fix $\eps>0$, and suppose that (\ref{Zasm}) and (\ref{Yasm}) hold. Let $A < 0$, and define
$$G_n = \bigg\{Z_{n, A}(0) \leq \frac{3}{\delta} \cdot \frac{\beta^{1/3} e^{\rho L}}{\rho^3}\bigg\} \cap \bigg\{Y_{n}(0) \leq \frac{1}{\rho^2} e^{\rho L} \bigg\}.$$
Then $\P(G_n) > 1 - 3\eps$ for sufficiently large $n$.
\end{Lemma}

\begin{proof}
Since \eqref{Yasm} holds, it suffices to show that for large $n$,
$$\P \bigg(Z_{n,A}(0) \leq \frac{3}{\delta} \cdot \frac{\beta^{1/3}}{\rho^3} e^{\rho L} \bigg) > 1 - 2 \eps.$$
If $x \geq L + A/\rho$, which implies that $L_A - x \leq 2|A|/\rho$, then $$z_{n,A}(x) = e^{\rho x} Ai((2 \beta)^{1/3}(L_A - x) + \gamma_1)\ind_{\{x<L_A\}} \lesssim e^{\rho x} \cdot (2 \beta)^{1/3}(L_A - x) \lesssim \frac{\beta^{1/3} |A| e^{\rho x}}{\rho}.$$  It follows that
\begin{equation}\label{fewnearA}
\sum_{i=1}^{N_n(0)} z_{n,A}(X_{i,n}(0)) \1_{\{X_{i,n}(0) \geq L + A/\rho\}} \lesssim \frac{\beta^{1/3} |A| Y_n(0)}{\rho}.
\end{equation}
In view of (\ref{Yasm}), it follows that for sufficiently large $n$,
\begin{equation}\label{ZAbigx}
\P \bigg( \sum_{i=1}^{N_n(0)} z_{n,A}(X_{i,n}(0)) \1_{\{X_{i,n}(0) \geq L + A/\rho\}} < \frac{1}{\delta} \cdot \frac{\beta^{1/3}}{\rho^3} e^{\rho L} \bigg) > 1 - \eps.
\end{equation}

Suppose that $x < L + A/\rho$.  Then $\frac{1}{2}(L_A - x) < L - x < L_A - x$.  Therefore, $z_{n,A}(x) \leq 2 z_{n,0}(x)$ by Lemma \ref{AiryZZA}.  It follows from (\ref{Zasm}) that for sufficiently large $n$,
\begin{equation}\label{ZAsmallx}
\P \bigg( \sum_{i=1}^{N_n(0)} z_{n,A}(X_{i,n}(0)) \1_{\{X_{i,n}(0) < L + A/\rho\}} \le \frac{2}{\delta} \cdot \frac{\beta^{1/3}}{\rho^3} e^{\rho L} \bigg) > 1 - \eps.
\end{equation}
The result follows from (\ref{ZAbigx}) and (\ref{ZAsmallx}).
\end{proof}

Lemma \ref{ZAlem} tells us that we may restrict our attention to the case in which the initial configuration of particles is such that $G_n$ occurs.

\begin{Lemma}\label{YZ0tlem}
If $t\gg\beta^{-2/3}$, then
\[\E[Y^*_{n,A}(t)|{\cal F}_0] \lesssim Z_{n,A}(0)e^{-\beta A t/\rho} + Y_{n}(0) e^{-\beta^2 t^3/73}.\]
\end{Lemma}

\begin{proof}
Recall that $H_A(t) = L_A-\beta t^2/9$, and that $a_{i,n}(s,t)$ is the position of the ancestor at time $s$ of the $i$th particle at time $t$. We divide into three subsets the particles $i$ that are below $L_A$ at time $t$:
\begin{align*}
S_{1,n} &= \{i: a_{i,n}(0,t) \leq H_{A}(t)\}, \\
S_{2,n} &= \{i \notin S_{1,n}: X_{i,n}(t) \leq H_{A}(t) \}, \\
S_{3,n} &= \{i \notin S_{1,n} \cup S_{2,n}\}.
\end{align*}
For $j \in \{1, 2, 3\}$, let
$$Y^*_{n,A,j}(t) = \sum_{i=1}^{N(t)} e^{\rho X_{i,n}(t)} \1_{\{a_{i,n}(s,t)<L_A \,\forall s\le t\}} \1_{\{i \in S_{j,n}\}}.$$
It follows from (\ref{explowx}) that
\begin{equation}\label{Yn1A}
\E[Y^*_{n,A,1}(t)|{\cal F}_{0}] \ll Y_{n}(0)e^{-\beta^2 t^3/73}.
\end{equation}
Likewise, it follows from (\ref{explowy}) that
\begin{equation}\label{Yn2A}
\E[Y^*_{n,A,2}(t)|{\cal F}_{0}] \ll Y_{n}(0)e^{-\beta^2 t^3/73}.
\end{equation}
Finally, noting that $H_A(t)$ was chosen so that \eqref{dcondbig} holds when $H_A(t) < x < L_A$ and $H_A(t) < y < L_A$, equation (\ref{densapprox}) implies that for sufficiently large $n$,
\begin{align}\label{Yn3A}
\E[Y^*_{n,A,3}(t)|{\cal F}_{0}] &\le \frac{(2 \beta)^{1/3}}{(Ai'(\gamma_1))^2}e^{-\beta A t/\rho} \bigg( \int_{-\infty}^{L_{A}} e^{\rho y} \cdot e^{-\rho y} \alpha(L_{A} - y) \: dy \bigg) Z_{n,A}(0) (1 + o(1)) \nonumber \\
&= \frac{e^{-\beta A t/\rho}}{(Ai'(\gamma_1))^2} \bigg(\int_0^{\infty} Ai(z + \gamma_1) \: dz \bigg) Z_{n,A}(0)(1 + o(1)) \nonumber \\
&\lesssim e^{-\beta A t/\rho} Z_{n,A}(0).
\end{align}
Now the result follows from (\ref{Yn1A}), (\ref{Yn2A}), and (\ref{Yn3A}).
\end{proof}

\begin{Cor}\label{YZAtcor}
Fix $\eps>0$ and $A<0$. There exists $\eta>0$, depending on $A$, such that if $\beta^{-2/3} \log(\beta^{1/3}/\rho)^{1/3} \ll t\le \rho/\eps\beta$, then on the event $G_n$,
\[\P\Big(Y^*_{n,A}(t) > \frac{1}{\eta}\cdot\frac{\beta^{1/3}}{\rho^3} e^{\rho L}\Big|\F_0\Big) < \eps.\]
\end{Cor}

\begin{proof}
Note that on the event $G_n$, by Lemma \ref{YZ0tlem},
\[\E[Y^*_{n,A}(t)|\F_0] \lesssim \frac{3}{\delta}\cdot\frac{\beta^{1/3}e^{\rho L}}{\rho^3}e^{-\beta A t/\rho} + \frac{e^{\rho L}}{\rho^2}e^{-\beta^2 t^3/73} \lesssim \frac{\beta^{1/3}e^{\rho L}}{\rho^3} e^{-A/\eps}.\]
By the conditional Markov inequality, on $G_n$,
\[\P\Big(Y^*_{n,A}(t) > \frac{1}{\eta}\cdot\frac{\beta^{1/3}}{\rho^3} e^{\rho L}\Big|\F_0\Big) \lesssim \frac{\beta^{1/3}e^{\rho L}e^{-A/\eps}}{\rho^3}\cdot\frac{\eta\rho^3}{\beta^{1/3}e^{\rho L}}=\eta e^{-A/\eps},\]
and the right-hand side can be made smaller than $\eps$ by choosing $\eta$ small.
\end{proof}

\subsection{A moving barrier}\label{moving_barrier_sec}

To prove the upper bounds in Proposition \ref{YZmain}, we need to upgrade the result on $Y^*$ in Corollary~\ref{YZAtcor} to results about $Y$ and $Z$. To do so, we need to bound how many particles go above $L_A$. As mentioned earlier, a fixed barrier does not give us accurate enough bounds, so we now define a moving barrier. For $A\in\R$ and $s\ge0$, let
\[\Lambda_A(s) = L_A - \frac{2A}{\rho}\Big(\exp\Big(\frac{2\beta}{\rho}s\Big)-1\Big) \hspace{6mm} \text{ and } \hspace{6mm} \Delta_A(s) = \Lambda_A(s)-L_A.\]
Throughout this argument we will take $A<0$, and therefore $\Delta_A$ and all its derivatives are non-negative.
We would like to study the process when particles are killed as soon as they hit the curve $(\Lambda_A(s))_{s\ge0}$. Let $\hat r^A_x(u,v)$ be the expected number of particles that hit the curve between times $u$ and $v$ in this modified process, when starting from a single particle at $x$.

\begin{Lemma}\label{rAabound}
Suppose that $A<0$. Then for any $t\ge 0$ and $x\le L_A$,
\[\hat r^A_x(0,t) \le e^{2A\beta t/\rho} r_x^{L_A}(0,t) - \frac{2A\beta}{\rho}\int_0^t e^{2A\beta s/\rho} r_x^{L_A}(0,s) ds.\]
\end{Lemma}

\begin{proof}
Recall that $(B_t)_{t\ge0}$ is a one-dimensional Brownian motion started at $x$ under $\P_x$, and $T_K = \inf\{t\ge 0 : B_t \ge K\}$. Let $(\mathcal G_t)_{t\ge0}$ be the natural filtration for this Brownian motion. Define
\[T^* = \inf\{t>0 : B_t > \Lambda_A(t)\}.\]
Then by the many-to-one lemma, reasoning as in (\ref{rx0t}), we have
\begin{equation}\label{hatrmto}
\hat r_x^A(u,v) = \E_x\big[ e^{\rho x - \rho \Lambda_A(T^*) - \rho^2 T^*/2 + \int_0^{T^*}\beta B_s ds} \ind_{\{u<T^*\le v\}}\big].
\end{equation}
Define a new probability measure $\Q_x$ by setting
\begin{align}
\frac{d\Q_x}{d\P_x}\bigg|_{\mathcal G_t} &= \exp\Big( \int_0^t \Delta'_A(s) dB_s - \frac12\int_0^t \Delta'_A(s)^2 ds\Big)\nonumber\\
&= \exp\Big(\Delta_A'(t) B_t - \Delta_A'(0)x - \int_0^t \Delta_A''(s)B_s ds - \frac12\int_0^t \Delta'_A(s)^2 ds\Big)\label{QAdef}
\end{align}
where the second expression follows from the first by (stochastic) integration by parts.

Combining \eqref{hatrmto} and \eqref{QAdef} tells us that $\hat r_x^A(u,v)$ equals
\[\Q_x\big[ e^{\rho x - \rho \Lambda_A(T^*) - \rho^2 T^*/2 -\Delta_A'(T^*) \Lambda_A(T^*) + \Delta_A'(0)x + \int_0^{T^*} \Delta_A''(s)B_s ds + \frac12 \int_0^{T^*} \Delta_A'(s)^2 ds + \int_0^{T^*}\beta B_s ds} \ind_{\{u<T^*\le v\}}\big].\]
Note that since $B_s\le \Lambda_A(s)$ for all $s\le T^*$ and $\Delta_A'(s) = \Lambda_A'(s)$ for all $s$, by the standard integration by parts formula,
\[\int_0^{T^*} \Delta_A''(s) B_s ds \le \int_0^{T^*} \Delta_A''(s)\Lambda_A(s) ds = \Delta_A'(T^*)\Lambda_A(T^*) - \Delta_A'(0)L_A - \int_0^{T^*}\Delta_A'(s)^2 ds,\]
and so, since $x\le L_A$ and $\Delta_A'(0)\ge 0$,
\begin{align*}
\hat r_x^A(u,v) &\le \Q_x\big[ e^{\rho x - \rho \Lambda_A(T^*) - \rho^2 T^*/2 + \Delta_A'(0)(x-L_A) - \frac12 \int_0^{T^*} \Delta_A'(s)^2 ds + \int_0^{T^*}\beta B_s ds} \ind_{\{u<T^*\le v\}}\big]\\
&\le \Q_x\big[e^{\rho x - \rho \Lambda_A(T^*) - \rho^2 T^*/2 + \int_0^{T^*}\beta B_s ds} \ind_{\{u<T^*\le v\}}\big].
\end{align*}
Under $\Q_x$, the process $(B_t-\Delta_A(t))_{t\ge0}$ is a Brownian motion started from $x$, so from above,
\begin{align*}
\hat r_x^A(u,v) &\le \P_x\big[e^{\rho x - \rho \Lambda_A(T_{L_A}) - \rho^2 T_{L_A}/2 + \int_0^{T_{L_A}}\beta (B_s+\Delta_A(s)) ds} \ind_{\{u<T_{L_A}\le v\}}\big]\\
&= \P_x\big[e^{ - \rho \Delta_A(T_{L_A}) + \int_0^{T_{L_A}}\beta \Delta_A(s) ds} \cdot e^{\rho x - \rho L_A - \rho^2 T_{L_A}/2 +\int_0^{T_{L_A}}\beta B_s ds} \ind_{\{u<T_{L_A}\le v\}}\big].
\end{align*}
We now note that for any time $t\ge0$,
\[-\rho \Delta_A(t) + \int_0^t \beta \Delta_A(s) ds = A(e^{2\beta t/\rho}-1) + 2A\beta t/\rho \le 2A\beta t/\rho,\]
so
\[\hat r_x^A(u,v) \le \P_x\big[e^{2A\beta T_{L_A}/\rho}\cdot e^{\rho x - \rho L_A - \rho^2 T_{L_A}/2 +\int_0^{T_{L_A}}\beta B_s ds} \ind_{\{u<T_{L_A}\le v\}}\big] \le e^{2A\beta u/\rho}r_x^{L_A}(u,v).\]
Recalling that $\tilde r_x^{L_A}(t)$ is the derivative of $r_x^{L_A}(0,t)$ with respect to $t$, we have
\[\frac{d}{dt} \hat r_x^A(0,t) \le e^{2A\beta t/\rho}\tilde r_x^{L_A}(t),\]
and integrating by parts completes the proof of the lemma.
\end{proof}

\subsection{Proof of the upper bounds in Proposition \ref{YZmain}}

\begin{Lemma}\label{hitL}
Fix $\eps>0$ and suppose that (\ref{Zasm}) and (\ref{Yasm}) hold and that $t\le \rho/\eps\beta$.  Recall the definition of $\Delta$ from (\ref{A3}), and let $C_1$ be the constant from Lemma \ref{Lsurvive}. Then there exists a negative real number $A'$, depending on $\delta$, $\eps$, and $\Delta$, such that if $A \le A'$, then the probability, conditional on $G_n$, that some particle hits the barrier $\Lambda_A(s)$ at some time $s \leq t$ and has descendants that survive for an additional time $2C_1 \rho^{-2}$ is bounded above by $\eps$ for large $n$.
\end{Lemma}

\begin{proof}
First consider a fixed barrier at $L_A$, and let $R_t$ be the number of particles killed at this barrier before time $t\ge0$. From Lemma \ref{Rxlem},
\begin{align*}
\E[R_t|\F_0]
&\lesssim e^{-\rho L_A}Y_n(0) + t\beta^{2/3}e^{-\rho L_A - \beta A t/\rho} Z_{n,A}(0).
\end{align*}
Now let $R^*_t$ be the number of particles that are killed at the moving barrier $\Lambda_A(s)$ for some $s\le t$. By Lemma \ref{rAabound} and the bound above, we have
\begin{align*}
\E[R^*_t|\F_0] &\lesssim e^{2A\beta t/\rho - \rho L_A} Y_n(0) + t\beta^{2/3}e^{(-A + 2A)\beta t/\rho-\rho L_A} Z_{n,A}(0)\\
&\hspace{10mm} -\frac{2A\beta}{\rho} \int_0^t e^{2A\beta s/\rho-\rho L_A} Y_n(0) ds - \frac{2A\beta}{\rho} \int_0^t s\beta^{2/3} e^{(-A + 2A)\beta s/\rho-\rho L_A} Z_{n,A}(0) ds \\
&= e^{-\rho L_A} Y_n(0) + \bigg( \frac{2 \rho}{A \beta^{1/3}} (e^{A \beta t/\rho} - 1) - t \beta^{2/3} e^{A \beta t/\rho} \bigg) e^{-\rho L_A} Z_{n,A}(0).
\end{align*}
Without loss of generality we may assume that $\eps\le 1$. Since $A<0$, it follows that
\[\E[R^*_t|\F_0] \lesssim e^{-\rho L_A} Y_n(0) + \frac{\rho}{\eps\beta^{1/3} |A|}e^{-\rho L_A} Z_{n,A}(0).\]

Recall that
\[G_n = \bigg\{Z_{n, A}(0) \leq \frac{3}{\delta} \cdot \frac{\beta^{1/3} e^{\rho L}}{\rho^3}\bigg\} \cap \bigg\{Y_n(0) \leq \frac{1}{\rho^2} e^{\rho L} \bigg\},\]
so on $G_n$,
\[\E[R^*_t|\F_0] \lesssim e^{-\rho L_A}\frac{1}{\rho^2} e^{\rho L} + \frac{\rho}{\eps\beta^{1/3} |A|}e^{-\rho L_A}\frac{3}{\delta} \cdot \frac{\beta^{1/3} e^{\rho L}}{\rho^3} = \frac{e^A}{\rho^2} + \frac{3e^A}{\delta\eps\rho^2 |A|}.\]
It follows that $A'$ can be chosen so that if $A \le A'$, then $E[R_t^*|{\cal F}_0] \leq \eps\Delta/2\rho^2$.
Because $t \le \rho/\eps\beta$, for sufficiently large $n$ we have $\rho^2/2 \leq \beta \Lambda_A(s) \leq \rho^2$ for all $s \leq t$.
From Lemma \ref{Lsurvive}, it follows that conditional on $G_n$, the probability that some particle reaches the boundary before an arbitrary time $t$ and has descendants that survive for an additional time $2C_1 \rho^{-2}$ is bounded above by a constant multiple of $\eps$, which is sufficient to imply the result because $\eps>0$ was arbitrary.
\end{proof}

We now have the ingredients to complete the proof of the upper bound Proposition \ref{YZmain}, in the form of the following lemma. Note in particular that, in view of (\ref{A1}), the result (\ref{mainYupper}) is stronger than the required conclusion that $\rho^{-2} e^{-\rho L} Y_n(t_n) \rightarrow_p 0$.

\begin{Lemma}\label{mainupper}
Fix $\eps>0$ and suppose that (\ref{Zasm}) and (\ref{Yasm}) hold, and that the sequence of times $t_n$ satisfy $\beta t_n/\rho \to \tau\in(0,\infty)$. Then there exists $\eta>0$ such that
\begin{equation}\label{mainYupper}
\P \bigg( Y_n(t_n) \leq \frac{1}{\eta} \cdot \frac{\beta^{1/3}}{\rho^3} e^{\rho L} \bigg) > 1 - 6 \eps
\end{equation}
and
\begin{equation}\label{mainZupper}
\P \bigg( Z_n(t_n) \leq \frac{1}{\eta} \cdot \frac{\beta^{1/3}}{\rho^3} e^{\rho L} \bigg) > 1 - 6 \eps.
\end{equation}
\end{Lemma}

\begin{proof}
Choose $A'$ as in Lemma \ref{hitL}, and fix $A \le A'$.  Fix $A^*$ such that for all large $n$,
\[A^* \le A+2A\Big(\exp\Big(\frac{2\beta}{\rho}\Big(t_n-\frac{2C_1}{\rho^2}\Big)\Big)-1\Big)\]
so that
\[L_{A^*} \ge \Lambda_{A}(t_n-2C_1/\rho^2).\]
Recall that, by Corollary \ref{YZAtcor}, there exists $\eta>0$ such that on $G_n$,
\[\P\Big(Y^*_{n,A^*}(t_n) > \frac{1}{\eta}\frac{\beta^{1/3}}{\rho^3} e^{\rho L}\Big|\F_0\Big) < \eps.\]
We therefore need to consider those particles that contribute to $Y_n(t_n)-Y^*_{n,A^*}(t_n)$; such particles must be above level $L_{A^*}$ at some time before $t_n$.

Note that $\Lambda_A(s) \le L_{A^*}$ for all $s\le t_n-2C_1/\rho^2$. Therefore in order to contribute to $Y_n(t_n)-Y^*_{n,A^*}(t_n)$, a particle must do one of the following:
\begin{enumerate}[label=(\alph*)]
\item start above $L_{A}$ and survive until time $t_n$;
\item hit $\Lambda_A(s)$ for some $s\le t_n-2C_1/\rho^2$ and then survive for an additional time of $2C_1/\rho^2$;
\item hit $L_{A^*}$ between times $t_n-2C_1/\rho^2$ and $t_n$.
\end{enumerate}

By Lemma \ref{aboveL}, the probability that any particle does (a) tends to $0$ as $n\to\infty$. Lemma~\ref{hitL} tells us that on $G_n$, the probability that any particle does (b) is bounded above by $\eps$. It therefore remains to consider case (c). Let $R'$ be the number of particles that hit $L_{A^*}$ between times $t_n-2C_1/\rho^2$ and $t_n$. Then on $G_n$, by Lemma \ref{Rxlem},
\[\E[R'|\F_0] \lesssim e^{-\rho L_{A^*} - \tau^3 \rho/10\beta} \cdot \frac{1}{\rho^2}e^{\rho L} + \frac{1}{\rho^2} e^{-\rho L_{A^*}}\beta^{2/3}e^{|A^*|\tau}\cdot \frac{\beta^{1/3}}{\rho^3}e^{\rho L}.\]
We have $\tau\asymp 1$ and $A^*$ does not depend on $n$.  Also, by \eqref{A1} we have $e^{-\tau^3 \rho/10 \beta}\ll \beta^{1/3}/\rho$ and $\beta/\rho^5\ll \beta^{1/3}/\rho^3$. Therefore
\[\E[R'|\F_0] \ll \frac{\beta^{1/3}}{\rho^3}e^{\rho L-\rho L_{A^*}}.\]
By Lemma \ref{YfromL}, the expected contribution from each of these particles to $Y_n(t_n)$ is at most $e^{\rho L_{A^*}}$.  Therefore, letting $Y'$ denote the total contribution to $Y_n(t_n)$ from these particles, we have
\[\E[Y'|\F_0] \ll \frac{\beta^{1/3}}{\rho^3}e^{\rho L}.\]
The conditional Markov's inequality implies that for large $n$, on $G_n$,
\[\P\Big(Y_n(t_n)-Y^*_{n,A^*}(t_n) > \frac{\beta^{1/3}}{\rho^3}e^{\rho L} \Big| \F_0\Big) \le 2\eps + \P\Big(Y' > \frac{\beta^{1/3}}{\rho^3}e^{\rho L} \Big| \F_0\Big) < 3\eps.\]
Since $\P(G_n)>1-3\eps$ for large $n$ by Lemma \ref{ZAlem}, this completes the proof for $Y_n(t_n)$. Because $|Ai(x)| \leq 1$ for all $x$ (see table 9.9.1 in \cite{dlmf}), the result for $Z_n(t_n)$ follows immediately.
\end{proof}

\subsection{Proof of the lower bound in Proposition \ref{YZmain}}

To prove Proposition \ref{YZmain}, it remains to establish the lower bound for $Z_n(t_n)$.  This requires using a second moment argument to control the fluctuations.  To do this, we will construct another modification of the original process. We would essentially like to use the moving barrier from Section \ref{moving_barrier_sec}, with $A$ chosen positive so that the barrier moves closer to the origin as time increases. However, our second moment bound Lemma \ref{2momprop} holds only for a fixed barrier at $L_A$. Developing the required second moment bounds for the moving barrier would require substantial extra work, and it is much more convenient to instead mimic the moving barrier with a series of fixed barriers that move progressively closer to $0$.

Fix $\eps > 0$ and choose $\delta>0$ such that (\ref{Zasm}) holds. Suppose that $\beta t_n/\rho\to \tau\in(0,\infty)$. Let $C_{12}$ be a positive constant chosen so that (\ref{mainvargeq}) is a strict inequality for all $n$ if the right-hand side is multiplied by $C_{12}$. Fix a positive number $A$ large enough that $1/A < \tau$ and
\begin{equation}\label{Aprime}
\frac{288 e^4 C_{12} e^{-A}((4e^2)^{\tau(1 + \eps)} - 1)}{\delta \log(4e^2)} < \eps.
\end{equation}
Let $A_0 = A$, and let $A_k = A + k \log (4e^2)$ for positive integers $k$. 
Choose $J_n\in\N$ and times $$0 = u_0 < u_1 < \dots < u_{J_n} = t_n$$ such that
\begin{equation}\label{ukcond}
\frac{\rho}{\beta A_k} \leq u_{k+1} - u_k \leq \frac{2 \rho}{\beta A_k}
\end{equation}
for all $k \in \{0, 1, \dots, J_{n-1}\}$, which is possible for sufficiently large $n$ because $1/A < \tau$. Note that
\[u_k\ge \frac{\rho}{\beta}\sum_{j=0}^{k-1} \frac{1}{A+j\log(4e^2)} \ge \frac{\rho}{\beta}\int_0^k \frac{1}{A+x\log(4e^2)}\:dx = \frac{\rho}{\beta \log(4e^2)}\log\Big(1+\frac{k\log(4e^2)}{A}\Big).\]
Since $u_{J_n} = t_n\sim \rho\tau/\beta$, it follows that for large $n$,
\[\frac{\rho\tau(1+\eps)}{\beta} > \frac{\rho}{\beta\log(4e^2)}\log\Big(1+\frac{J_n\log(4e^2)}{A}\Big),\]
and therefore
\begin{equation}\label{Jprime}
J_n \leq \frac{A((4e^2)^{\tau(1 + \eps)} - 1)}{\log(4e^2)}.
\end{equation}
Choose positive numbers $C_{13}$ and $C_{14}$ such that $0 < C_{13} < C_{14} < \infty$ and
\begin{equation}\label{c1c2}
\int_{2^{1/3} C_{13}}^{2^{1/3} C_{14}} Ai(z + \gamma_1)^2 \: dz > \frac{(Ai'(\gamma_1))^2}{2},
\end{equation}
which is possible by (\ref{orth}) with $j = k = 1$.  Recalling (\ref{Kdef}), we consider a modified process in which particles are killed at time $0$ unless they lie in the interval
\begin{equation}\label{intL0}
I_0 = \big( K_{A}(u_1), \: L_{A} \big).
\end{equation}
For $k \in \{0, 1, \dots, J_n - 1\}$, particles are killed if they reach $L_{A_k}$ between times $u_k$ and $u_{k+1}$, and then particles are also killed at time $u_{k+1}$ unless they are in the interval
\begin{equation}\label{intLk}
I_{k+1} = \Big[ L_{A_k} - C_{14} \beta^{-1/3}, \: L_{A_k} - C_{13} \beta^{-1/3} \Big].
\end{equation}
Letting $G_i'(t)$ be the event that the $i$th particle at time $t$ in the original process has not been killed by time $t$ in this modification, we define $$Z_{n,A}'(t) = \sum_{i=1}^{N_n(t)} e^{\rho X_{i,n}(t)} Ai((2 \beta)^{1/3}(L_A - X_{i,n}(t)) + \gamma_1) \1_{\{X_{i,n}(t) < L_A\}} \1_{G_i'(t)}.$$

\begin{Lemma}\label{ZAlem2}
Fix $\eps>0$ and suppose that (\ref{Zasm}) and (\ref{Yasm}) hold.  Then for sufficiently large $n$,
$$\P \bigg(Z_{n,A}'(0) \geq \frac{\delta}{4} \cdot \frac{\beta^{1/3}}{\rho^3} e^{\rho L} \bigg) > 1 - 3 \eps.$$
\end{Lemma}

\begin{proof}
If $x \geq L - 2A/\rho$, then $L - x \leq 2A/\rho$ and therefore
$$z_{n,0}(x) = e^{\rho x} Ai((2 \beta)^{1/3}(L - x) + \gamma_1) \lesssim \frac{A \beta^{1/3} e^{\rho x}}{\rho}.$$
Thus
$$\sum_{i=1}^{N_n(0)} z_{n,0}(X_{i,n}(0)) \1_{\{X_{i,n}(0) \geq L - 2A/\rho\}} \lesssim \frac{A \beta^{1/3} Y_n(0)}{\rho}.$$  In view of (\ref{Yasm}), it follows for sufficiently large $n$ that $$\P \bigg( \sum_{i=1}^{N_n(0)} z_{n,0}(X_{i,n}(0)) \1_{\{X_{i,n}(0) \geq L - 2A/\rho\}} < \frac{\delta}{4} \cdot \frac{\beta^{1/3}}{\rho^3} e^{\rho L} \bigg) > 1 - \eps.$$  
Likewise, by the reasoning that led to (\ref{Znprime}),
$$\P \bigg( \sum_{i=1}^{N_n(0)} z_{n,0}(X_{i,n}(0)) \1_{\{X_{i,n}(0) \leq K_{A}(u_1)\}} < \frac{\delta}{4} \cdot \frac{\beta^{1/3}}{\rho^3} e^{\rho L} \bigg) > 1 - \eps.$$
Therefore, in view of (\ref{Zasm}),
\begin{equation}\label{z0trunc}
\P \bigg( \sum_{i=1}^{N_n(0)} z_{n,0}(X_{i,n}(0)) \1_{G_i'(0)} \1_{\{X_{i,n}(0) < L - 2A/\rho\}} > \frac{\delta}{2} \cdot \frac{\beta^{1/3}}{\rho^3} e^{\rho L} \bigg) > 1 - 3\eps.
\end{equation}

Now suppose $x < L - 2A/\rho$.  Then $\frac{1}{2}(L - x) \leq L_{A} - x \leq L - x$.  Therefore, by Lemma \ref{AiryZZA}, we have $z_{n,0}(x) \leq 2 z_{n,A}(x)$.  In view of (\ref{z0trunc}), it follows that $$\P \bigg( \sum_{i=1}^{N_n(0)} z_{A,n}(X_{i,n}(0)) \1_{G_i'(0)} \1_{\{X_{i,n}(0) < L - 2A/\rho\}} > \frac{\delta}{4} \cdot \frac{\beta^{1/3}}{\rho^3} e^{\rho L} \bigg) > 1 - 3\eps,$$ which implies the result.
\end{proof}

Lemma \ref{mainZlow} below gives the lower bound on $Z_n(t_n)$ that is needed to complete the proof of Proposition \ref{YZmain}.  In particular, Proposition \ref{YZmain} follows directly from Lemmas \ref{mainupper} and \ref{mainZlow}.

\begin{Lemma}\label{mainZlow}
Fix $\eps>0$ and suppose that (\ref{Zasm}) and (\ref{Yasm}) hold. Suppose that $\beta t_n/\rho\to\tau\in(0,\infty)$. There exists $\eta>0$ such that for sufficiently large $n$,
$$\P \bigg( Z_n(t_n) \geq \eta \cdot \frac{\beta^{1/3}}{\rho^3} e^{\rho L} \bigg) > 1 - 5 \eps.$$
\end{Lemma}

\begin{proof}
Let $J = A((4e^2)^{\tau(1 + \eps)} - 1)/\log(4e^2)$ be the right-hand side of (\ref{Jprime}), and let $$0 < \eta < \frac{\delta}{4 (4e^2)^{J}}.$$

For $k \in \{0, 1, \dots, J_n\}$, let $$G_{k,n} = \bigg\{Z'_{n,A_k}(u_k) \geq \frac{\delta}{4 (4 e^2)^k} \cdot \frac{\beta^{1/3}}{\rho^3} e^{\rho L}\bigg\}.$$  Lemma \ref{ZAlem2} implies that $\P(G_{0,n}) > 1 - 3\eps$ for sufficiently large $n$.

For $k \in \{0, 1, \dots, J_n - 1\}$, we consider the evolution of the process between times $u_k$ and $u_{k+1}$.  
Recall that all particles at time $u_k$ are in the interval $I_k$, while particles will be killed at time $u_{k+1}$ unless they are in the interval $I_{k+1}$.  Particles will also be killed if they reach $L_{A_k,n}$ between these two times.  Note that these intervals have been chosen in such a way that if $x \in I_k$ and $y \in I_{k+1}$, then the density $$p_{u_{k+1} - u_k}^{L_{A_k}}(x,y)$$ can be approximated using Lemma \ref{dapprox4} because the error term in (\ref{densapprox}) tends to zero uniformly over $x \in I_k$ and $y \in I_{k+1}$ as $n \rightarrow \infty$.  Likewise, Lemma \ref{2momprop} can be applied for second moment calculations because any $x \in I_k$ and $y \in I_{k+1}$ will satisfy the conditions of Lemma \ref{2momprop} if $n$ is large enough.

By Lemma \ref{dapprox4},
$$\E[Z'_{n,A_k}(u_{k+1})|{\cal F}_{u_k}] \sim \frac{(2 \beta)^{1/3} e^{-\beta A_k (u_{k+1} - u_k)/\rho}}{(Ai'(\gamma_1))^2} Z'_{n,A_k}(u_k) \int_{I_{k+1}} \alpha(L_{A_k} - y)^2 \: dy.$$
Making the substitution $z = (2 \beta)^{1/3}(L_{A_k} - y)$ and applying (\ref{c1c2}), we have
$$\frac{(2 \beta)^{1/3}}{(Ai'(\gamma_1))^2} \int_{I_{k+1}} \alpha(L_{A_k} - y)^2 \: dy = \frac{1}{(Ai'(\gamma_1))^2} \int_{2^{1/3} C_{13}}^{2^{1/3} C_{14}} Ai(z + \gamma_1)^2 \: dz > \frac{1}{2}.$$  By (\ref{ukcond}), $e^{-\beta A_k (u_{k+1} - u_k)/\rho} \geq e^{-2}.$  It follows that for sufficiently large $n$,
\begin{equation}\label{EZprime}
\E[Z'_{n,A_k}(u_{k+1})|{\cal F}_{u_k}] \geq \frac{1}{2e^2} Z'_{n,A_k}(u_k).
\end{equation}
Because all particles at time $u_k$ and their descendants evolve independently between times $u_k$ and $u_{k+1}$, it follows from Lemma \ref{2momprop} that
$$\Var(Z'_{n,A_k}(u_{k+1})|{\cal F}_{u_k}) \leq \frac{C_{12} \beta^{2/3}e^{\rho L_{A_k}}}{\rho^4} \big( Y_n(u_k) + \beta^{2/3}(u_{k+1} - u_k) Z'_{n,A_k}(u_k) \big).$$  Therefore, by Chebyshev's Inequality,
\begin{align*}
&\P \bigg( \big|Z'_{n,A_k}(u_{k+1}) - \E\big[Z'_{n,A_k}(u_{k+1})|{\cal F}_{u_k} \big] \big| > \frac{1}{6e^2} Z'_{n,A_k}(u_k) \Big| {\cal F}_{u_k} \bigg) \\
&\hspace{1in}\leq \frac{36 e^4 C_{12} \beta^{2/3} e^{\rho L_{A_k}}\big(Y_n(u_k) + \beta^{2/3}(u_{k+1} - u_k) Z'_{n,A_k}(u_k)\big)}{\rho^4 (Z'_{n,A_k}(u_k))^2}.
\end{align*}
Note that $\rho^2 e^{-\rho L} Y_n(0) \rightarrow_p 0$ as $n \rightarrow \infty$ from (\ref{Yasm}) and $\rho^2 e^{-\rho L} Y_n(u_k) \rightarrow_p 0$ as $n \rightarrow \infty$ for $k \in \{1, \dots, J_n - 1\}$ by Lemma \ref{mainupper}, with $u_k$ in place of $t_n$.  Therefore, for all $k \in \{0, 1, \dots, J_n - 1\}$, as $n \rightarrow \infty$,
$$\frac{36 e^4 C_{12} \beta^{2/3} e^{\rho L_{A_k}}Y_n(u_k) \1_{G_{k,n}}}{\rho^4 (Z'_{n,A_k}(u_k))^2} \rightarrow_p 0.$$  Using (\ref{Aprime}) and recalling that $A_k = A + k \log (4e^2)$, we also have
$$\frac{36 e^4 C_{12} \beta^{2/3} e^{\rho L_{A_k}} \cdot \beta^{2/3}(u_{k+1} - u_k) Z'_{n,A_k}(u_k) \1_{G_{k,n}}}{\rho^4 (Z'_{n,A_k}(u_k))^2} \leq \frac{288 e^4 C_{12} e^{-A}}{\delta A_k} < \frac{\eps \log(4e^2)}{A_k ((4e^2)^{\tau(1 + \eps)} - 1)}.$$
It follows that
\begin{align*}
&\limsup_{n \rightarrow \infty} \P \bigg( \bigcup_{k=0}^{J_n - 1} \Big\{\big|Z'_{n,A_k}(u_{k+1}) - \E\big[Z'_{n,A_k}(u_{k+1})|{\cal F}_{u_k} \big] \big| > \frac{1}{6e^2} Z'_{n,A_k}(u_k) \Big\} \cap G_{k,n} \bigg) \\
&\hspace{3.7in}\leq \sum_{k=0}^{J_n - 1} \frac{\eps \log(4e^2)}{A_k ((4e^2)^{\tau(1 + \eps)} - 1)} \leq \eps,
\end{align*}
where the last inequality is from \eqref{Jprime}. Combining this result with (\ref{EZprime}), we get
\begin{equation}\label{Zprime2}
\limsup_{n \rightarrow \infty} \P \bigg( \bigcup_{k=0}^{J_n - 1} \Big\{ Z'_{n,A_k}(u_{k+1}) < \frac{1}{3e^2} Z'_{n,A_k}(u_k) \Big\} \cap G_{k,n} \bigg) \leq \eps.
\end{equation}
For all $k \in \{1, \dots, J_n\}$ and all positive real numbers $a_1$ and $a_2$, we have
\begin{equation}\label{ZA1A2}
\lim_{n \rightarrow \infty} \inf_{y \in I_{k}} \frac{z_{n,a_1}(y)}{z_{n,a_2}(y)} = \lim_{n \rightarrow \infty} \sup_{y \in I_{k}} \frac{z_{n,a_1}(y)}{z_{n,a_2}(y)} = 1.
\end{equation}
It follows from (\ref{Zprime2}) and (\ref{ZA1A2}) that
$$\limsup_{n \rightarrow \infty} \P \bigg( \bigcup_{k=0}^{J_n - 1} \Big\{ Z'_{n,A_{k+1}}(u_{k+1}) < \frac{1}{4e^2} Z'_{n,A_k}(u_k) \Big\} \cap G_{k,n} \bigg) \leq \eps.$$
Therefore, by the definition of the events $G_{k,n}$, and using that $\P(G_{0,n}) > 1 - 3\eps$ for sufficiently large $n$, we have $$\liminf_{n \rightarrow \infty} \P(G_{J_n,n}) \geq \liminf_{n \rightarrow \infty} \P \bigg( \bigcap_{k=0}^{J_n} G_{k,n} \bigg) \geq 1 - 4 \eps.$$  That is,
$$\liminf_{n \rightarrow \infty} \P \bigg( Z_{n,A_{J_n}}(u_{J_n}) \geq \frac{\delta}{4(4 e^2)^{J}} \cdot \frac{\beta^{1/3}}{\rho^3} e^{\rho L} \bigg) \geq 1 - 4 \eps.$$  The result now follows from another application of (\ref{ZA1A2}).
\end{proof}

\section{Second moment calculations}\label{2momproof}

\subsection{Some integral estimates}

\begin{Lemma}\label{BMmoments}
Take $k\in\N$, $\rho > 0$, and $t>0$. If $a\ge \rho t/2$, then
\begin{equation}\label{BMmom1}
\int_0^\infty x^k e^{-(x-a)^2/t - \rho x} dx \asymp t^{1/2}\big( (a-\rho t/2)^k + t^{k/2}\big)e^{\rho^2 t/4 - a\rho},
\end{equation}
and if $a< \rho t/2$, then
\begin{equation}\label{BMmom2}
\int_0^\infty x^k e^{-(x-a)^2/t - \rho x} dx \asymp \frac{t^{k+1}}{(\rho t/2-a)^{k+1}+t^{(k+1)/2}}e^{-a^2/t}.
\end{equation}
\end{Lemma}

\begin{proof}
Note that
\[\int_0^\infty x^k e^{-(x-a)^2/t - \rho x} dx = e^{\rho^2 t/4 - a\rho}\int_0^\infty x^k e^{-(x-a+\rho t/2)^2/t} dx,\]
so it suffices to show that for $b\ge 0$,
\[\int_0^\infty y^k e^{-(y-b)^2/t} dy \asymp t^{1/2}(t^{k/2}+b^k)\]
and
\[\int_0^\infty y^k e^{-(y+b)^2/t} dy \asymp \frac{t^{k+1}}{b^{k+1}+t^{(k+1)/2}}e^{-b^2/t}.\]
For the former,
\begin{align*}
\int_0^\infty y^k e^{-(y-b)^2/t} dy &= \int_{-b}^\infty (y+b)^k e^{-y^2/t} dy\\
&= \int_0^\infty (y+b)^k e^{-y^2/t} dy + \int_{-b}^0 (y+b)^k e^{-y^2/t} dy\\
&\asymp \int_0^\infty (y+b)^k e^{-y^2/t} dy\\
&\asymp \int_0^\infty y^k e^{-y^2/t} dy + b^k \int_0^\infty e^{-y^2/t} dy\\
&\asymp t^{(1+k)/2} + b^k t^{1/2},
\end{align*}
as required.
For the latter,
\[\int_0^\infty y^k e^{-(y+b)^2/t} dy = t^{(k+1)/2} \int_{0}^\infty x^k e^{-(x + bt^{-1/2})^2} dx\]
so it suffices to show that for $\gamma\ge 0$,
\[\int_0^\infty x^k e^{-(x+\gamma)^2} dx \asymp \frac{1}{\gamma^{k+1}+1}e^{-\gamma^2}.\]
This clearly holds when $\gamma\in[0,1]$, so we may assume that $\gamma>1$. We have
\[\int_0^\infty x^k e^{-(x+\gamma)^2} dx = e^{-\gamma^2} \int_0^\infty x^k e^{-x^2-2\gamma x} dx\]
which is bounded above by
\[e^{-\gamma^2}\int_0^\infty x^k e^{-2\gamma x} dx \asymp \frac{1}{\gamma^{k+1}}e^{-\gamma^2}\]
and below, using the assumption that $\gamma>1$, by
\[e^{-\gamma^2}\int_{1/2\gamma}^{1/\gamma} x^k e^{-x^2-2\gamma x} dx \asymp \frac{1}{\gamma^{k+1}}e^{-\gamma^2}.\]
The result follows.
\end{proof}

\begin{Lemma}\label{expintlem}
For any fixed $k \geq 0$, we have, for all $a \geq 0$ and $\lambda > 0$,
\begin{equation}\label{expint1}
\int_a^{\infty} x^k e^{-\lambda x} \: dx \lesssim e^{-\lambda a} \bigg( \frac{1}{\lambda^{k+1}} + \frac{a^k}{\lambda} \bigg).
\end{equation}
If $0 \leq k \leq 1$, then for all $a \geq 0$ and $\lambda > 0$, we have
\begin{equation}\label{expint2}
\int_a^{\infty} x^k e^{-\lambda x^2} \: dx \lesssim e^{-\lambda a^2} \cdot \frac{a^{k-1}}{\lambda}.
\end{equation}
\end{Lemma}

\begin{proof}
Making the substitution $y = x - a$, we have
$$\int_a^{\infty} x^k e^{-\lambda x} \: dx = \int_0^{\infty} (y + a)^k e^{-\lambda(y + a)} \: dy \lesssim e^{-\lambda a} \int_0^{\infty} (y^k + a^k) e^{-\lambda y} \: dy \lesssim e^{-\lambda a} \bigg( \frac{1}{\lambda^{k+1}} + \frac{a^k}{\lambda} \bigg),$$ which gives (\ref{expint1}).
Making the substitution $y = x^2$, we have
$$\int_a^{\infty} x^k e^{-\lambda x^2} \: dx = \frac{1}{2} \int_{a^2}^{\infty} y^{(k-1)/2} e^{-\lambda y} \: dy.$$
The bound $y^{(k-1)/2} \leq a^{k-1}$ leads to (\ref{expint2}).
\end{proof}

\begin{Lemma}\label{alphaintlem}
For any fixed $k > 0$, we have
\begin{equation}
\int_{-\infty}^{L_A} [\alpha(L_A - y)]^k \: dy \asymp \beta^{-1/3}
\end{equation}
\end{Lemma}

\begin{proof}
Recalling (\ref{alphadef}), making the substitution $z = (2 \beta)^{1/3}(L_A - y)$, and then using (\ref{Airyasymp}) and the continuity of the Airy function, we get
$$\int_{-\infty}^{L_A} [\alpha(L_A - y)]^k \: dy = (2 \beta)^{-1/3} \int_0^{\infty} [Ai(z + \gamma_1)]^k \: dz \asymp \beta^{-1/3},$$
as claimed.
\end{proof}

\subsection{Proof of Lemma \ref{2momprop}}

Recall that
\[V_{\varphi, n, A}(t) = \sum_{i=1}^{N_n(t)} e^{\rho X_{i,n}(t)} \varphi(X_{i,n}(t)) \1_{\{X_{i,n}(t) < L_A\}}.\]
Standard second moment calculations, which go back to early work on branching Markov processes by Ikeda, Nagasawa, and Watanabe (see p. 146 of \cite{inw}) give
\begin{align*}
\E_x[V_{\varphi, n, A}(t)^2] &= \int_{-\infty}^{L_A} e^{2 \rho y} \varphi(y)^2 p_t^{L_A}(x,y) \: dy \\
&\hspace{.4in}+ 2 \int_0^t \int_{-\infty}^{L_A} p_s^{L_A}(x,z)b_n(z) \bigg( \int_{-\infty}^{L_A} e^{\rho y} \varphi(y) p_{t-s}^{L_A}(z,y) \: dy \bigg)^2 \: dz \: ds.
\end{align*}
For $z \leq L_A$, the birth rate $b_n(z)$ is bounded by (\ref{A3}).  Using also that $\varphi$ is bounded and equals zero except on $[K_A(t),L_A]$, we have
\begin{equation}\label{2momeq}
\E_x[V_{\varphi,n,A}(t)^2] \lesssim \int_{K_A(t)}^{L_A} e^{2 \rho y} p_t^{L_A}(x,y) \: dy + \int_0^t \int_{-\infty}^{L_A} p_s^{L_A}(x,z) \bigg( \int_{K_A(t)}^{L_A} e^{\rho y} p_{t-s}^{L_A}(z,y) \: dy \bigg)^2 \: dz \: ds.
\end{equation}
We now split the last term into six parts.  We define
$$b = \min\{8 \beta^{-2/3}, t/2\}$$ and note that we may, and will, assume that $\rho^{-2} \leq t/2$ because of (\ref{A1}) and the assumption that $t \gtrsim \beta^{-2/3}$. Recall from \eqref{Kdef} that $l(t)=\beta t^2/33$ and $K_A(t) = L_A-l(t)/2$. 
We write
\[\textrm{I}:= \int_0^{\rho^{-2}} \int_{L_A-l(t)}^{L_A} p_s^{L_A}(x,z) \bigg( \int_{K_A(t)}^{L_A} e^{\rho y} p_{t-s}^{L_A}(z,y) \: dy \bigg)^2 \: dz \: ds,\]
\[\textrm{II}:= \int_{\rho^{-2}}^{b} \int_{L_A-l(t)}^{L_A} p_s^{L_A}(x,z) \bigg( \int_{K_A(t)}^{L_A} e^{\rho y} p_{t-s}^{L_A}(z,y) \: dy \bigg)^2 \: dz \: ds,\]
\[\textrm{III}:= \int_{b}^{t/2} \int_{L_A-l(t)}^{L_A} p_s^{L_A}(x,z) \bigg( \int_{K_A(t)}^{L_A} e^{\rho y} p_{t-s}^{L_A}(z,y) \: dy \bigg)^2 \: dz \: ds,\]
\[\textrm{IV}:= \int_{t/2}^{t-b} \int_{L_A-l(t)}^{L_A} p_s^{L_A}(x,z) \bigg( \int_{K_A(t)}^{L_A} e^{\rho y} p_{t-s}^{L_A}(z,y) \: dy \bigg)^2 \: dz \: ds,\]
\[\textrm{V}:= \int_{t-b}^{t} \int_{L_A-l(t)}^{L_A} p_s^{L_A}(x,z) \bigg( \int_{K_A(t)}^{L_A} e^{\rho y} p_{t-s}^{L_A}(z,y) \: dy \bigg)^2 \: dz \: ds,\]
\[\textrm{VI}:= \int_0^{t} \int_{-\infty}^{L_A-l(t)} p_s^{L_A}(x,z) \bigg( \int_{K_A(t)}^{L_A} e^{\rho y} p_{t-s}^{L_A}(z,y) \: dy \bigg)^2 \: dz \: ds.\]
The next seven lemmas, which bound these six terms as well as the first term on the right-hand side of (\ref{2momeq}), will imply Lemma \ref{2momprop}.

\begin{Lemma}
Under the assumptions of Lemma \ref{2momprop}, we have
$$\int_{K_A(t)}^{L_A} e^{2\rho y} p_t^{L_A}(x,y) \: dy \ll \frac{\beta^{4/3}}{\rho^4}e^{\rho L_A} t z_A(x).$$
\end{Lemma}

\begin{proof}
When $K_A(t) < x < L_A$ and $K_A(t) < y < L_A$, equation (\ref{dcond}) holds because $\beta^{-2/3} \lesssim t$.  Therefore, by Lemma \ref{dapprox4} and (\ref{airyder2}),
\begin{align*}
\int_{K_A(t)}^{L_A} e^{2 \rho y} p_t^{L_A}(x,y) \: dy &\lesssim \beta^{1/3} z_A(x) \int_{K_A(t)}^{L_A} e^{\rho y} \alpha(L_A - y) \: dy \\
&= \beta^{1/3} z_A(x) e^{\rho L_A} \int_0^{L_A - K_A(t)} e^{-\rho y} \alpha(y) \: dy \\
&\lesssim  \beta^{1/3} z_A(x) e^{\rho L_A} \int_0^{\infty} e^{-\rho y} \cdot \beta^{1/3} y \: dy \\
&\lesssim \frac{\beta^{2/3}}{\rho^2} e^{\rho L_A} z_A(x),
\end{align*}
which implies the lemma because $\rho^{-2} \beta^{2/3} t \rightarrow \infty$ by (\ref{A2}) and the assumption that $t \gtrsim \beta^{-2/3}$.
\end{proof}

\begin{Lemma}\label{LemI}
Under the assumptions of Lemma \ref{2momprop}, we have $$\textup{I} \lesssim \frac{\beta^{2/3}}{\rho^4}e^{\rho x} e^{\rho L_A}.$$
\end{Lemma}

\begin{proof}
We use Lemma \ref{dapprox1} to bound $p_s^{L_A}(x,z)$ and Lemma \ref{dapprox4} to bound $p_{t-s}^{L_A}(z,y)$.  If $z \geq L_A - l(t)$, $y \geq K_A(t) = L_A - l(t)/2$, and $s \leq t/2$, then $$(L_A - z)^{1/2} + (L_A - y)^{1/2} \leq 2 l(t)^{1/2} = \frac{2}{\sqrt{33}} \beta^{1/2} t \leq \frac{4}{\sqrt{33}} \beta^{1/2}(t - s).$$  Therefore, using that $\beta^{-2/3} \lesssim t$, the expression
$$(2 \beta)^{1/6}[(L_A - z)^{1/2} + (L_A - y)^{1/2}] - 2^{-1/3} \beta^{2/3} (t - s)$$
is bounded above by a negative constant, so we can apply (\ref{densapproxrough}).  We get
\begin{multline*}
\textup{I} \lesssim \int_0^{\rho^{-2}} \int_{L_A - l(t)}^{L_A} \frac{1}{\sqrt{s}} \exp \Big( \rho x - \rho z - \frac{(z - x)^2}{2s} - \frac{\rho^2 s}{2} + \beta L_A s \Big) \\
\times \bigg( \int_{K_A(t)}^{L_A} \beta^{1/3} e^{-\beta A (t-s)/\rho} e^{\rho z} \alpha(L_A - z) \alpha(L_A - y) \: dy  \bigg)^2 \: dz \: ds.
\end{multline*}
Note that $\beta L_A s$ and $\rho^2s/2$ are both bounded above by constants.  Using also that $A \geq 0$, we get
$$\textup{I} \lesssim \beta^{2/3} e^{\rho x} \int_0^{\rho^{-2}} \int_{L_A - l(t)}^{L_A} \frac{1}{\sqrt{s}} e^{\rho z} e^{-(z - x)^2/2s} \alpha(L_A - z)^2 \bigg( \int_{K_A(t)}^{L_A} \alpha(L_A - y) \: dy \bigg)^2 \: dz \: ds.$$ Now applying Lemma \ref{alphaintlem} with $k = 1$ gives
$$\textup{I} \lesssim e^{\rho x} \int_0^{\rho^{-2}} \int_{L_A - l(t)}^{L_A} \frac{1}{\sqrt{s}} e^{\rho z} e^{-(z - x)^2/2s} \alpha(L_A - z)^2 \: dz \: ds.$$
Next, we reverse the roles of $z$ and $L_A - z$ and use that $\alpha(z) \lesssim \beta^{1/3} z$ by (\ref{airyder}) and (\ref{airyder2}) to get
\begin{equation}\label{10}
\textup{I} \lesssim \beta^{2/3} e^{\rho x} e^{\rho L_A} \int_0^{\rho^{-2}} \frac{1}{\sqrt{s}} \int_0^{\infty} z^2 e^{-\rho z} e^{-(z - (L_A - x))^2/2s} \: dz \: ds.
\end{equation}

To evaluate the inner integral, we apply Lemma \ref{BMmoments} with $k = 2$, $t = 2s$, and $a = L_A - x$.
We now split the argument into two cases depending on the value of $x$.  First, suppose $L_A - x \geq \rho^{-1}$.  Then, because $s \leq \rho^{-2}$, we have $L_A - x \geq \rho s$, so we can apply (\ref{BMmom1}).  Noting also that in this case we have $(L_A - x)^2 \geq \rho^{-2} \geq s$, we have
$$\int_0^{\infty} z^2 e^{-\rho z} e^{-(z - (L_A - x))^2/2s} \: dz \asymp s^{1/2} \big((L_A - x - \rho s)^2 + s \big) e^{\rho^2 s/2 - \rho(L_A - x)} \lesssim s^{1/2}(L_A - x)^2 e^{-\rho(L_A - x)}$$ and therefore
\begin{equation}\label{11}
\int_0^{\rho^{-2}} \frac{1}{\sqrt{s}} \int_0^{\infty} z^2 e^{-\rho z} e^{-(z - (L_A - x))^2/2s} \: dz \: ds \lesssim \frac{1}{\rho^4} \cdot \rho^{2} (L_A - x)^2 e^{-\rho (L_A - x)} \lesssim \frac{1}{\rho^4}.
\end{equation}
Next, suppose instead that $L_A - x < \rho^{-1}$.  This time, we must use (\ref{BMmom1}) when $s \leq (L_A - x)/\rho$ and (\ref{BMmom2}) when $s > (L_A - x)/\rho$ to get
\begin{align}\label{12}
&\int_0^{\rho^{-2}} \frac{1}{\sqrt{s}} \int_0^{\infty} z^2 e^{-\rho z} e^{-(z - (L_A - x))^2/2s} \: dz \: ds \nonumber \\
&\hspace{.2in}\asymp \int_0^{(L_A - x)/\rho} \big( (L_A - x - \rho s)^2 + s \big) e^{\rho^2 s/2 - \rho(L_A - x)} \: ds + \int_{(L_A - x)/\rho}^{\rho^{-2}} \frac{s^{5/2}e^{-(L_A - x)^2/2s}}{(\rho s - (L_A - x))^3 + s^{3/2}} \: ds \nonumber \\
&\hspace{.2in}\lesssim \int_0^{(L_A - x)/\rho} \big( (L_A - x)^2 + s \big) \: ds + \int_{(L_A - x)/\rho}^{\rho^{-2}} s \: ds \nonumber \\
&\hspace{.2in}\lesssim \frac{(L_A - x)^3}{\rho} + \frac{(L_A - x)^2}{\rho^2} + \frac{1}{\rho^4} \nonumber \\
&\hspace{.2in}\asymp \frac{1}{\rho^4}.
\end{align}
The result follows from (\ref{10}), (\ref{11}), and (\ref{12}).
\end{proof}

\begin{Lemma}\label{LemII}
Under the assumptions of Lemma \ref{2momprop}, we have
\begin{equation}\label{IIeq}
\textup{II} \lesssim \frac{\beta^{2/3}}{\rho^4}e^{\rho x} e^{\rho L_A}.
\end{equation}
\end{Lemma}

\begin{proof}
We use Lemma \ref{dapprox2} to bound $p_s^{L_A}(x,z)$ and Lemma \ref{dapprox4} to bound $p_{t-s}^{L_A}(z,y)$.  Recall from the proof of Lemma \ref{LemI} that when $z \geq L_A - l(t)$, $y \geq K_A(t)$, and $s \leq t/2$, we can apply (\ref{densapproxrough}) to get
\begin{multline*}
\textup{II} \lesssim \int_{\rho^{-2}}^{b} \int_{L_A - l(t)}^{L_A} \frac{(L_A - x)(L_A - z)}{s^{3/2}} \exp \bigg( \rho x - \rho z - \frac{(z - x)^2}{2s} - \frac{\rho^2 s}{2} + \beta L_A s \bigg) \\
\times \bigg( \int_{K_A(t)}^{L_A} \beta^{1/3} e^{-\beta A (t-s)/\rho} e^{\rho z} \alpha(L_A - z) \alpha(L_A - y) \: dy \bigg)^2 \: dz \: ds.
\end{multline*}
It follows from (\ref{LAdef}) that when $s \leq 8\beta^{-2/3}$, the quantity $\beta L_A s - \rho^2s/2$ is bounded above by a positive constant.  Using also that $A \geq 0$, we get
\begin{multline*}
\textup{II} \lesssim \beta^{2/3} e^{\rho x} (L_A - x) \int_{\rho^{-2}}^{b} \frac{1}{s^{3/2}} \int_{L_A - l(t)}^{L_A} e^{\rho z - (z - x)^2/2s} \\
\times (L_A - z) \alpha(L_A - z)^2 \bigg(\int_{K_A(t)}^{L_A} \alpha(L_A - y) \: dy \bigg)^2 \: dz \: ds.
\end{multline*} 
Next, we apply Lemma \ref{alphaintlem} with $k = 1$, interchange the roles of $z$ and $L_A - z$, and use that $\alpha(z) \lesssim \beta^{1/3} z$ to get
\begin{equation}\label{prelim2}
\textup{II} \lesssim \beta^{2/3} e^{\rho x} e^{\rho L_A} (L_A - x) \int_{\rho^{-2}}^{b} \frac{1}{s^{3/2}} \int_0^{\infty} z^3 e^{-\rho z} e^{-(z - (L_A - x))^2/2s} \: dz \: ds.
\end{equation}
To evaluate the double integral, we will apply Lemma \ref{BMmoments} with $k = 3$, $t = 2s$, and $a = L - x$.  This will involve considering two cases, depending on the value of $x$.

First, suppose $L_A - x \leq 1/\rho$.  Then, when $s > \rho^{-2}$, we have $L_A - x < \rho s$.  Therefore, we apply (\ref{BMmom2}), discarding the $e^{-a^2/t}$ term there, to get
\begin{align*}
\int_{\rho^{-2}}^{b} \frac{1}{s^{3/2}} \int_0^{\infty} z^3 e^{-\rho z} e^{-(z - (L_A - x))^2/2s} \: dz \: ds &\lesssim \int_{\rho^{-2}}^{b} \frac{s^{5/2}}{(\rho s - (L_A - x))^4 + s^2} \: ds \\
&\lesssim \int_{\rho^{-2}}^{2 \rho^{-2}} s^{1/2} \: ds + \int_{2 \rho^{-2}}^{8 \beta^{2/3}} \frac{1}{\rho^4 s^{3/2}} \: ds \\
&\lesssim \frac{1}{\rho^3}.
\end{align*}
Combining this with (\ref{prelim2}) and using that $L_A - x \leq 1/\rho$, we get that (\ref{IIeq}) holds in this case.

Next, suppose $L_A - x > 1/\rho$.  We split the double integral in (\ref{prelim2}) into three pieces, denoted $J_1$, $J_2$, and $J_3$, depending on whether $\rho^{-2} \leq s \leq (L_A - x)/\rho$, $(L_A - x)/\rho < s < 2(L_A - x)/\rho$, or $2(L_A - x)/\rho \leq s \leq 8 \beta^{-2/3}$ respectively.  When $s \leq (L_A - x)/\rho$, we can apply (\ref{BMmom1}) to get
$$J_1 \asymp \int_{\rho^{-2}}^{(L_A - x)/\rho} \frac{1}{s^{3/2}} \cdot s^{1/2} \big( (L_A - x - \rho s)^3 + s^{3/2} \big) e^{\rho^2s/2 - \rho(L_A - x)} \: ds.$$
Now using the bound $1/s \leq \rho^2$ and then making the substitution $u = ((L_A - x)/\rho - s) \rho^2/2$, so that $ds/du = -2/\rho^2$, we get
\begin{align}\label{J1bound}
J_1 &\lesssim \rho^2 e^{-\rho(L_A - x)} \int_{\rho^{-2}}^{(L_A - x)/\rho}  \big( (L_A - x - \rho s)^3 + s^{3/2} \big) e^{\rho^2s/2} \: ds \nonumber \\
&\leq \rho^2 e^{-\rho(L_A - x)} \int_{\rho^{-2}}^{(L_A - x)/\rho}  \bigg( (L_A - x - \rho s)^3 + \Big( \frac{L_A - x}{\rho} \Big)^{3/2} \bigg) e^{\rho^2s/2} \: ds \nonumber \\
&\leq \rho^2 e^{-\rho(L_A - x)} \int_0^{\infty} \bigg( \Big( \frac{2u}{\rho} \Big)^3 + \Big( \frac{L_A - x}{\rho} \Big)^{3/2} \bigg) e^{(\rho(L_A - x)/2) - u} \cdot \frac{2}{\rho^2} \: du \nonumber \\
&\lesssim e^{-\rho(L_A - x)/2} \int_0^{\infty} \bigg(\frac{u^3}{\rho^3} + \frac{(L_A - x)^{3/2}}{\rho^{3/2}} \bigg) e^{-u} \: du \nonumber \\
&\lesssim e^{-\rho(L_A - x)/2} \bigg(\frac{1}{\rho^3} + \frac{(L_A - x)^{3/2}}{\rho^{3/2}}\bigg) \nonumber \\
&\lesssim \frac{1}{\rho^4 (L_A - x)} \cdot (\rho(L_A - x))^{5/2} e^{-\rho(L_A - x)/2}.
\end{align}
When $s > (L_A - x)/\rho$, we instead apply (\ref{BMmom2}) and get
\begin{align}\label{J2bound}
J_2 &\asymp \int_{(L_A - x)/\rho}^{2(L_A - x)/\rho} \frac{s^{5/2}}{(\rho s - (L_A - x))^4 + s^2} e^{-(L_A - x)^2/2s} \: ds \nonumber \\
&\leq \int_{(L_A - x)/\rho}^{2(L_A - x)/\rho} s^{1/2} e^{-(L_A - x)^2/2s} \: ds \nonumber \\
&\lesssim \frac{(L_A - x)^{3/2}}{\rho^{3/2}} e^{-\rho (L_A - x)/4} \nonumber \\
&= \frac{1}{\rho^4 (L_A - x)} \cdot (\rho(L_A - x))^{5/2} e^{-\rho(L_A - x)/4}.
\end{align}
Also, using that $\rho s - (L_A - x) \asymp \rho s$ when $s \geq 2(L_A - x)/\rho$, we have
$$J_3 \asymp \int_{2(L_A - x)/\rho}^{8 \beta^{-2/3}} \frac{s^{5/2}}{(\rho s - (L_A - x))^4 + s^2} e^{-(L_A - x)^2/2s} \: ds \lesssim \frac{1}{\rho^4} \int_{2(L_A - x)/\rho}^{8 \beta^{-2/3}} \frac{1}{s^{3/2}} e^{(L_A - x)^2/2s} \: ds.$$  Therefore, by Lemma \ref{32int},
\begin{equation}\label{J3bound}
J_3 \lesssim \frac{1}{\rho^4(L_A - x)}.
\end{equation}
It follows from (\ref{J1bound}), (\ref{J2bound}), and (\ref{J3bound}) that
$$J_1 + J_2 + J_3 \lesssim \frac{1}{\rho^4(L_A - x)},$$ which, in combination with (\ref{prelim2}), implies that (\ref{IIeq}) also holds when $L_A - x > 1/\rho$.
\end{proof}

\begin{Lemma}\label{LemIII}
Under the assumptions of Lemma \ref{2momprop}, we have
\begin{equation}\label{IIIeq}
\textup{III} \lesssim \frac{\beta^{2/3} e^{\rho L_A}}{\rho^4} \big( e^{\rho x} + \beta^{2/3} t z_A(x) \big).
\end{equation}
\end{Lemma}

\begin{proof}
We may, and will, assume that $8 \beta^{-2/3} < t/2$, as otherwise the term $\textup{III}$ is zero.
Write $$m = \max \bigg\{8 \beta^{-2/3}, 8 \sqrt{\frac{L_A - x}{\beta}} \bigg\}.$$  Now define
$$\textup{III}_1 = \int_{m}^{t/2} \int_{L_A - \beta s^2/64}^{L_A} p_s^{L_A}(x,z) \bigg( \int_{K_A(t)}^{L_A} e^{\rho y} p_{t-s}^{L_A}(z,y) \: dy \bigg)^2 \: dz \: ds.$$
Note that if $s \geq m$, then $L_A - x \leq \beta s^2/64$.  Therefore, if $s \geq m$ and $L_A - z \leq \beta s^2/64$, then
\begin{equation}\label{condfor4}
(2 \beta)^{1/6}[(L_A - x)^{1/2} + (L_A - z)^{1/2}] - 2^{-1/3} \beta^{2/3} s \leq \beta^{2/3} s \bigg( \frac{2^{1/6}}{4} - 2^{-1/3} \bigg) < - \frac{\beta^{2/3} s}{8} \leq -1,
\end{equation}
so we can use Lemma \ref{dapprox4} to estimate $p_s^{L_A}(x,z)$.  We can also use Lemma~\ref{dapprox4} to estimate $p_{t-s}^{L_A}(z,y)$ as in the proofs of Lemmas \ref{LemI} and \ref{LemII}.  Therefore, using that $A \geq 0$ along with Lemma \ref{alphaintlem} and the bound $\alpha(z) \lesssim \beta^{1/3} z$, we get
\begin{align}\label{int31}
\textup{III}_1 &\lesssim \int_{m}^{t/2} \int_{L_A - \beta s^2/64}^{L_A} \beta^{1/3} e^{\rho x} \alpha(L_A - x) e^{-\rho z} \alpha(L_A - z) \nonumber \\
&\hspace{.5in}\times \bigg( \int_{K_A(t)}^{L_A} \beta^{1/3} e^{\rho z} \alpha(L_A - z) \alpha(L_A - y) \: dy \bigg)^2 \: dz \: ds \nonumber \\
&= \beta z_A(x) \int_{m}^{t/2} \int_{L_A - \beta s^2/64}^{L_A} e^{\rho z} \alpha(L_A - z)^3 \bigg( \int_{K_A(t)}^{L_A} \alpha(L_A - y) \: dy \bigg)^2 \: dz \: ds \nonumber \\
&\lesssim \beta^{1/3} z_A(x) e^{\rho L_A} \int_{m}^{t/2} \int_{0}^{\beta s^2/64} e^{-\rho z} \beta z^3 \: dz \: ds \nonumber \\
&\lesssim \frac{\beta^{4/3}}{\rho^4} z_A(x) e^{\rho L_A} t.
\end{align}

Next, we consider the case in which $s \geq m$ but $L_A - z > \beta s^2/64$.  Define
$$\textup{III}_2 = \int_{m}^{t/2} \int_{L_A - l(t)}^{L_A - \beta s^2/64} p_s^{L_A}(x,z) \bigg( \int_{K_A(t)}^{L_A} e^{\rho y} p_{t-s}^{L_A}(z,y) \: dy \bigg)^2 \: dz \: ds.$$
In this case, we use Lemma~\ref{dapprox2} to bound $p_s^{L_A}(x,z)$ and Lemma \ref{dapprox4} to bound $p_{t-s}^{L_A}(z,y)$.  Using also Lemma \ref{alphaintlem}, we have
\begin{align}\label{int32a}
\textup{III}_2 &\lesssim \int_{m}^{t/2} \int_{L_A - l(t)}^{L_A - \beta s^2/64} \frac{(L_A - x)(L_A - z)}{s^{3/2}} \exp \bigg( \rho x - \rho z - \frac{\rho^2 s}{2} + \beta L_A s \bigg) \nonumber \\
&\hspace{2in}\times \bigg( \int_{K_A(t)}^{L_A} \beta^{1/3} e^{\rho z} \alpha(L_A - z) \alpha(L_A - y) \: dy \bigg)^2 \: dz \: ds \nonumber \\
&= \beta^{2/3} e^{\rho x}e^{\rho L_A}(L_A - x) \int_m^{t/2} \frac{e^{\beta L_A s - \rho^2s/2}}{s^{3/2}} \int_{\beta s^2/64}^{l(t)} e^{-\rho z} z \alpha(z)^2 \bigg( \int_{K_A(t)}^{L_A} \alpha(L_A - y) \: dy \bigg)^2 \: dz \: ds \nonumber \\
&\lesssim e^{\rho x}e^{\rho L_A}(L_A - x) \int_m^{t/2} \frac{e^{\beta L_A s - \rho^2s/2}}{s^{3/2}} \bigg( \int_{\beta s^2/64}^{l(t)} e^{-\rho z} z \alpha(z)^2 \: dz \bigg) \: ds.
\end{align}
Because $A \geq 0$, we have
\begin{equation}\label{int32b}
e^{\beta L_A s - \rho^2 s/2} \leq e^{-2^{-1/3} \gamma_1 \beta^{2/3} s}.
\end{equation}
Because the function $\alpha$ is bounded and $\beta s^2 \gg 1/\rho$ by (\ref{A1}) whenever $s \gtrsim \beta^{-2/3}$, it follows from (\ref{expint1}) when $k = 1$ that
\begin{equation}\label{int32c}
\int_{\beta s^2/64}^{l(t)} e^{-\rho z} z \alpha(z)^2 \: dz \lesssim \int_{\beta s^2/64}^{\infty} e^{-\rho z} z \: dz \lesssim \frac{\beta s^2}{\rho} e^{-\rho \beta s^2/64}.
\end{equation}
Combining (\ref{int32a}), (\ref{int32b}), and (\ref{int32c}), and then using that $\rho \beta s^2 \gtrsim \rho \beta^{1/3} s \gg \beta^{2/3} s$ by (\ref{A1}) when $s \geq 8 \beta^{-2/3}$, we get
\begin{align}\label{int32d}
\textup{III}_2 &\lesssim \frac{\beta}{\rho} e^{\rho x} e^{\rho L_A} (L_A - x) \int_m^{\infty} s^{1/2} \exp \bigg(-2^{-1/3} \gamma_1 \beta^{2/3} s - \frac{\rho \beta s^2}{64} \bigg) \: ds \nonumber \\
&\lesssim \frac{\beta}{\rho} e^{\rho x} e^{\rho L_A} (L_A - x) \int_m^{\infty} s^{1/2} \exp \bigg(- \frac{\rho \beta s^2}{128} \bigg) \: ds.
\end{align}
By (\ref{expint2}) with $k = 1/2$,
\begin{equation}\label{int32e}
\int_m^{\infty} s^{1/2} \exp \bigg(- \frac{\rho \beta s^2}{128} \bigg) \: ds \lesssim \frac{e^{-\rho \beta m^2/128}}{m^{1/2} \rho \beta}.
\end{equation}
We now claim that
\begin{equation}\label{3claim}
\frac{\rho^2 (L_A - x)}{\beta^{2/3} m^{1/2}} e^{-\rho \beta m^2/128} \rightarrow 0.
\end{equation}
It will then follow from (\ref{int32d}), (\ref{int32e}), and (\ref{3claim}) that
\begin{equation}\label{int32}
\textup{III}_2 \ll \frac{\beta^{2/3}}{\rho^4}e^{\rho x} e^{\rho L_A}.
\end{equation}
To prove (\ref{3claim}), we consider two cases.  First, suppose $L_A - x \leq \beta^{-1/3}$, so that $m = 8 \beta^{-2/3}$.  Then (\ref{A1}) implies that
$$\frac{\rho^2 (L_A - x)}{\beta^{2/3} m^{1/2}} e^{-\rho \beta m^2/128} \lesssim \bigg( \frac{\rho}{\beta^{1/3}} \bigg)^2 e^{-\rho \beta^{-1/3}/2} \rightarrow 0.$$  Alternatively, suppose $L_A - x > \beta^{-1/3}$.  Then $m = 8 \sqrt{(L_A - x)/\beta}$ and we have
$$\frac{\rho^2 (L_A - x)}{\beta^{2/3} m^{1/2}} e^{-\rho \beta m^2/128} \lesssim \frac{\rho^2 (L_A - x)^{3/4}}{\beta^{5/12}} e^{-\rho(L_A - x)/2} = \bigg( \frac{\rho}{\beta^{1/3}} \bigg)^{5/4} [\rho(L_A - x)]^{3/4} e^{-\rho (L_A - x)/2}.$$  Using that $L_A - x > \beta^{-1/3}$ and that the function $x \mapsto x^{3/4} e^{-x/2}$ is decreasing for sufficiently large $x$, we have
$$\bigg( \frac{\rho}{\beta^{1/3}} \bigg)^{5/4} [\rho(L_A - x)]^{3/4} e^{-\rho (L_A - x)/2} \lesssim \bigg( \frac{\rho}{\beta^{1/3}} \bigg)^{5/4} \bigg( \frac{\rho}{\beta^{1/3}} \bigg)^{3/4} e^{-\rho \beta^{-1/3}/2} \rightarrow 0,$$
and again (\ref{3claim}) holds.

It remains to consider the case in which $8 \beta^{-2/3} \leq s < m$, which is possible only when $L_A - x > \beta^{-1/3}$ and $m = 8 \sqrt{(L_A - x)/\beta}$.  Define
$$\textup{III}_3 = \int_{8 \beta^{-2/3}}^m \int_{L_A - l(t)}^{L_A} p_s^{L_A}(x,z) \bigg( \int_{K_A(t)}^{L_A} e^{\rho y} p_{t-s}^{L_A}(z,y) \: dy \bigg)^2 \: dz \: ds.$$
 We again use Lemma~\ref{dapprox2} to bound $p_s^{L_A}(x,z)$  and Lemma~\ref{dapprox4} to bound $p_{t-s}^{L_A}(z,y)$.  Using also Lemma~\ref{alphaintlem} and the bound $\alpha(z) \lesssim \beta^{1/3} z$, we get
\begin{align}\label{int33a}
\textup{III}_3 &\leq \int_{8 \beta^{-2/3}}^m \int_{L_A - l(t)}^{L_A} \frac{(L_A - x)(L_A - z)}{s^{3/2}} \exp \bigg( \rho x - \rho z - \frac{(z - x)^2}{2s} - \frac{\rho^2 s}{2} + \beta L_A s \bigg) \nonumber \\
&\hspace{.5in}\times \bigg( \int_{K_A(t)}^{L_A} \beta^{1/3} e^{\rho z} \alpha(L_A - z) \alpha(L_A - y) \: dy \bigg)^2 \: dz \: ds \nonumber \\
&\lesssim e^{\rho x} (L_A - x) \int_{8 \beta^{-2/3}}^m \frac{e^{\beta L_A s - \rho^2s/2}}{s^{3/2}} \int_{L_A - l(t)}^{L_A} e^{\rho z} (L_A - z) \alpha(L_A - z)^2 e^{-(z - x)^2/2s} \: dz \: ds \nonumber \\
&= e^{\rho x} e^{\rho L_A} (L_A - x) \int_{8 \beta^{-2/3}}^m \frac{e^{\beta L_A s - \rho^2s/2}}{s^{3/2}} \int_0^{l(t)} e^{-\rho z} z \alpha(z)^2 e^{-(z - (L_A - x))^2/2s} \: dz \: ds \nonumber \\
&\lesssim \beta^{2/3} e^{\rho x} e^{\rho L_A} (L_A - x) \int_{8 \beta^{-2/3}}^m \frac{e^{\beta L_A s - \rho^2s/2}}{s^{3/2}} \bigg( \int_0^{\infty} e^{-\rho z} z^3 e^{-(z - (L_A - x))^2/2s} \: dz \bigg) \: ds.
\end{align}

We now estimate the inner integral using Lemma \ref{BMmoments} with $k = 3$, $t = 2s$, and $a = L - x$.
We need to consider three cases.  First, suppose $s \geq 2(L_A - x)/\rho$.  Then $L_A - x \leq \frac{1}{2} \rho s$, so we use (\ref{BMmom2}) and the fact that $s \leq m$ to get
\begin{equation}\label{int33b}
\int_0^{\infty} e^{-\rho z} z^3 e^{-(z - (L_A - x))^2/2s} \: dz \asymp \frac{s^4}{(\rho s - (L_A - x))^4 + s^2} e^{-(L_A - x)^2/2s} \lesssim \frac{1}{\rho^4} e^{-(L_A - x)^2/2m}.
\end{equation}
Combining (\ref{int33a}) with (\ref{int33b}), and using (\ref{int32b}) again along with the fact that $m = 8 \sqrt{(L_A - x)/\beta}$, we get
\begin{align}\label{IIIsbig}
&(L_A - x) \int_{8 \beta^{-2/3} \vee 2(L_A - x)/\rho}^m \frac{e^{\beta L_A s - \rho^2s/2}}{s^{3/2}} \bigg( \int_0^{\infty} e^{-\rho z} z^3 e^{-(z - (L_A - x))^2/2s} \: dz \bigg) \: ds \nonumber \\
&\hspace{1in}\lesssim \frac{L_A - x}{\rho^4} e^{-2^{-1/3} \gamma_1 \beta^{2/3} m} e^{-(L_A - x)^2/2m} \int_{8 \beta^{-2/3}}^m \frac{1}{s^{3/2}} \: ds \nonumber \\
&\hspace{1in}\lesssim \frac{1}{\rho^4} \cdot \beta^{1/3}(L_A - x) e^{-2^{8/3} \gamma_1 \beta^{1/6}(L_A - x)^{1/2}} e^{-\beta^{1/2}(L_A - x)^{3/2}/16} \nonumber \\
&\hspace{1in}\lesssim \frac{1}{\rho^4}.
\end{align}
Next, suppose $(L_A - x)/\rho \leq s < 2(L_A - x)/\rho$.  Then $L_A - x \leq \rho s$, so again we use (\ref{BMmom2}).  This time, we keep the $s^2$ term in the denominator, and we get
$$\int_0^{\infty} e^{-\rho z} z^3 e^{-(z - (L_A - x))^2/2s} \: dz \asymp \frac{s^4}{(\rho s - (L_A - x))^4 + s^2} e^{-(L_A - x)^2/2s} \lesssim s^2 e^{-\rho (L_A - x)/4}.$$
Combining this result with (\ref{int33a}) and (\ref{int32b}), we get
\begin{align}\label{IIIsmiddle}
&(L_A - x) \int_{(L_A - x)/\rho}^{2(L_A - x)/\rho} \frac{e^{\beta L_A s - \rho^2s/2}}{s^{3/2}} \bigg( \int_0^{\infty} e^{-\rho z} z^3 e^{-(z - (L_A - x))^2/2s} \: dz \bigg) \: ds \nonumber \\
&\hspace{1in}\lesssim (L_A - x) e^{-2^{2/3} \gamma_1 \beta^{2/3} (L_A - x)/\rho} e^{-\rho(L_A - x)/4} \int_{(L_A - x)/\rho}^{2(L_A - x)/\rho} s^{1/2} \: ds \nonumber \\
&\hspace{1in}\lesssim (L_A - x) \bigg( \frac{L_A - x}{\rho} \bigg)^{3/2} e^{-2^{2/3} \gamma_1 (\beta^{2/3}/\rho^2) \rho(L_A - x)} e^{-\rho(L_A - x)/4} \nonumber \\
&\hspace{1in}\lesssim \frac{1}{\rho^4} \cdot (\rho(L_A - x))^{5/2} e^{-2^{2/3} \gamma_1 (\beta^{2/3}/\rho^2) \rho(L_A - x)} e^{-\rho(L_A - x)/4} \nonumber \\
&\hspace{1in}\lesssim \frac{1}{\rho^4}.
\end{align}
Now, suppose $s < (L_A - x)/\rho$.  Then $L_A - x > \rho s$, so this time we use (\ref{BMmom1}) to get
$$\int_0^{\infty} e^{-\rho z} z^3 e^{-(z - (L_A - x))^2/2s} \: dz \asymp s^{1/2} \big( (L_A - x - \rho s)^3 + s^{3/2} \big) e^{\rho^2 s/2} e^{-\rho(L_A - x)}.$$
We may assume that $(L_A - x)/\rho > 8 \beta^{-2/3}$, which implies that $\rho(L_A - x) \geq 8 \rho^2 \beta^{-2/3} \rightarrow \infty$ and therefore $s^{3/2}\leq (L_A - x)^{3/2} \rho^{-3/2} \ll (L_A - x)^3$.  It follows that
$$\int_0^{\infty} e^{-\rho z} z^3 e^{-(z - (L_A - x))^2/2s} \: dz \lesssim s^{1/2} (L_A - x)^3 e^{\rho^2 s/2} e^{-\rho(L_A - x)}.$$
Therefore,
\begin{align*}
&(L_A - x) \int_{8 \beta^{-2/3}}^{(L_A - x)/\rho} \frac{e^{\beta L_A s - \rho^2s/2}}{s^{3/2}} \bigg( \int_0^{\infty} e^{-\rho z} z^3 e^{-(z - (L_A - x))^2/2s} \: dz \bigg) \: ds \nonumber \\
&\hspace{1in}\lesssim (L_A - x)^4 e^{-\rho (L_A - x)} \int_{8 \beta^{-2/3}}^{(L_A - x)/\rho} \frac{1}{s} e^{\beta L_A s} \: ds \\
&\hspace{1in}\lesssim (L_A - x)^4 e^{-\rho (L_A - x)} \cdot \frac{\rho}{\beta L_A(L_A - x)} e^{\beta L_A (L_A - x)/\rho}.
\end{align*}
Because $L_A \leq 2 \rho^2/3 \beta$ for sufficiently large $n$, we have $e^{\beta L_A (L_A - x)/\rho} \leq e^{2 \rho (L_A -x )/3}$, and therefore
\begin{align}\label{IIIssmall}
&(L_A - x) \int_{8 \beta^{-2/3}}^{(L_A - x)/\rho} \frac{e^{\beta L_A s - \rho^2s/2}}{s^{3/2}} \bigg( \int_0^{\infty} e^{-\rho z} z^3 e^{-(z - (L_A - x))^2/2s} \: dz \bigg) \: ds \nonumber \\
&\hspace{1in}\lesssim (L_A - x)^4 e^{-\rho (L_A - x)} \cdot \frac{1}{\rho(L_A - x)} e^{2 \rho(L_A - x)/3} \nonumber \\
&\hspace{1in}= \frac{1}{\rho^4} \cdot (\rho(L_A - x))^3 e^{-\rho(L_A - x)/3} \nonumber \\
&\hspace{1in}\lesssim \frac{1}{\rho^4}.
\end{align}
It now follows from (\ref{int33a}), (\ref{IIIsbig}), (\ref{IIIsmiddle}), and (\ref{IIIssmall}) that
\begin{equation}\label{int33}
\textup{III}_3 \lesssim \frac{\beta^{2/3}}{\rho^4} e^{\rho x} e^{\rho L_A}.
\end{equation}
The result follows from (\ref{int31}), (\ref{int32}), and (\ref{int33}).
\end{proof}

\begin{Lemma}\label{LemIV}
Under the assumptions of Lemma \ref{2momprop}, we have $$\textup{IV} \lesssim \frac{\beta^{4/3}}{\rho^4}e^{\rho L_A} t z_A(x).$$
\end{Lemma}

\begin{proof}
We may, and will, assume that $8 \beta^{-2/3} < t/2$, as otherwise the term $\textup{IV}$ is zero.
When $s \geq t/2$, $x \geq K_A(t) = L_A - l(t)/2$, and $z \geq L_A - l(t)$, we have
$$(L_A - x)^{1/2} + (L_A - z)^{1/2} \leq 2 l(t)^{1/2} = \frac{2}{\sqrt{33}} \beta^{1/2} t \leq \frac{4}{\sqrt{33}} \beta^{1/2} s.$$  Therefore, using that $\beta^{-2/3} \lesssim s$, we see that $$(2 \beta)^{1/6}[(L_A - x)^{1/2} + (L_A - z)^{1/2}] - 2^{-1/3} \beta^{2/3} s$$ is bounded above by a negative constant, so we can apply Lemma \ref{dapprox4} to approximate $p_s^{L_A}(z,y)$ and use (\ref{densapproxrough}).

If $L_A - y \leq \beta(t-s)^2/64$ and $L_A - z \leq \beta (t-s)^2/128$, then reasoning as in (\ref{condfor4}), we have
$$(2 \beta)^{1/6}[(L_A - z)^{1/2} + (L_A - y)^{1/2}] - 2^{-1/3}\beta^{2/3}(t - s) \leq -1,$$ so we can use Lemma \ref{dapprox4} to estimate $p_{t-s}^{L_A}(z,y)$.  We define
$$\textup{IV}_1 = \int_{t/2}^{t - 8 \beta^{-2/3}} \int_{L_A - \beta(t-s)^2/128}^{L_A} p_s^{L_A}(x,z) \bigg( \int_{L_A - \beta (t-s)^2/64}^{L_A} e^{\rho y} p_{t-s}^{L_A}(z,y) \: dy \bigg)^2 \: dz \: ds.$$
Therefore, using also that $A \geq 0$ along with Lemma \ref{alphaintlem} and the bound $\alpha(z) \lesssim \beta^{1/3} z$, we have
\begin{align}\label{int41}
\textup{IV}_1 &\lesssim \int_{t/2}^{t - 8 \beta^{-2/3}} \int_{L_A - \beta(t-s)^2/128}^{L_A} \beta^{1/3} e^{\rho x} \alpha(L_A - x) e^{-\rho z} \alpha(L_A - z) \nonumber \\
&\hspace{.5in}\times \bigg( \int_{L_A - \beta(t-s)^2/64}^{L_A} \beta^{1/3} e^{\rho z} \alpha(L_A - z)\alpha(L_A - y) \: dy \bigg)^2 \: dz \: ds \nonumber \\
&= \beta z_A(x) \int_{t/2}^{t - 8 \beta^{-2/3}} \int_{L_A - \beta(t-s)^2/128}^{L_A} e^{\rho z} \alpha(L_A - z)^3 \bigg( \int_{L_A - \beta(t-s)^2/64}^{L_A} \alpha(L_A - y) \: dy \bigg)^2 \: dz \: ds \nonumber \\
&\lesssim \beta^{1/3} z_A(x) e^{\rho L_A} \int_{t/2}^{t - 8 \beta^{-2/3}} \int_0^{\beta(t - s)^2/128} e^{-\rho z} \beta z^3 \: dz \: ds \nonumber \\
&\lesssim \frac{\beta^{4/3}}{\rho^4} z_A(x) e^{\rho L_A} t.
\end{align}

Next, we consider the case in which $L_A - z > \beta(t-s)^2/128$.  Define
$$\textup{IV}_2 = \int_{t/2}^{t - 8 \beta^{-2/3}} \int_{L_A - l(t)}^{L_A - \beta(t-s)^2/128} p_s^{L_A}(x,z) \bigg( \int_{K_A(t)}^{L_A} e^{\rho y} p_{t-s}^{L_A}(z,y) \: dy \bigg)^2 \: dz \: ds.$$
We use Lemma \ref{dapprox4} to bound $p^{L_A}_s(x,z)$ as before and Lemma \ref{dapprox2} to bound $p^{L_A}_{t-s}(z,y)$.  We get
\begin{align*}
\textup{IV}_2 &\lesssim \int_{t/2}^{t - 8 \beta^{-2/3}} \int_{L_A - l(t)}^{L_A - \beta(t-s)^2/128} \beta^{1/3} e^{\rho x} \alpha(L_A - x) e^{-\rho z} \alpha(L_A - z) \\
&\hspace{.2in}\times \bigg(\int_{K_A(t)}^{L_A} \frac{(L_A - z)(L_A - y)}{(t-s)^{3/2}} \exp \Big(\rho z - \frac{(z - y)^2}{2(t-s)} - \frac{\rho^2(t-s)}{2} + \beta L_A(t-s) \Big) \: dy \bigg)^2 \: dz \: ds \\
&= \beta^{1/3}z_A(x) \int_{t/2}^{t-8 \beta^{-2/3}} \frac{1}{(t-s)^2} e^{(2 \beta L_A - \rho^2)(t-s)} \int_{L_A - l(t)}^{L_A - \beta(t-s)^2/128} e^{\rho z} (L_A - z)^2 \alpha(L_A - z) \\
&\hspace{.2in}\times \bigg( \int_{K_A(t)}^{L_A} \frac{1}{\sqrt{t-s}} (L_A - y) e^{-(z - y)^2/2(t-s)} \: dy \bigg)^2 \: dz \: ds.
\end{align*}
Now interchanging the roles of $s$ and $t-s$, $z$ and $L_A - z$, and $y$ and $L_A - y$, we get
\begin{equation}\label{IVex}
\textup{IV}_2 \lesssim \beta^{1/3}z_A(x)e^{\rho L_A} \int_{8 \beta^{-2/3}}^{t/2} \frac{e^{(2 \beta L_A - \rho^2)s}}{s^2} \int_{\beta s^2/128}^{l(t)} e^{-\rho z} z^2 \alpha(z) \bigg( \int_0^{l(t)/2} \frac{y}{\sqrt{s}} e^{-(z - y)^2/2s} \: dy \bigg)^2 \: dz \: ds.
\end{equation}
Note that if $W$ has a normal distribution with mean $z \geq 0$ and variance $s$, then
\begin{equation}\label{Wnorm}
\int_0^{l(t)/2} \frac{y}{\sqrt{s}} e^{-(z-y)^2/2s} \: dy \leq \int_0^{\infty} \frac{y}{\sqrt{s}} e^{-(z-y)^2/2s} \: dy = \sqrt{2 \pi} \: \E[\max\{0, W\}] \lesssim z + \sqrt{s}.
\end{equation}
Also, because the function $\alpha$ is bounded and $\beta s^2 \gg 1/\rho$ when $s \gtrsim \beta^{-2/3}$ by (\ref{A1}), it follows from (\ref{expint1}) that
$$\int_{\beta s^2/128}^{l(t)} e^{-\rho z} z^2 \alpha(z) (z + \sqrt{s})^2 \: dz \lesssim \int_{\beta s^2/128}^{\infty} e^{-\rho z}(z^4 + z^2 s) \: dz \lesssim e^{-\rho \beta s^2/128} \bigg(\frac{(\beta s^2)^4}{\rho} + \frac{(\beta s^2)^2 s}{\rho} \bigg).$$
Because $A \geq 0$, we have
\begin{equation}\label{LArho}
e^{(2 \beta L_A - \rho^2)s} \leq e^{-2^{2/3} \gamma_1 \beta^{2/3} s}.
\end{equation}
Combining these bounds, we get
\begin{align*}
\textup{IV}_2 &\lesssim \beta^{1/3}z_A(x)e^{\rho L_A} \int_{8 \beta^{-2/3}}^{t/2} \frac{1}{s^2} e^{-2^{2/3} \gamma_1 \beta^{2/3}s} e^{-\rho \beta s^2/128} \bigg(\frac{(\beta s^2)^4}{\rho} + \frac{(\beta s^2)^2 s}{\rho} \bigg) \: ds \\
&= \beta^{1/3}z_A(x)e^{\rho L_A} \int_{8 \beta^{-2/3}}^{t/2} \frac{1}{s^2} e^{-2^{2/3} \gamma_1 \beta^{2/3}s} e^{-\rho \beta s^2/128} \bigg(\frac{(\rho \beta s^2)^4}{\rho^5} + \frac{(\rho \beta s^2)^2 s}{\rho^3} \bigg) \: ds.
\end{align*}
Now $\beta^{2/3} s \ll \rho \beta s^2$ when $s \gtrsim \beta^{-2/3}$, which implies that for any $k > 0$,
we have $$e^{-2^{2/3} \gamma_1 \beta^{2/3}s} e^{-\rho \beta s^2/128} (\rho \beta s^2)^k \ll e^{-\rho \beta s^2/256}.$$  It follows, using (\ref{expint2}) with $k = 0$ for the third inequality, that
\begin{align*}
\textup{IV}_2 &\lesssim \beta^{1/3}z_A(x)e^{\rho L_A} \int_{8 \beta^{-2/3}}^{t/2} e^{-\rho \beta s^2/256} \bigg( \frac{1}{s^2 \rho^5} + \frac{1}{s \rho^3} \bigg) \: ds \\
&\lesssim \beta^{1/3}z_A(x)e^{\rho L_A} \bigg( \frac{\beta^{4/3}}{\rho^5} + \frac{\beta^{2/3}}{\rho^3} \bigg) \int_{8 \beta^{-2/3}}^{t/2} e^{-\rho \beta s^2/256} \: ds \\
&\lesssim \beta^{1/3}z_A(x)e^{\rho L_A} \bigg( \frac{\beta^{4/3}}{\rho^5} + \frac{\beta^{2/3}}{\rho^3} \bigg) \cdot e^{\rho \beta^{-1/3}/4} \frac{1}{\rho \beta^{1/3}} \\
&= \frac{\beta^{2/3}}{\rho^4} z_A(x) e^{\rho L_A} \bigg(\frac{\beta^{2/3}}{\rho^2} + 1 \bigg) e^{\rho \beta^{-1/3}/4}.
\end{align*}
Finally, using the assumption that $t \gtrsim \beta^{-2/3}$, we have
\begin{equation}\label{int42}
\textup{IV}_2 \ll \frac{\beta^{2/3}}{\rho^4} z_A(x) e^{\rho L_A} \lesssim \frac{\beta^{4/3}}{\rho^4} e^{\rho L_A} t z_A(x). 
\end{equation}

Now, we consider the remaining case in which $L_A - z \leq \beta (t-s)^2/128$ and $L_A - y > \beta(t-s)^2/64$.  Define
$$\textup{IV}_3 = \int_{t/2}^{t - 8 \beta^{-2/3}} \int_{L_A - \beta(t-s)^2/128}^{L_A} p_s^{L_A}(x,z) \bigg( \int_{K_A(t)}^{L_A - \beta(t-s)^2/64} e^{\rho y} p_{t-s}^{L_A}(z,y) \: dy \bigg)^2 \: dz \: ds.$$
We again use Lemma \ref{dapprox4} to bound $p_s^{L_A}(x,z)$ and Lemma \ref{dapprox2} to bound $p_{t-s}^{L_A}(z,y)$, and following the same steps that led to (\ref{IVex}), we get
$$\textup{IV}_3 \lesssim \beta^{1/3} z_A(x) e^{\rho L_A} \int_{8 \beta^{-2/3}}^{t/2} \frac{e^{(2 \beta L_A - \rho^2)s}}{s^2} \int_0^{\beta s^2/128} e^{-\rho z} z^2 \alpha(z) \bigg( \int_{\beta s^2/64}^{l(t)/2} \frac{y}{\sqrt{s}} e^{-(z-y)^2/2s} \: dy \bigg)^2 \: dz \: ds.$$
Note that if $z \leq \beta s^2/128$ and $y \geq \beta s^2/64$, then letting $\delta = 1/(4 \cdot 128^2)$,
$$e^{-(z - y)^2/2s} = e^{-(z-y)^2/4s} e^{-(z-y)^2/4s} \leq e^{-\delta \beta^2 s^3} e^{-(z-y)^2/4s}.$$
Therefore, reasoning as in (\ref{Wnorm}),
$$\int_{\beta s^2/64}^{l(t)/2} \frac{y}{\sqrt{s}} e^{-(z-y)^2/2s} \: dy \leq e^{-\delta \beta^2 s^3} \int_0^{\infty} \frac{y}{\sqrt{s}} e^{-(z-y)^2/4s} \: dy \lesssim e^{-\delta \beta^2 s^3} (z + \sqrt{s}).$$
Using also that $\alpha(z) \lesssim \beta^{1/3} z$ and (\ref{LArho}), we get
\begin{align*}
\textup{IV}_3 &\lesssim \beta^{1/3} z_A(x) e^{\rho L_A} \int_{8 \beta^{-2/3}}^{t/2} \frac{e^{(2 \beta L_A - \rho^2)s}}{s^2} \int_0^{\beta s^2/128} e^{-\rho z} z^2 \alpha(z) \cdot e^{-2 \delta \beta^2 s^3} (z^2 + s) \: dz \: ds \\
&\lesssim \beta^{2/3} z_A(x) e^{\rho L_A} \int_{8 \beta^{-2/3}}^{t/2} \frac{1}{s^2} e^{-2^{2/3} \gamma_1 \beta^{2/3} s} e^{-2 \delta \beta^2 s^3} \int_0^{\infty} e^{-\rho z} (z^5 + s z^3) \: dz \: ds \\
&\lesssim \beta^{2/3} z_A(x) e^{\rho L_A} \int_{8 \beta^{-2/3}}^{t/2} \frac{1}{s^2} e^{-2^{2/3} \gamma_1 \beta^{2/3} s} e^{-2 \delta \beta^2 s^3} \bigg( \frac{1}{\rho^6} + \frac{s}{\rho^4} \bigg) \: ds.
\end{align*}
When $s \geq 8 \beta^{-2/3}$, we have $s/\rho^4 \gg 1/\rho^6$ by (\ref{A1}), so we may disregard the $1/\rho^6$ term.  Therefore, making the substitution $u = \beta^{2/3} s$ and then using that $\beta^{-2/3} \lesssim t$,
\begin{align}\label{int43}
\textup{IV}_3 &\lesssim \frac{\beta^{2/3}}{\rho^4} z_A(x) e^{\rho L_A} \int_{8 \beta^{-2/3}}^{\infty} \frac{1}{s} e^{-2^{2/3} \gamma_1 \beta^{2/3} s} e^{- 2\delta \beta^2 s^3} \: ds \nonumber \\
&= \frac{\beta^{2/3}}{\rho^4} z_A(x) e^{\rho L_A} \int_8^{\infty} \frac{1}{u} e^{-2^{2/3} \gamma_1 u - 2 \delta u^3} \: du \nonumber \\
&\lesssim \frac{\beta^{2/3}}{\rho^4} z_A(x) e^{\rho L_A} \nonumber \\
&\lesssim \frac{\beta^{4/3}}{\rho^4} e^{\rho L_A} t z_A(x).
\end{align}
The result follows from (\ref{int41}), (\ref{int42}), and (\ref{int43}).
\end{proof}

\begin{Lemma}\label{LemV}
Under the assumptions of Lemma \ref{2momprop}, we have $$\textup{V} \lesssim \frac{\beta^{4/3}}{\rho^4}e^{\rho L_A} t z_A(x).$$
\end{Lemma}

\begin{proof}
As in the proof of Lemma \ref{LemIV}, we can apply Lemma \ref{dapprox4} to estimate $p_s^{L_A}(z,y)$.  To bound $p_{t-s}(z,y)$, we will use either the bound from Lemma \ref{dapprox1} or the bound from Lemma \ref{dapprox2}, whichever is smaller.  Using also that, when $t - s \leq \beta^{-2/3}$, we have $$\exp \bigg( \beta L_A (t - s) - \frac{\rho^2 (t-s)}{2} \bigg) \lesssim 1,$$ we get
\begin{align*}
\textup{V} &\lesssim \int_{t - b}^t \int_{L_A - l(t)}^{L_A} \beta^{1/3} e^{\rho x} \alpha(L_A - x) e^{-\rho z} \alpha(L_A - z) \\
&\hspace{.5in}\times \bigg( \int_{K_A(t)}^{L_A} \min \bigg\{\frac{1}{\sqrt{t - s}}, \frac{(L_A - z)(L_A - y)}{(t - s)^{3/2}} \bigg\} \exp \bigg(\rho z - \frac{(z - y)^2}{2(t-s)} \bigg) \: dy \bigg)^2 \: dz \: ds \\
&\lesssim \beta^{1/3} z_A(x) \int_{t - b}^t \int_{L_A - l(t)}^{L_A} e^{\rho z} \alpha(L_A - z) \\
&\hspace{.5in} \times \bigg( \int_{K_A(t)}^{L_A} \min \bigg\{1, \frac{(L_A - z)(L_A - y)}{t - s} \bigg\} \cdot \frac{1}{\sqrt{t-s}} e^{-(z - y)^2/2(t-s)} \: dy \bigg)^2 \: dz \: ds.
\end{align*}
Interchanging the roles of $s$ and $t - s$, $z$ and $L_A - z$, and $y$ and $L_A - y$, we get
\begin{equation}\label{Vprelim}
\textup{V} \lesssim \beta^{1/3} z_A(x) e^{\rho L_A} \int_{0}^{b} \int_0^{l(t)} e^{-\rho z} \alpha(z) \bigg( \int_0^{l(t)/2} \min \bigg\{1, \frac{yz}{s} \bigg\} \cdot \frac{1}{\sqrt{s}} e^{-(z - y)^2/2s} \: dy \bigg)^2 \: dz \: ds.
\end{equation}  
We now use (\ref{Wnorm}) to estimate the integral with respect to $y$, which yields
\begin{align*}
&\int_0^{l(t)/2} \min \bigg\{1, \frac{yz}{s} \bigg\} \cdot \frac{1}{\sqrt{s}} e^{-(z - y)^2/2s} \: dy \\
&\hspace{.5in}\lesssim \min \bigg\{ \int_0^{l(t)/2} \frac{1}{\sqrt{s}} e^{-(z - y)^2/2s} \: ds, \: \frac{z}{s} \int_0^{l(t)/2} \frac{y}{\sqrt{s}} e^{-(z - y)^2/2s} \: dy \bigg\} \\
&\hspace{.5in}\lesssim \min \bigg\{ 1, \frac{z}{s} (z + \sqrt{s}) \bigg\}.
\end{align*}
Now noting that $z^2/s \geq z/\sqrt{s}$ whenever either of these expressions is larger than one, it follows that
$$\int_0^{l(t)/2} \min \bigg\{1, \frac{yz}{s} \bigg\} \cdot \frac{1}{\sqrt{s}} e^{-(z - y)^2/2s} \: dy \leq \min \bigg\{1, \frac{z^2}{s} \bigg\}.$$
Plugging this result into (\ref{Vprelim}), and using that $\alpha(z) \lesssim \beta^{1/3} z$, we get
\begin{align*}
\textup{V} &\lesssim \beta^{2/3} z_A(x) e^{\rho L_A} \int_0^{b} \int_0^{l(t)} e^{-\rho z} z \min \bigg\{1, \frac{z^4}{s^2} \bigg\} \: dz \: ds \\
&\lesssim \beta^{2/3} z_A(s) e^{\rho L_A} \bigg( \int_0^{\rho^{-2}} \int_0^{\infty} e^{-\rho z} z \: dz \: ds + \int_{\rho^{-2}}^{8 \beta^{-2/3}} \frac{1}{s^2} \int_0^{\infty} e^{-\rho z} z^5 \: dz \: ds \bigg) \\
&\lesssim \beta^{2/3} z_A(s) e^{\rho L_A} \bigg( \int_0^{\rho^{-2}} \frac{1}{\rho^2} \: ds + \frac{1}{\rho^6} \int_{\rho^{-2}}^{8 \beta^{-2/3}} \frac{1}{s^2} \: ds \bigg) \\
&\lesssim \frac{\beta^{2/3}}{\rho^4} z_A(x) e^{\rho L_A}.
\end{align*}
Because $t \gtrsim \beta^{-2/3}$, the result follows.
\end{proof}

\begin{Lemma}
Under the assumptions of Lemma \ref{2momprop}, we have $$\textup{VI} \ll \frac{\beta^{4/3}}{\rho^4}e^{\rho L_A} t z_A(x).$$
\end{Lemma}

\begin{proof}
We use Lemma \ref{dapprox2} to bound both $p_s^{L_A}(x,z)$ and $p_{t-s}^{L_A}(z,y)$ and get
\begin{align*}
\textup{VI} &\lesssim \int_0^t \int_{-\infty}^{L_A - l(t)} \frac{(L_A - x)(L_A - z)}{s^{3/2}} \exp \Big( \rho x - \rho z - \frac{(z - x)^2}{2s} - \frac{\rho^2 s}{2} + \beta L_A s \Big) \\
&\hspace{.2in}\times \bigg( \int_{K_A(t)}^{L_A} \frac{(L_A - z)(L_A - y)}{(t - s)^{3/2}} \exp \Big( \rho z - \frac{(y - z)^2}{2(t-s)} - \frac{\rho^2 (t - s)}{2} + \beta L_A (t - s) \Big) \: dy \bigg)^2 \: dz \: ds \\
&= e^{\rho x}(L_A - x) \int_0^t \frac{1}{s^{3/2}(t - s)^{3}} \exp \Big( \Big( \beta L_A - \frac{\rho^2}{2} \Big)(2t - s) \Big) \\
&\hspace{.2in}\times \int_{-\infty}^{L_A - l(t)} e^{\rho z} (L_A - z)^3 e^{-(x - z)^2/2s} \bigg( \int_{K_A(t)}^{L_A} (L_A - y) e^{-(y - z)^2/2(t-s)} \: dy \bigg)^2 \: dz \: ds.
\end{align*}
Note that if $x > K_A(t)$, $y > K_A(t)$, and $z < L_A - l(t)$, then $(x - z)^2 \geq l(t)^2/4$ and $(y - z)^2 \geq l(t)^2/4$.
Also, recalling that $\gamma_1 < 0$ and $A \geq 0$, we have that for $s \geq 0$, $$\bigg( \beta L_A - \frac{\rho^2}{2} \bigg)(2t - s) = \bigg( - 2^{-1/3} \beta^{2/3} \gamma_1 - \frac{\beta A}{\rho} \bigg)(2t - s) \leq -(2 \beta)^{2/3} \gamma_1 t.$$
It follows that
\begin{align}\label{71}
\textup{VI} &\lesssim e^{\rho x}(L_A - x) e^{-(2 \beta)^{2/3} \gamma_1 t} \bigg( \int_0^t \frac{1}{s^{3/2} ( t - s)^{3}} e^{-l(t)^2/8s} e^{-l(t)^2/4(t-s)} \: ds \bigg) \nonumber \\
&\hspace{.5in}\times \bigg( \int_{-\infty}^{L_A - l(t)} e^{\rho z} (L_A - z)^3 \: dz \bigg) \bigg( \int_{K_A(t)}^{L_A} (L_A - y) \: dy \bigg)^2.
\end{align}
Because $t \gtrsim \beta^{-2/3}$, we have $l(t) \gg 1/\rho$ and therefore by (\ref{expint1}),
\begin{equation}\label{72}
\int_{-\infty}^{L_A - l(t)} e^{\rho z} (L_A - z)^3 \: dz \asymp e^{\rho(L_A - l(t))} \cdot \frac{l(t)^3}{\rho}.
\end{equation}
Also, 
\begin{equation}\label{73}
\int_{K_A(t)}^{L_A} (L_A - y) \: dy = \frac{l(t)^2}{8}.
\end{equation}
Because $\int_0^{\infty} s^{-b} e^{-a/s} \: ds = a^{1-b}\int_0^\infty u^{b-2}e^{-u}\:du \asymp a^{1-b}$ when $b > 1$, we have, using Lemma \ref{32int},
$$\int_0^{t/2} \frac{1}{s^{3/2} ( t - s)^{3}} e^{-l(t)^2/8s} e^{-l(t)^2/4(t-s)} \: ds \lesssim \frac{e^{-l(t)^2/4t}}{t^3} \int_0^{t/2} \frac{1}{s^{3/2}} e^{-l(t)^2/8s} \: ds \lesssim \frac{e^{-l(t)^2/4t}}{t^3 l(t)}$$
and, interchanging the roles of $s$ and $t-s$ to evaluate the integral,
$$\int_{t/2}^t \frac{1}{s^{3/2} ( t - s)^{3}} e^{-l(t)^2/8s} e^{-l(t)^2/4(t-s)} \: ds \lesssim \frac{e^{-l(t)^2/8t}}{t^{3/2}} \int_{t/2}^t \frac{1}{(t-s)^3} e^{-l(t)^2/4(t-s)} \: ds \lesssim \frac{e^{-l(t)^2/8t}}{t^{3/2} l(t)^4}.$$
Because $l(t) \asymp \beta t^2$ and $t \gtrsim \beta^{-2/3}$, we have $t^3 l(t) \lesssim t^{3/2} l(t)^4$.  Therefore, summing the previous two integrals gives
\begin{equation}\label{74}
\int_0^t \frac{1}{s^{3/2} ( t - s)^{3}} e^{-l(t)^2/8s} e^{-l(t)^2/[4(t-s)]} \: ds \lesssim \frac{1}{t^3 l(t)} e^{-l(t)^2/8t}.
\end{equation}
Combining (\ref{71}), (\ref{72}), (\ref{73}), and (\ref{74}), we get
\begin{align}\label{75}
\textup{VI} &\lesssim \frac{z_A(x) l(t)^6}{\rho t^3} \exp \bigg( \rho(L_A - l(t)) - (2 \beta)^{2/3} \gamma_1 t - \frac{l(t)^2}{8t} \bigg) \nonumber \\
&= \frac{\beta^{4/3}}{\rho^4} e^{\rho L_A} t z_A(x) \cdot \frac{l(t)^2}{\rho \beta^{4/3} t^4} \cdot \rho^4 l(t)^4 e^{-\rho l(t)} \cdot \exp \bigg( - (2 \beta)^{2/3} \gamma_1 t - \frac{l(t)^2}{8t} \bigg).
\end{align}
We will show that the second and third factors on the right-hand side of (\ref{75}) tend to zero, while the fourth is bounded above, which will complete the proof of the lemma.  Because $l(t) \asymp \beta t^2$ and $t \gtrsim \beta^{-2/3}$, we have
$$\frac{l(t)^2}{\rho \beta^{4/3} t^4} \asymp \frac{\beta^{2/3}}{\rho} \rightarrow 0.$$
Also, $\rho l(t) \asymp \rho \beta t^2 \gtrsim \rho \beta^{-1/3} \rightarrow \infty$ by (\ref{A1}), so
$\rho^4 l(t)^4 e^{-\rho l(t)} \rightarrow 0.$  Finally, $l(t)^2/8t \asymp \beta^2 t^3$, which implies that the fourth factor is bounded because $t \gtrsim \beta^{2/3}$.
\end{proof}

\subsection{Proof of Lemma \ref{mainvarg}}

Standard second moment calculations (again see p. 146 of \cite{inw}) give
\begin{multline*}
\E_x \bigg[ \bigg( \sum_{i=1}^{N(t)} g\big( X_i(t) \sqrt{\beta/\rho} \big) \bigg)^2 \bigg] = \int_{-\infty}^{L} g \big( y \sqrt{\beta/\rho} \big)^2 p^{L}_t(x,y) \: dy \\
+ 2 \int_0^t \int_{-\infty}^{L} p^{L}_u(x,z) b_n(z) \bigg( \int_{-\infty}^{L} g \big( y \sqrt{\beta/\rho} \big) p^{L}_{t-u}(z,y) \: dy \bigg)^2 \: dz \: du.
\end{multline*}
When $n$ is large, the birth rate $b_n(z)$ is bounded for $z\le L$ by (\ref{A3}), and $g$ is also bounded, so
\begin{align}\label{vargform}
&\E_x \bigg[ \bigg( \sum_{i=1}^{N(t)} g\big( X_i(t) \sqrt{\beta/\rho} \big) \bigg)^2 \bigg] \nonumber \\
&\hspace{.4in} \lesssim \int_{-\infty}^{L} p^{L}_t(x,y) \: dy + \int_0^t \int_{-\infty}^{L} p^{L}_u(x,z) \bigg( \int_{-\infty}^{L} p^{L}_{t-u}(z,y) \: dy \bigg)^2 \: dz \: du.
\end{align}

From (\ref{intpeq}) and the fact that $s \geq 0$, we get
\begin{align*}
\int_{-\infty}^L p^{L}_t(x,y) \: dy \leq \int_{-\infty}^{\infty} p_t(x,y) \: dy &= \exp \bigg( \beta x t + \frac{\beta^2 t^3}{6} - \frac{\beta \rho t^2}{2} \bigg) \\
&= \exp \bigg( \rho x  - \beta x s - \frac{\rho^3}{3 \beta} + \frac{\rho^2 s}{2} - \frac{\beta^2 s^3}{6} \bigg).
\end{align*}
Because $s \geq 0$, and because the assumption $L - x \ll \rho^2/\beta$ implies that $x \geq 0$ for sufficiently large $n$, it follows that
\begin{equation}\label{zvarterm11}
\int_{-\infty}^L p^{L}_t(x,y) \: dy \lesssim \exp \bigg( \rho x - \frac{\rho^3}{3 \beta} + \frac{\rho^2 s}{2} \bigg).
\end{equation}
We claim that
\[\exp \bigg( - \frac{\rho^3}{3 \beta} + \frac{\rho^2 s}{2} \bigg) \ll \bigg( \frac{\rho^3}{\beta} \bigg)^{-2/3} \exp \bigg( - \frac{\rho^3}{6 \beta} - \frac{\gamma_1 \rho}{(2\beta)^{1/3}} - \frac{\beta^2 s^3}{3} \bigg).\]
Indeed, comparing the two sides of the inequality, we see that because $s \ll \rho/\beta$ by assumption, the dominant terms are those of order $\rho^3/\beta$ inside the exponentials; therefore the left-hand side is smaller when $n$ is large. Substituting this into \eqref{zvarterm11}, we obtain
\begin{align*}
\int_{-\infty}^L p^{L}_t(x,y) \: dy &\ll \bigg( \frac{\rho^3}{\beta} \bigg)^{-2/3} \exp \bigg(\rho x - \frac{\rho^3}{6 \beta} - \frac{\gamma_1 \rho}{(2\beta)^{1/3}} - \frac{\beta^2 s^3}{3} \bigg)\\
&= \frac{\beta^{2/3}}{\rho^2} \exp \bigg(\rho x + \rho L - \frac{2 \rho^3}{3 \beta} - \frac{\beta^2 s^3}{3} \bigg)\\
&\ll \frac{\beta^{2/3}}{\rho^4} \exp \bigg(\rho x + \rho L - \frac{2 \rho^3}{3 \beta} - \frac{\beta^2 s^3}{3} \bigg).
\end{align*}

It remains to bound the second term on the right-hand side of (\ref{vargform}).  We will use three different bounds for $p^{L}_u(x,z)$ depending on the value of $u$.  When $\rho^{1/2}/\beta^{5/6} \leq u \leq t$, we will use that $p^{L_A}_u(x,z) \leq p_u(x,z)$ and then use (\ref{ptxy}).  Noting that $$\frac{\beta(z + x)u}{2} = \beta L u - \frac{\beta (L - z) u}{2} - \frac{\beta (L - x) u}{2} \leq \beta L u - \frac{\beta(L - z) u}{2},$$ we get
$$p^L_u(x,z) \lesssim \frac{1}{u^{1/2}} \exp \bigg(  \rho x - \rho z - \frac{(x - z)^2}{2u} - \frac{\rho^2 u}{2} + \beta L u - \frac{\beta(L - z) u}{2} + \frac{\beta^2 u^3}{24} \bigg).$$  We will use Lemma \ref{dapprox1} when $0 \leq u \leq 1/\rho^2$ and Lemma \ref{dapprox2} when $1/\rho^2 < u < \rho^{1/2}/\beta^{5/6}$.  Combining these estimates yields that if we set
\begin{equation}\label{Mdef}
M(u,x,z) = \min\bigg\{\frac{(L - x)(L - z)}{u^{3/2}}, \frac{1}{u^{1/2}}, \frac{1}{u^{1/2}} 
\exp \bigg( - \frac{\beta(L - z) u}{2} + \frac{\beta^2 u^3}{24} \bigg) \bigg\},
\end{equation}
then
\begin{equation}\label{vubound}
p^{L}_u(x,z) \lesssim M(u,x,z) \exp \bigg( \rho x - \rho z - \frac{(x - z)^2}{2u} - \frac{\rho^2 u}{2} + \beta L u \bigg).
\end{equation}

To bound $p_{t-u}^L(z,y)$, we use (\ref{intpeq}) when $\rho^{1/2}/\beta^{5/6} \leq u \leq t$ and Lemma \ref{intpL} when $0 \leq u < \rho^{1/2}/\beta^{5/6}$.  
From (\ref{intpeq}) with $t - u = \frac{\rho}{\beta} - s - u$ in place of $t$, we get
\begin{align*}
\int_{-\infty}^{L} p^{L}_{t-u}(z, y) \: dy &\leq \int_{-\infty}^{\infty} p_{t-u}(z,y) \: dy \\
&= \exp \bigg( \beta z \Big( \frac{\rho}{\beta} - (u + s) \Big) + \frac{\beta^2}{6} \Big( \frac{\rho}{\beta} - (u + s) \Big)^3 - \frac{\beta \rho}{2} \Big( \frac{\rho}{\beta} - (u + s) \Big)^2 \bigg) \\
&= \exp \bigg( \rho z + \Big( \frac{\rho^2}{2} - \beta z \Big)(u + s) - \frac{\rho^3}{3 \beta} - \frac{\beta^2(u + s)^3}{6} \bigg).
\end{align*}
Therefore, using also that $(u + s)^3 \geq u^3 + s^3$, we have
\begin{equation}\label{vtubound}
\bigg( \int_{-\infty}^{L} p^{L}_{t-u}(z, y) \: dy \bigg)^2 \leq \exp \bigg( 2 \rho z + (\rho^2 - 2 \beta z)(u + s) - \frac{2 \rho^3}{3 \beta} - \frac{\beta^2u^3 }{3} - \frac{\beta^2 s^3}{3} \bigg).
\end{equation}
Using Lemma \ref{intpL} and following the same calculation, we get that if $u \ll t$, then
\begin{align}\label{vtubound2}
&\bigg( \int_{-\infty}^{L} p^{L}_{t-u}(z, y) \: dy \bigg)^2 \nonumber \\
&\hspace{.1in}\lesssim \beta^{2/3}(L - z)^2 \exp \bigg( 2 \rho z + (\rho^2 - 2 \beta z)(u + s) - \frac{2 \rho^3}{3 \beta} - \frac{\beta^2u^3 }{3} - \frac{\beta^2 s^3}{3} + 2 \beta^{2/3} (u + s) \bigg).
\end{align}
Combining (\ref{vtubound}) and (\ref{vtubound2}), and using that $\rho^{1/2}/\beta^{5/6} \ll t$ by (\ref{A1}), 
we get that if we set $$N(u,z) = \beta^{2/3}(L - z)^2 \1_{\{u < \rho^{1/2}/\beta^{5/6}\}} + \1_{\{u \geq \rho^{1/2}/\beta^{5/6}\}},$$ then
\begin{align}\label{vtubound3}
&\bigg( \int_{-\infty}^{L} p^{L}_{t-u}(z, y) \: dy \bigg)^2 \nonumber \\
&\hspace{.4in}\lesssim N(u,z) \exp \bigg( 2 \rho z + (\rho^2 - 2 \beta z)(u + s) - \frac{2 \rho^3}{3 \beta} - \frac{\beta^2u^3 }{3} - \frac{\beta^2 s^3}{3} + 2 \beta^{2/3} (u + s) \bigg).
\end{align}

Denoting the second term on the right-hand side of (\ref{vargform}) by $T$ and then combining (\ref{vubound}) and (\ref{vtubound3}), we get
\begin{align*}
T &\lesssim \exp \bigg( \rho x - \frac{2 \rho^3}{3 \beta} - \frac{\beta^2 s^3}{3} \bigg) \int_0^t \int_{-\infty}^L M(u,x,z) N(u,z) \\
&\hspace{.4in}\times \exp \bigg(\rho z - \frac{\rho^2 u}{2} + \beta L u - \frac{\beta^2 u^3}{3} - \frac{(x - z)^2}{2u} + (\rho^2 - 2 \beta z + 2 \beta^{2/3})(u + s) \bigg) \: dz \: du.
\end{align*}
Interchanging the roles of $z$ and $L - z$, and separating out three terms from the exponential that involve only $u$, we get
\begin{align*}
T &\lesssim \exp \bigg( \rho x + \rho L - \frac{2 \rho^3}{3 \beta} - \frac{\beta^2 s^3}{3} \bigg) \int_0^t \exp \bigg( - \frac{\rho^2 u}{2} + \beta L u - \frac{\beta^2 u^3}{3} \bigg) \int_0^{\infty} M(u, x, L - z) 
 \\
&\hspace{.2in}\times N(u, L-z) \exp \bigg(- \rho z - \frac{((L - x) - z)^2}{2u} + (\rho^2 - 2 \beta (L - z) + 2 \beta^{2/3})(u + s) \bigg) \: dz \: du.
\end{align*}
Now, using (\ref{LAdef}) and (\ref{zerovalue}), we have $\rho^2 - 2 \beta L + 2 \beta^{2/3} = (2^{2/3} \gamma_1 + 2)\beta^{2/3} < 0$, 
so we can discard the term $(\rho^2 - 2 \beta L + 2 \beta^{2/3})(u + s)$.  Therefore, 
\begin{align}\label{TReq}
T &\lesssim \exp \bigg( \rho x + \rho L - \frac{2 \rho^3}{3 \beta} - \frac{\beta^2 s^3}{3} \bigg) \int_0^t \exp \bigg( - \frac{\rho^2 u}{2} + \beta L u - \frac{\beta^2 u^3}{3} \bigg) \nonumber \\
&\hspace{.2in}\times \int_0^{\infty} M(u, x, L - z) N(u, L-z) \exp \bigg(- (\rho - 2 \beta (u + s)) z - \frac{((L - x) - z)^2}{2u} \bigg) \: dz \: du.
\end{align}
Thus, denoting the double integral in the expression above by $R$, the proof will be complete if we can show that $R \lesssim \beta^{2/3}/\rho^4$.  
To do this, choose a positive number $r$ such that $\sqrt{3/7} < r < 2/3$, and write $$R = R_1 + R_2 + R_3 + R_4,$$ where
$R_1$ is the portion of the double integral for which $0 \leq u \leq 1/\rho^2$, $R_2$ is the portion of the double integral for which $1/\rho^2 < u < \rho^{1/2}/\beta^{5/6}$, $R_3$ is the portion with $\rho^{1/2}/\beta^{5/6} \leq u < r \rho/\beta$, and $R_4$ is the portion with $r\rho/\beta \leq u \leq t$.

We first estimate $R_1$.  Because $\gamma_1 < 0$ and the function $x \mapsto bx - cx^3$ is bounded above on $[0, \infty)$ for any $b > 0$ and $c > 0$, we have
\begin{equation}\label{expconst}
- \frac{\rho^2 u}{2} + \beta L u - \frac{\beta^2 u^3}{3} = -2^{-1/3} \beta^{2/3} \gamma_1 u - \frac{\beta^2 u^3}{3} = -2^{-1/3} \gamma_1 (\beta^{2/3} u) - \frac{1}{3} (\beta^{2/3} u)^3 = O(1).
\end{equation}
Therefore, using also that $2 \beta(u + s) \ll \rho$ when $u \leq 1/\rho^2$, we have
\begin{align}\label{R1main}
R_1 &\lesssim \int_0^{1/\rho^2} \int_0^{\infty} \frac{1}{u^{1/2}} \cdot \beta^{2/3} z^2 \exp \bigg(- (\rho - 2 \beta (u + s)) z - \frac{((L - x) - z)^2}{2u} \bigg) \: dz \: du \nonumber \\
&\lesssim \beta^{2/3} \int_0^{1/\rho^2} \frac{1}{u^{1/2}} \int_0^{\infty} z^2 \exp \bigg(- (\rho - 2 \beta (u + s)) z \bigg) \: dz \: du \nonumber \\
&\lesssim \beta^{2/3} \int_0^{1/\rho^2} \frac{1}{u^{1/2} \rho^3} \: du \nonumber \\
&\lesssim \frac{\beta^{2/3}}{\rho^4}.
\end{align}

Next, we estimate $R_2$.  Suppose $1/\rho^2 < u < \rho^{1/2}/\beta^{5/6}$.  Since $u \ll \rho/\beta$ and $s \ll \rho/\beta$, we have $\rho - 2 \beta(u + s) \geq \rho/2$ for sufficiently large $n$.  Using also (\ref{expconst}) again, we get
\begin{align}\label{R2int}
R_2 &\leq \int_{1/\rho^2}^{\rho^{1/2}/\beta^{5/6}} \int_0^{\infty} \frac{(L-x)z}{u^{3/2}} \cdot \beta^{2/3} z^2 \exp \bigg( - \frac{\rho z}{2} - \frac{((L-x) - z)^2}{2u} \bigg) \: dz \: du. \nonumber \\
&= \beta^{2/3}(L - x) \int_{1/\rho^2}^{\rho^{1/2}/\beta^{5/6}} \frac{1}{u^{3/2}} \int_0^{\infty} z^3 \exp \bigg( - \frac{\rho z}{2} - \frac{((L-x) - z)^2}{2u} \bigg) \: dz \: du.
\end{align}
To evaluate the inner integral, we will apply Lemma \ref{BMmoments} with $k = 3$, with $L - x$ in place of $a$, with $2u$ in place of $t$, and with $\rho/2$ in place of $\rho$.  The condition $a \geq \rho t/2$ in Lemma \ref{BMmoments} becomes $u \leq 2(L - x)/\rho$.  Noting that $\rho^2 t/4 - a \rho \leq -a \rho/2$ when $a \geq \rho t/2$, we have
\begin{align}\label{R2a}
&\beta^{2/3}(L - x) \int_{1/\rho^2}^{2(L - x)/\rho} \frac{1}{u^{3/2}} \int_0^{\infty} z^3 \exp \bigg( - \frac{\rho z}{2} - \frac{((L-x) - z)^2}{2u} \bigg) \: dz \: du \nonumber \\
&\hspace{.4in}\lesssim \beta^{2/3}(L - x) \int_{1/\rho^2}^{2(L - x)/\rho} \frac{1}{u^{3/2}} \cdot u^{1/2}\big((L - x)^3 + u^{3/2}\big) e^{-\rho(L-x)/4} \: du \nonumber \\
&\hspace{.4in}\lesssim \beta^{2/3} e^{-\rho(L-x)/4} \bigg( (L - x)^4 \int_{1/\rho^2}^{2(L - x)/\rho} \frac{1}{u} \: du + (L - x) \int_{1/\rho^2}^{2(L - x)/\rho} u^{1/2} \: du \bigg) \nonumber \\
&\hspace{.4in}\lesssim \frac{\beta^{2/3}}{\rho^4} e^{-\rho(L-x)/4} \Big( \rho^4 (L - x)^4 \log (2 \rho(L-x)) + \rho^{5/2}(L - x)^{5/2} \Big) \nonumber \\
&\hspace{.4in}\lesssim \frac{\beta^{2/3}}{\rho^4}.
\end{align}
When $u > 2(L-x)/\rho$, we apply instead the second part of Lemma \ref{BMmoments}.  When $2(L-x)/\rho < u \leq 4(L-x)/\rho$, we disregard the first term in the denominator on the right-hand side of (\ref{BMmom2}) and get
\begin{align}\label{R2b}
&\beta^{2/3}(L - x) \int_{2(L-x)/\rho}^{4(L - x)/\rho} \frac{1}{u^{3/2}} \int_0^{\infty} z^3 \exp \bigg( - \frac{\rho z}{2} - \frac{((L-x) - z)^2}{2u} \bigg) \: dz \: du \nonumber \\
&\hspace{.4in}\lesssim \beta^{2/3}(L - x) \int_{2(L-x)/\rho}^{4(L - x)/\rho} \frac{1}{u^{3/2}} \cdot u^2 e^{-(L - x)^2/2u} \: du \nonumber \\
&\hspace{.4in}\lesssim \beta^{2/3} (L - x) e^{-\rho(L-x)/8} \bigg(\frac{L-x}{\rho} \bigg)^{3/2} \nonumber \\
&\hspace{.4in} \lesssim \frac{\beta^{2/3}}{\rho^4} e^{-\rho(L-x)/8} \rho^{5/2}(L - x)^{5/2} \nonumber \\
&\hspace{.4in} \lesssim \frac{\beta^{2/3}}{\rho^4}.
\end{align}
Finally, when $4(L-x)/\rho < u \leq \rho^{1/2} /\beta^{5/6}$, we disregard the second term in the denominator on the right-hand side of (\ref{BMmom2}) and then apply Lemma \ref{32int} to get
\begin{align}\label{R2c}
&\beta^{2/3}(L - x) \int_{4(L-x)/\rho}^{\rho^{1/2}/\beta^{5/6}} \frac{1}{u^{3/2}} \int_0^{\infty} z^3 \exp \bigg( - \frac{\rho z}{2} - \frac{((L-x) - z)^2}{2u} \bigg) \: dz \: du \nonumber \\
&\hspace{.4in}\lesssim \frac{\beta^{2/3} (L - x)}{\rho^4} \int_0^{\infty} \frac{1}{u^{3/2}} e^{-(L - x)^2/2u} \: du \nonumber \\
&\hspace{.4in}\lesssim \frac{\beta^{2/3}}{\rho^4}.
\end{align}
By combining (\ref{R2int}), (\ref{R2a}), (\ref{R2b}), and (\ref{R2c}), we get
\begin{equation}\label{R2main}
R_2 \lesssim \frac{\beta^{2/3}}{\rho^4}.
\end{equation}

Next, we estimate $R_3$.  For $u \geq \rho^{1/2}/\beta^{5/6}$, we have $N(u, L-z) = 1$.  Bounding $M(u, x, L-z)$ by the third expression in (\ref{Mdef}), we get from (\ref{TReq})
\begin{align}\label{R3initial}
R_3 &\lesssim \int_{\rho^{1/2}/\beta^{5/6}}^{r \rho/\beta} \frac{1}{u^{1/2}} \exp \bigg( - \frac{\rho^2 u}{2} + \beta L u - \frac{7 \beta^2 u^3}{24} \bigg) \nonumber \\
&\hspace{.4in}\times \int_0^{\infty} \exp \bigg( -(\rho - 2 \beta(u + s)) z - \frac{\beta u z}{2} - \frac{((L - x) - z)^2}{2u} \bigg) \: dz \: du.
\end{align}
We now discard the $((L - x) - z)^2/2u$ term and apply (\ref{expconst}) with $1/24$ in place of $1/3$ as the constant in front of $\beta^2 u^3$.  Using also that $r < 2/3$ and $s \ll \rho/\beta$, we get
\begin{align*}
R_3 &\lesssim \int_{\rho^{1/2}/\beta^{5/6}}^{r \rho/\beta} \frac{1}{u^{1/2}} \exp \bigg( - \frac{\beta^2 u^3}{4} \bigg) \int_0^{\infty} \exp \bigg( \Big( - \rho + \frac{3 \beta u}{2} + 2 \beta s \Big)z \bigg) \: dz \: du \\
&\lesssim \frac{1}{\rho} \int_{\rho^{1/2}/\beta^{5/6}}^{r \rho/\beta} \frac{1}{u^{1/2}} \exp \bigg( - \frac{\beta^2 u^3}{4} \bigg) \: du.
\end{align*}
Noting that the integrand is decreasing in $u$ and therefore largest when $u = \rho^{1/2}/\beta^{5/6}$, we get
\begin{equation}\label{R3main}
R_3 \lesssim \frac{1}{\rho} \cdot \frac{r \rho}{\beta} \cdot \frac{\beta^{5/12}}{\rho^{1/4}} \exp \bigg(- \frac{\rho^{3/2}}{4 \beta^{1/2}} \bigg) = \frac{r\beta^{2/3}}{\rho^4} \bigg( \frac{\rho^3}{\beta} \bigg)^{5/4} \exp \bigg(- \frac{1}{4} \bigg(\frac{\rho^3}{\beta} \bigg)^{1/2} \bigg) \ll \frac{\beta^{2/3}}{\rho^4}.
\end{equation}

To evaluate $R_4$, we reason as in (\ref{R3initial}) and use that, by the well-known formula for the moment generating function of the normal distribution,
\begin{align*}
&\int_{-\infty}^{\infty} \exp \bigg( -(\rho - 2 \beta(u + s)) z - \frac{\beta u z}{2} - \frac{((L - x) - z)^2}{2u} \bigg) \: dz \\
&\hspace{.5in} = \sqrt{2 \pi u} \: \exp \bigg( (L - x) \bigg( -\rho + \frac{3 \beta u}{2} + 2 \beta s \bigg) +  \frac{u}{2} \bigg( -\rho + \frac{3 \beta u}{2} + 2 \beta s \bigg)^2 \bigg),
\end{align*}
to get
\begin{align*}
R_4 &\lesssim \int_{r \rho/\beta}^t \exp \bigg( - \frac{\rho^2 u}{2} + \beta L u - \frac{7 \beta^2 u^3}{24} \bigg) \nonumber \\
&\hspace{.4in}\times \exp \bigg( (L - x) \bigg( -\rho + \frac{3 \beta u}{2} + 2 \beta s \bigg) +  \frac{u}{2} \bigg( -\rho + \frac{3 \beta u}{2} + 2 \beta s \bigg)^2 \bigg) \: du.
\end{align*}
Because $u \leq \rho/\beta$, we have $-\rho + 3 \beta u/2 + 2 \beta s \leq \rho/2 + 2 \beta s$.  Considering also that $L - x \ll \rho^2/\beta$ and $s \ll \rho/\beta$, it follows that
$$\exp \bigg( (L - x) \bigg( -\rho + \frac{3 \beta u}{2} + 2 \beta s \bigg) + \frac{u}{2}\bigg( -\rho + \frac{3 \beta u}{2} + 2 \beta s \bigg)^2  \bigg) \leq \exp \bigg( \frac{\rho^2 u}{8} + o \bigg( \frac{\rho^3}{\beta} \bigg) \bigg).$$ Also, we have $-\rho^2 u/2 + \beta L u = -2^{-1/3} \gamma_1 \beta^{2/3} u \ll \rho^3/\beta$.
Therefore,
\begin{align*}
R_4 &\lesssim \int_{r \rho/\beta}^t \exp \bigg( - \frac{7 \beta^2 u^3}{24} + \frac{\rho^2 u}{8} + o \bigg( \frac{\rho^3}{\beta} \bigg) \bigg) \: du \lesssim \frac{\rho}{\beta} \max_{r \rho/\beta \leq u \leq t} \exp \bigg( - \frac{7 \beta^2 u^3}{24} + \frac{\rho^2 u}{8} + o \bigg( \frac{\rho^3}{\beta} \bigg) \bigg).
\end{align*}
Because $r>\sqrt{1/7}$, the function $f(x) = -\frac{7}{24} x^3 + \frac{x}{8}$ is decreasing over $[r, 1]$, and therefore the contribution from the sum of the first two terms in the exponential above is maximized when $u = r \rho/\beta$. This leads to
$$R_4 \lesssim \frac{\beta^{2/3}}{\rho^4} \bigg(\frac{\rho^3}{\beta} \bigg)^{5/3} \exp \bigg( \bigg( -\frac{7 r^3 }{24} + \frac{r}{8} \bigg) \frac{\rho^3}{\beta} + o \bigg( \frac{\rho^3}{\beta} \bigg) \bigg).$$
Because $r > \sqrt{3/7}$, the coefficient in front of $\rho^3/\beta$ in the exponential is negative.  It follows that $R_4 \ll \beta^{2/3}/\rho^4$, which in combination with (\ref{R1main}), (\ref{R2main}), and (\ref{R3main}) completes the proof.

\bigskip
\bigskip
\noindent {\bf {\Large Acknowledgments}}

\bigskip
\noindent The authors thank Jiaqi Liu for spotting a typo in an earlier version of the manuscript, and Daniel Fisher for bringing to their attention the references \cite{f13, gd13, grbhd12}.  They also thank two referees for helpful comments that improved the exposition of the paper.

\end{document}